\documentclass{amsart}
\usepackage{amssymb}
\usepackage{amsfonts}
\usepackage{amsmath}
\usepackage{mathrsfs}

\newtheorem{theorem}{Theorem}[section]
\newtheorem{lemma}[theorem]{Lemma}
\theoremstyle{definition}
\newtheorem{definition}[theorem]{Definition}

\newtheorem{proposition}[theorem]{Proposition}
\newtheorem{corollary}[theorem]{Corollary}
\theoremstyle{remark}
\newtheorem{remark}[theorem]{Remark}
\numberwithin{equation}{section}

\begin{document}
\title[Sturm-Liouville Problems with Distributional Potentials]{Eigenvalues
of Sturm-Liouville Operators with Distributional Potentials}
\author{Jun Yan}
\address{Department of Mathematics, Tianjin University, Tianjin, 300072, P.
R. China}
\email{junyantju@126.com}
\author{Guoliang Shi}
\address{Department of Mathematics, Tianjin University, Tianjin, 300072, P.
R. China}
\email{glshi@tju.edu.cn}
\author{Jia Zhao}
\address{Department of Mathematics, Hebei University of Technology, Tianjin, 300401, P.
R. China}
\email{zhaojia@tju.edu.cn}
\thanks{*This research was supported by the National Youth Scientific
Foundation of China under Grant No. 11601372.}

\begin{abstract}
We introduce a novel approach for dealing with eigenvalue problems of Sturm-Liouville operators
generated by the differential expression
\begin{equation*}
Ly=\frac{1}{r}\left( -(p\left[ y^{\prime }+sy\right] )^{\prime }+sp\left[
y^{\prime }+sy\right] +qy\right)
\end{equation*}%
which is based on norm resolvent convergence of classical Sturm-Liouville operators. This
enables us to describe the continuous dependence of the $n$-th eigenvalue on the space of
self-adjoint boundary conditions and the coefficients of the differential equation
after giving the inequalities among the eigenvalues. Moreover, oscillation properties of
the eigenfunctions are also characterized. In particular, our main results can be applied to
solve a class of Sturm-Liouville problems with transmission conditions.
\end{abstract}

\maketitle







\section*{Introduction}

The prime motivation behind this paper is to discuss the properties of
eigenvalues of self-adjoint Sturm-Liouville operators generated by the
differential expression%
\begin{equation}
Ly=\frac{1}{r}\left( -(p\left[ y^{\prime }+sy\right] )^{\prime }+sp\left[
y^{\prime }+sy\right] +qy\right) ,\text{ on }J=(a,b),\text{ }-\infty
<a<b<\infty ,  \label{a}
\end{equation}%
where the coefficients $p,$ $q,$ $r,$ $s$ are real-valued and
\begin{equation}
1/p,\text{ }q,\text{ }r,\text{ }s\in L(J,%
\mathbb{R}
),\text{ }p>0,\text{ }r>0\text{ a.e. on }J.  \label{tx}
\end{equation}

Note that when $p(x)\equiv 1$, the definition and the self-adjoint domain of
(\ref{a}) have been characterized by A. M. Savchuk and A. A. Shkalikov in
\cite{CM1} and \cite{CC3}. Moreover, in the special case $s\equiv 0$ this
differential expression reduces to the standard one, that is, one obtains,%
\begin{equation}
Ly=\frac{1}{r}\left( -(py^{\prime })^{\prime }+qy\right) .  \label{classi}
\end{equation}%
In the paper \cite{xx14}, J. Eckhardt, F. Gesztesy, R. Nichols and G. Teschl
have given a description of all the self-adjoint operators generated by the
expression (\ref{a}). Following \cite{xx14}, we introduce the
quasi-derivative $y^{[1]}=p\left[ y^{\prime }+sy\right] ,$ the self-adjoint
boundary conditions are given as follows:%
\begin{equation}
A\left(
\begin{array}{c}
y(a) \\
y^{[1]}(a)%
\end{array}%
\right) +B\left(
\begin{array}{c}
y(b) \\
y^{[1]}(b)%
\end{array}%
\right) =\left(
\begin{array}{c}
0 \\
0%
\end{array}%
\right) ,  \label{25}
\end{equation}%
where the complex $2\times 2$ matrices $A$ and $B$ satisfy$:$
\begin{equation}
\text{the }2\times 4\text{ matrix }(A|B)\text{ has full rank, and }AEA^{\ast
}=BEB^{\ast },\text{ }E=\left(
\begin{array}{cc}
0 & -1 \\
1 & 0%
\end{array}%
\right) .  \label{juzhen}
\end{equation}%
Note that $A^{\ast }$ is the complex conjugate transpose of the complex
matrix $A.$ The boundary conditions (\ref{25}) can be divided into three
classes of boundary conditions as follows:

1.Separated self-adjoint boundary conditions:%
\begin{equation}
S_{\alpha ,\beta }:\left\{
\begin{array}{c}
\cos \alpha y(a)-\sin \alpha y^{[1]}(a)=0,\text{ }\alpha \in \lbrack 0,\pi ),
\\
\cos \beta y(b)-\sin \beta y^{[1]}(b)=0,\text{ }\beta \in (0,\pi ].%
\end{array}%
\right.  \label{11}
\end{equation}%
~ $\ \ $2.All real coupled self-adjoint boundary conditions:
\begin{equation}
Y(b)=KY(a),\text{ }K\in \mathit{\mathrm{SL}}(2,%
\mathbb{R}
).  \label{21}
\end{equation}

3.All complex coupled self-adjoint boundary conditions: \
\begin{equation}
Y(b)=e^{i\gamma }KY(a),\text{ }-\pi <\gamma <0\text{ or }0<\gamma <\pi ,%
\text{ }K\in \mathit{\mathrm{SL}}(2,%
\mathbb{R}
),  \label{ee}
\end{equation}%
where
\begin{equation*}
K\in \mathit{\mathrm{SL}}(2,%
\mathbb{R}
)=:\left\{ \left(
\begin{array}{cc}
k_{11} & k_{12} \\
k_{21} & k_{22}%
\end{array}%
\right) ;\text{ }k_{ij}\in
\mathbb{R}
,\text{ }\det K=1\right\} ,Y(\cdot )=\left(
\begin{array}{c}
y(\cdot ) \\
y^{[1]}(\cdot )%
\end{array}%
\right) .
\end{equation*}
Actually, (\ref{21}) can be treated as a case of (\ref{ee}) when $\gamma
=0.\ $

In the last decades, Schr\"{o}dinger operators with distributional
potentials have attracted tremendous interest since they can be used as
solvable models in many situations. We should mention that there were
actually earlier papers dealing with Schr\"{o}dinger operators involving
strongly singular and oscillating potentials, such as, M.-L. Baeteman and K.
Chadan \cite{15}, \cite{16}, M. Combescure \cite{28}, M. Combescure and J.
Ginibre \cite{27}, D. B. Pearson \cite{131}, F. S. Rofe-Beketov and E. H.
Hristov \cite{134}, \cite{135}, and a more recent contribution treating
distributional potentials by J. Herczy\'{n}ski \cite{72}. In addition,
numerous results on the case of point interactions can be found in some
standard monographs by S. Albeverio, F. Gesztesy, R. H\o gh-Krohn, and H.
Holden \cite{2} and S. Albeverio and P. Kurasov \cite{5}. It was not until
1999 that A. M. Savchuk and A. A. Shkalikov \cite{CM1} started a new
development for Schr\"{o}dinger operators with distributional potential
coefficients. And the operators with distribution potentials proposed by
A.M. Savchuk and A.A. Shkalikov have been received enormous attention. We
also emphasize that similar differential expressions have already been
studied by C. Bennewitz and W. N. Everitt \cite{C. Bennewitz} (see also \cite%
{Everitt}).

In the paper \cite{xx14} and \cite{inverse}, J. Eckhardt, F. Gesztesy, R.
Nichols and G. Teschl have given a systematical development of
Weyl--Titchmarsh theory and inverse spectral theory for singular
differential operators on arbitrary intervals $(a,b)\subset
\mathbb{R}
$ associated with the differential expressions (\ref{a}). Under the
assumption (\ref{tx}) on the coefficients, the discreteness and boundedness
from below of the spectrum has been proved in \cite{xx14} for the
self-adjoint differential operators associated with the differential
expression (\ref{a}). In this paper, we will continue to discuss the
properties of eigenvalues of the self-adjoint differential operators with
distributional potentials under the assumption (\ref{tx}).

Actually, in this paper, by a different method, we will also show that the
spectrum of the self-adjoint differential operators associated with the
differential expression (\ref{a}) is discrete, and the eigenvalues can also
be ordered to form a non-decreasing sequence%
\begin{equation*}
\lambda _{0},\lambda _{1},\lambda _{2},\lambda _{3},\ldots
\end{equation*}%
approaching $+\infty $ so that the number of times an eigenvalue appears in
the sequence is equal to its multiplicity. Here multiplicity refers to both
the algebraic and geometric multiplicities, since in this paper we will show
that the two multiplicities are equal for self-adjoint boundary conditions.\
Note that for the Sturm-Liouville problems with distributional potentials,
the algebraic multiplicity of an eigenvalue we introduce is the order of its
zero as a root of the characteristic function discussed in Lemma \ref{L1}.
The geometric multiplicity of an eigenvalue is naturally the number of the
linearly independent eigenfunctions of this eigenvalue.

As we have known, for Sturm-Liouville operators with regular potentials,
i.e., the operators generated by the differential expression (\ref{classi}),
M.S.P. Eastham, Q. Kong, H. Wu, and A. Zettl (\cite{CC5}, \cite{XIN1} and
\cite{kong}) have characterized the inequalities among the eigenvalues
corresponding to different self-adjoint boundary conditions, the continuity
region of the $n$-th eigenvalue as a function on the space of self-adjoint
boundary conditions, the dependence of the $n$-th eigenvalue on the
coefficients of the differential equation, and the oscillation properties of
the eigenfunctions. In contrast, such theory for Sturm-Liouville operators
with distributional potentials has not yet been developed, and it is
precisely the purpose of this paper to have a discussion on the
corresponding properties.

Enlightened by the space introduced by Q. Kong and A. Zettl in \cite{kong},
in this paper we will introduce a more general \textquotedblleft boundary
value problem space\textquotedblright\ with a metric to study the
Sturm-Liouville problems with distributional potentials. Let $\Omega
=\{\omega =(A,B,1/p,q,r,s);(\ref{tx})$ and $(\ref{juzhen})$ hold$\}.$ For
the topology of $\Omega $ we use a metric $d$ defined as follows: For $%
\omega =(A,B,1/p,q,r,s)\in \Omega ,$ $\omega _{0}=(A_{0},B_{0},\frac{1}{p_{0}%
},q_{0},r_{0},s_{0})\in \Omega ,$ define $d(\omega ,\omega _{0})=\left\Vert
A-A_{0}\right\Vert +\left\Vert B-B_{0}\right\Vert +\int_{a}^{b}\left(
\left\vert \frac{1}{p}-\frac{1}{p_{0}}\right\vert +\left\vert
q-q_{0}\right\vert +\left\vert r-r_{0}\right\vert +\left\vert
s-s_{0}\right\vert \right) $ where $\left\Vert \cdot \right\Vert $ denotes
any matrix norm. Denote the space of all complex self-adjoint boundary
conditions by $\mathscr{B}_{S}^{%
\mathbb{C}
}$ which is the similar to the space associated with the Sturm-Liouville
problems with regular potentials introduced firstly in \cite{8}. Under such
a topology, we will have a research on the continuity region of the $n$-th
eigenvalue as a function on $\mathscr{B}_{S}^{%
\mathbb{C}
}$, the differentiability and monotonicity of the $n$-th eigenvalue with
respect to $\alpha ,\beta $ in the separated boundary conditions are also
given. We will prove the continuous dependence and differentiability of the $%
n$-th eigenvalue with respect to the coefficients $1/p,$ $q,$ $r,$ $s$ in
the sense of Frechet derivative in the Banach space $L(J,%
\mathbb{R}
).$

It is worth mentioning that, in order to analyze
the eigenvalues of Sturm-Liouville operators with distributional potentials,
we introduce an approach that relies heavily on the \textquotedblleft norm
resolvent convergence\textquotedblright\ and an asymptotic form of the
fundamental solutions of the equation $(\ref{fangcheng})$ for sufficiently
negative $\lambda $, which is different from that for Sturm-Liouville operators with regular
potentials. In this paper, we will find a sequence of Sturm-Liouville operators $L_{m}$
with regular potentials to approximate the Sturm-Liouville operator $L$ with a distributional potential
in norm resolvent convergence (Lemma \ref{L13}).
Furthermore, in Lemma \ref{L0 copy(1)}, we will show that the lowest eigenvalues of $L_{m}$ are uniformly
semi-bounded from below, this will guarantee the sequence of the $n$-th eigenvalues of Sturm-Liouville operators $L_{m}$
converges to the $n$-th eigenvalue of the operator $L$(Lemma \ref{norm resovent}).
In a word, our approach not only enables us to obtain  a series of results, but also
yields a relation between the eigenvalues of Sturm-Liouville operators with distributional potentials
and the eigenvalues of Sturm-Liouville operators with regular potentials.
Moreover, the main conclusions obtained in this paper can be applied
to solve the eigenvalue problems of Sturm-Liouville operators with transmission
conditions which have been an important research topic in
mathematical physics \cite{tr1,tr2,tr3,tr4}

This paper is organized as follows. In Section 1, we recall some basic
results, and prove a condition for norm resolvent convergence. In Section 2,
some preliminary and important lemmas for the main results are stated and
proved. In Section 3, we give a comment on the continuity region, the
differentiability and monotonicity of the $n$-th eigenvalue with respect to $%
\alpha ,\beta $ in the separated boundary conditions. Oscillation properties
of the eigenfunctions of all the self-adjoint Sturm-Liouville problems are
given in Section 5 after discussions on the inequalities among eigenvalues
in Section 4. Section 6 is devoted to describe the continuity region of the $%
n$-th eigenvalue as a function on the space of self-adjoint boundary
conditions. In Section 7, we also comment on the continuous
dependence and differentiability of the $n$-th eigenvalue on the
coefficients of the differential equation. Finally, in Section 8, we solve some eigenvalue problems
of a class of Sturm-Liouville operators with transmission conditions.

\section{Notation and prerequisites results}

We introduce the quasi-derivative $y^{[1]}=p\left[ y^{\prime }+sy\right] $
and rewrite expression (\ref{a}) in the form%
\begin{equation}
Ly=\frac{1}{r}\left( -(y^{[1]})^{\prime }+sy^{[1]}+qy\right) ,\text{ }x\in
(a,b).  \label{b}
\end{equation}

Let $\phi _{1}$ and $\phi _{2}$ be the fundamental solutions of
\begin{equation}
-(y^{[1]})^{\prime }+sy^{[1]}+qy=\lambda ry,\text{ }x\in (a,b),
\label{fangcheng}
\end{equation}%
determined by the initial conditions%
\begin{equation}
\phi _{1}(a,\lambda )=\phi _{2}^{[1]}(a,\lambda )=1,\text{ }\phi
_{2}(a,\lambda )=\phi _{1}^{[1]}(a,\lambda )=0,\text{ }\lambda \in
\mathbb{C}
.  \label{c}
\end{equation}

Denote
\begin{equation}
\Phi \left( x,\lambda \right) =\left(
\begin{array}{cc}
\phi _{1}(x,\lambda ) & \phi _{2}(x,\lambda ) \\
\phi _{1}^{[1]}(x,\lambda ) & \phi _{2}^{[1]}(x,\lambda )%
\end{array}%
\right) ,  \label{3}
\end{equation}%
then $\Phi \left( x,\lambda \right) $ is the fundamental matrix solution of%
\begin{equation}
Y^{\prime }(x)=[P(x)-\lambda W(x)]Y(x),\text{ }Y(a)=I,\text{ }x\in (a,b),
\label{4}
\end{equation}%
where
\begin{equation*}
\text{ }P(x)=\left(
\begin{array}{cc}
-s(x) & \frac{1}{p(x)} \\
q(x) & s(x)%
\end{array}%
\right) \text{, \ }W(x)=\left(
\begin{array}{cc}
0 & 0 \\
r(x) & 0%
\end{array}%
\right) .
\end{equation*}%
For $K\in \mathit{\mathrm{SL}}(2,%
\mathbb{R}
),\ $and\ $\lambda \in
\mathbb{C}
$, we define%
\begin{eqnarray*}
D(\lambda ) &=&k_{11}\phi _{2}^{[1]}(b,\lambda )-k_{21}\phi _{2}(b,\lambda
)+k_{22}\phi _{1}(b,\lambda )-k_{12}\phi _{1}^{[1]}(b,\lambda ), \\
A(\lambda ) &=&k_{11}\phi _{1}^{[1]}(b,\lambda )-k_{21}\phi _{1}(b,\lambda ),
\\
B(\lambda ) &=&k_{11}\phi _{2}^{[1]}(b,\lambda )+k_{12}\phi
_{1}^{[1]}(b,\lambda )-k_{21}\phi _{2}(b,\lambda )-k_{22}\phi _{1}(b,\lambda
), \\
D_{1}(\lambda ) &=&k_{11}\phi _{2}^{[1]}(b,\lambda )-k_{21}\phi
_{2}(b,\lambda ), \\
D_{2}(\lambda ) &=&k_{22}\phi _{1}(b,\lambda )-k_{12}\phi
_{1}^{[1]}(b,\lambda ), \\
C(\lambda ) &=&k_{22}\phi _{2}(b,\lambda )-k_{12}\phi _{2}^{[1]}(b,\lambda ).
\end{eqnarray*}%
Note that%
\begin{equation*}
K^{-1}\Phi \left( b,\lambda \right) =\left(
\begin{array}{cc}
D_{2}(\lambda ) & C(\lambda ) \\
A(\lambda ) & D_{1}(\lambda )%
\end{array}%
\right) ,
\end{equation*}%
\begin{equation}
D(\lambda )=D_{1}(\lambda )+D_{2}(\lambda ),\ B(\lambda )=D_{1}(\lambda
)-D_{2}(\lambda ).  \label{GP}
\end{equation}%
For $K\in \mathit{\mathrm{SL}}(2,%
\mathbb{R}
),\ $consider the separated boundary conditions:%
\begin{eqnarray}
y(a) &=&0,k_{22}y(b)-k_{12}y^{[1]}(b)=0,  \label{ff} \\
y^{[1]}(a) &=&0,k_{21}y(b)-k_{11}y^{[1]}(b)=0,  \label{gg}
\end{eqnarray}%
and denote the $n$-th eigenvalue for (\ref{ff}) and (\ref{gg}) by $\mu
_{n}=\mu _{n}(K)$ and $\nu _{n}=\nu _{n}(K)$ respectively, $n\in
\mathbb{N}
_{0}=\{0,1,2,3,\ldots \}$. Denote the $n$-th eigenvalue for (\ref{21}), (\ref%
{ee}) by $\lambda _{n}(K)$, $\lambda _{n}(\gamma ,K)$ respectively, $n\in
\mathbb{N}
_{0}$.

Note that if $ps$ is a smooth function$,$ then the operators generated by
the differential expression (\ref{a}) with the boundary conditions (\ref{25}%
) are Sturm-Liouville operators with regular potentials.

We also introduce the conditions: $($i$)$ Dirichlet boundary conditions $%
S_{0,\pi }$: $y(a)=y(b)=0,$ $($ii$)$ periodic boundary conditions $($when $%
K=I):y(a)=y(b),y^{[1]}(a)=y^{[1]}(b),$ $($iii$)$ semi-periodic boundary
conditions $($when $K=-I)$: $y(a)=-y(b)$, $y^{[1]}(a)=-y^{[1]}(b)$.

Now we recall some notations introduced by T. Kato in \cite{xx13}. Consider
closed linear manifolds $M$ and $N$ of a Banach space $X$. We denote by $%
S_{M}$ the unit sphere of $M$ (the set of all $u\in M$ with $\left\Vert
u\right\Vert $ $=1$). For any two closed linear manifolds $M,$ $N$ of $X$,
we set%
\begin{equation*}
\delta (M,N)=\sup\limits_{u\in S_{M}}\mathrm{dist}(u,N),\text{ }\hat{\delta}%
(M,N)=\max [\delta (M,N),\delta (N,M)].
\end{equation*}%
Consider the set $\mathscr{C}(X,Y)$ of all closed operators from a Banach
space $X$ to a Banach space $Y$. If $T,S$ $\in $ $\mathscr{C}(X,Y)$, their
graphs $G(T),$ $G(S)$ are closed linear manifolds of the product space $%
X\times Y$. We set%
\begin{equation*}
\delta (T,S)=\delta (G(T),G(S)),\hat{\delta}(T,S)=\hat{\delta}\left(
G(T),G(S)\right) =\max [\delta (T,S),\delta (S,T)].
\end{equation*}%
$\hat{\delta}(T,S)$ will be called the gap between $T$ and $S.$ We recall
that $T_{n}$\textbf{\ converges to }$T$\textbf{\ in the generalized sense if
}$\hat{\delta}(T_{n},T)$\textbf{\ }$\rightarrow $\textbf{\ }$0$.

Now we recall the definition of the norm resolvent convergence.~Let $T$ and $%
\{T_{m}\}_{m\in
\mathbb{N}
}$ be closed operators in Hilbert space $H$. We say that the sequence of the
operators $T_{m}$ converges to $T$ in the sense of norm resolvent
convergence, i.e. $T_{m}\overset{R}{\Longrightarrow }T$, if there is a
number $\mu \in
\mathbb{C}
$ belonging to the resolvent sets $\rho (T)$ and $\rho (T_{m})$ for all
sufficiently large $m$ and the sequence of bounded operators $(T_{m}-\mu
)^{-1}$ converges uniformly to the operator $(T-\mu )^{-1}.$

\begin{lemma}
Let $T,T_{m}$ $\in \mathscr{C}(X,Y)$, $n\in
\mathbb{N}
.$ We denote by $\mathscr{B}(X,Y)$ the set of all bounded operators on $X$
to $Y$.

$(1)$ If $T^{-1}$ exists and belongs to $\mathscr{B}(Y,X)$, $T_{m}$ $%
\rightarrow $ $T$ in the generalized sense if and only if $T_{m}^{-1}$
exists and belongs to $\mathscr{B}(Y,X)$, for sufficiently large $m$ and $%
\left\Vert T_{m}^{-1}-T^{-1}\right\Vert \rightarrow 0$.

$(2)$ $T_{m}$ $\rightarrow $ $T$ in the generalized sense and if $A\in %
\mathscr{B}(X,Y)$, then $T_{m}$ $+A\rightarrow T+A$ in the generalized sense.
\end{lemma}

\begin{proof}
See \cite[Theorem 2.23]{xx13}.
\end{proof}

\begin{remark}
If $T_{m}\overset{R}{\Longrightarrow }T,$ from the above claim $(1)$, it can
be seen that $T_{m}-\mu \rightarrow T-\mu $ in the generalized sense, $\mu $
$\in
\mathbb{C}
$ is a number belonging to the resolvent sets $\rho (T)$ and $\rho (T_{m})$
for all sufficiently large $m$. From the above claim $(2)$, it follows that $%
T_{m}$ $\rightarrow $ $T$ in the generalized sense.
\end{remark}

\begin{lemma}
\label{geo copy(1)}Let $T$ $\in $ $\mathscr{C}(X)$ and let $\Gamma $ be a
compact subset of the resolvent set $\rho (T)$. Then there is a $\delta >0$
such that $\Gamma \subset \rho (S)$ for any $S\in $ $\mathscr{C}(X)$ with $%
\hat{\delta}(S,T)<\delta $.
\end{lemma}

\begin{proof}
See \cite[Theorem 3.1]{xx13}.
\end{proof}

\begin{corollary}
\label{geo copy(2)}For self-adjoint operators $T_{m}$ and $T,$ $T_{m}\overset%
{R}{\Longrightarrow }T.$

$\left( 1\right) $ If $\lambda \left( m\right) \in \sigma \left(
T_{m}\right) ,$ and $\lambda \left( m\right) \rightarrow c,$ as $%
m\rightarrow \infty ,$ then $c\in \sigma \left( T\right) .$

$\left( 2\right) $ If $\lambda \in \sigma \left( T\right) ,$ then there must
exist $\lambda \left( m\right) \in \sigma \left( T_{m}\right) ,$ such that $%
\lambda \left( m\right) \rightarrow \lambda ,$ as $m\rightarrow \infty .$
\end{corollary}

\begin{proof}
$\left( 1\right) $ If $c\in \rho \left( T\right) \cap
\mathbb{R}
,$ since $\rho (T)$ is open, so there exists a $\gamma >0$ such that $\left[
c-\gamma ,c+\gamma \right] \subset \rho (T).$ So from Lemma $\ref{geo
copy(1)},$ there is a $\delta >0$ such that $\left[ c-\gamma ,c+\gamma %
\right] \subset \rho (S)$ for any $S\in $ $\mathscr{C}(X)$ with $\hat{\delta}%
(S,T)<\delta $. If $T_{m}\overset{R}{\Longrightarrow }T,$ then $T_{m}$ $%
\rightarrow $ $T$ in the generalized sense. So there exists $N>0$ such that
if $m>N,$ $\hat{\delta}(T_{m},T)<\delta .$ Then $\left[ c-\gamma ,c+\gamma %
\right] \subset \rho (T_{m})$ if $m>N.$ This contradicts to the fact that $%
\lambda \left( m\right) \rightarrow c,$ as $m\rightarrow \infty .$

$\left( 2\right) $ See \cite[Theorem VIII.24]{xx12}.
\end{proof}

\begin{lemma}
\label{geo}For self-adjoint operators $T_{m}$ and $T,$ $T_{m}\overset{R}{%
\Longrightarrow }T,$ if $\lambda _{0}$ is an isolated eigenvalue of the
operator $T$ with finite geometric multiplicity $\chi $, then there are
finitely many eigenvalues of the operators $T_{m}$ in an arbitrary
sufficiently small $\delta $-neighborhood of the point $\lambda _{0}$ if $m$
is large enough. Moreover, their total geometric multiplicity equals $\chi $.
\end{lemma}

\begin{proof}
See \cite[4.3.4 and 4.3.5]{xx13} and \cite[Lemma 5]{CM1}. Note that for
self-adjoint operators, from \cite[Proposition 6.3]{intro}, the geometric
multiplicity of an eigenvalue is equal to the multiplicity described in \cite%
{xx13} and \cite{CM1}.
\end{proof}

\begin{lemma}
\label{L-1}For self-adjoint operators $T_{m}$ and $T,$ $T_{m}\overset{R}{%
\Longrightarrow }T,$ as $m\rightarrow \infty ,m\in
\mathbb{N}
,$ the spectrum of $T_{m}$ are discrete and uniformly semi-bounded from
below, denote the n-th eigenvalue of $T_{m}$ by $\lambda _{n}(m)$, $n\in
\mathbb{N}
_{0}$. Then we obtain the following conclusions:

$(1)$ The spectrum of $T$ is discrete and semi-bounded from below.

$(2)$ The sequence of the n-th eigenvalues $\lambda _{n}(m)$ of the
operators $T_{m}$ converges to the n-th eigenvalue $\lambda _{n}(0)$ of the
operator $T$, i.e., $\lambda _{n}(m)\rightarrow \lambda _{n}(0),$ as $%
m\rightarrow \infty .$ $($Note that the eigenvalues are ordered with
geometric multiplicities.$)$
\end{lemma}

\begin{proof}
$(1)$ For self-adjoint operators $T_{m},$ the spectrum of $T_{m}$ is
discrete if and only if $T_{m}$ has compact resolvent, together with the
fact $T_{m}\overset{R}{\Longrightarrow }T,$ so $T$ has compact resolvent and
the spectrum of $T$ is discrete. Denote $r$ the uniform bound of the
spectrum of $T_{m}.$ Assume the spectrum of $T$ is not semi-bounded from
below, there must exist an eigenvalue $\lambda $ of $T$ such that $\lambda
<r-1.$ Since $T_{m}\overset{R}{\Longrightarrow }T,$ for the eigenvalue $%
\lambda $ of $T$, there must exist $\lambda \left( m\right) \in \sigma
\left( T_{m}\right) ,$ such that%
\begin{equation*}
\lambda \left( m\right) \rightarrow \lambda <r-1\text{, as }m\rightarrow
\infty .
\end{equation*}
Now we reach a contradiction to obtain our claim.

$(2)$ Next, we will show that for $n\in
\mathbb{N}
_{0}$, $\lambda _{n}(m)\rightarrow \lambda _{n}(0),$ as $m\rightarrow \infty
.$For simplicity, we assume $\lambda _{0}(0)$ is geometrically simple and $%
\lambda _{1}(0)$ is geometrically double, and only prove the claim for $%
\lambda _{0}(0)$ and $\lambda _{1}(0),$ the proofs for other cases follow
from a similar process.

$($i$)$ For the simple eigenvalue $\lambda _{0}(0)$ of $T$, there exist $%
\Lambda _{0}\left( m\right) \in \sigma \left( T_{m}\right) $ such that
\begin{equation*}
\Lambda _{0}\left( m\right) \rightarrow \lambda _{0}(0),\text{ as }%
m\rightarrow \infty .
\end{equation*}%
It suffices to show that $\Lambda _{0}\left( m\right) =\lambda _{0}\left(
m\right) $ for sufficiently large $m.$ Assume the contrary$,$ there exists a
subsequence $\left\{ \lambda _{0}(m_{j})\right\} _{j=1}^{\infty }$ such that%
\begin{equation*}
\Lambda _{0}\left( m_{j}\right) >\lambda _{0}(m_{j}).
\end{equation*}%
Since the spectrum of $T_{m}$ are uniformly semi-bounded from below, without
loss of generality, assume $\lambda _{0}(m_{j})\rightarrow c,$ then $\lambda
_{0}(0)$ $\geq c\in \sigma \left( T\right) .$ Since $\lambda _{0}(0)$ is
geometrically simple, from Lemma $\ref{geo},$ one deduces that $c<\lambda
_{0}(0).$ This contrary implies
\begin{equation*}
\lambda _{0}(m)\rightarrow \lambda _{0}(0)\text{, as }m\rightarrow \infty .
\end{equation*}%
$($ii$)$ For the double eigenvalue $\lambda _{1}(0)=\lambda _{2}(0)$, there
exist eigenvalues $\Lambda _{1}\left( m\right) $ and $\Lambda _{2}\left(
m\right) \ $of $T_{m}$ such that%
\begin{equation*}
\Lambda _{1}\left( m\right) \rightarrow \lambda _{1}(0),\text{ }\Lambda
_{2}\left( m\right) \rightarrow \lambda _{2}(0)\text{, as }m\rightarrow
\infty ,
\end{equation*}%
and%
\begin{equation*}
\Lambda _{1}\left( m\right) \leq \Lambda _{2}\left( m\right) .
\end{equation*}%
Then it suffices to show that for sufficiently large $m$,
\begin{equation*}
\Lambda _{1}\left( m\right) =\lambda _{1}(m)\text{, }\Lambda _{2}\left(
m\right) =\lambda _{2}(m).
\end{equation*}%
Assume there exists a subsequence $\left\{ \lambda _{1}(m_{j})\right\}
_{j=1}^{\infty }$ such that $\Lambda _{1}\left( m_{j}\right) >\lambda
_{1}(m_{j})$. Since $\left\{ \lambda _{1}(m_{j})\right\} _{j=1}^{\infty }$
is bounded, without loss of generality, assume
\begin{equation*}
\lambda _{1}(m_{j})\rightarrow c\text{, as }j\rightarrow \infty .
\end{equation*}%
Hence it follows from Corollary $\ref{geo copy(2)}$ that $c\in \sigma \left(
T\right) $ and $\lambda _{0}(0)\leq c\leq \lambda _{1}(0).$ Since $\lambda
_{0}(0)$ is geometrically simple and $\lambda _{1}(0)$ is geometrically
double, from Lemma $\ref{geo},$ one deduces that
\begin{equation*}
\lambda _{0}(0)<c<\lambda _{1}(0).
\end{equation*}%
This contrary implies $\Lambda _{1}\left( m\right) =\lambda _{1}(m)$ for
sufficiently large $m$.

Assume there exists a subsequence $\left\{ \lambda _{2}(m_{j})\right\}
_{j=1}^{\infty }$ such that $\Lambda _{2}\left( m_{j}\right) >\lambda
_{2}(m_{j})$. Since $\left\{ \lambda _{2}(m_{j})\right\} _{j=1}^{\infty }$
is bounded, without loss of generality, assume
\begin{equation*}
\lambda _{2}(m_{j})\rightarrow c\text{, as }j\rightarrow \infty .
\end{equation*}%
Hence $c\in \sigma \left( T\right) $ and $\lambda _{1}(0)\leq c\leq \lambda
_{2}(0).$ Since $\lambda _{1}(0)$ is geometrically double, from Lemma $\ref%
{geo},$ one deduces a contradiction to imply $\Lambda _{2}\left( m\right)
=\lambda _{2}(m)$ when $m$ is sufficiently large.

Proceeding as in the proof for $\lambda _{0}(0)$ and $\lambda _{1}(0)$, this
theorem will be completed.
\end{proof}

\begin{lemma}
\label{L2} For any $x_{0}\in \lbrack a,b],\ $the initial problem consisting
of equation $(\ref{fangcheng})$ with the initial value
\begin{equation}
y(x_{0},\lambda )=c_{1},\text{ }y^{[1]}(x_{0},\lambda )=c_{2},  \label{6}
\end{equation}
where $c_{1},$ $c_{2}\in
\mathbb{C}
$, has a unique solution $y(x,\lambda )$. And each of the functions $%
y(x,\lambda )$ and $y^{[1]}(x,\lambda )$ is continuous on $[a,b]\times
\mathbb{C}
,$ in particular, the functions $y(x,\lambda )$ and $y^{[1]}(x,\lambda )$
are entire functions of $\lambda \in
\mathbb{C}
.$
\end{lemma}

\begin{proof}
The proof is similar to the proof of Sturm-Liouville problems with regular
potentials, see \cite{CC7}. The last conclusion can also be found in \cite%
{xx14}.
\end{proof}

\begin{lemma}
\bigskip \label{continuous}Consider the initial value problem consisting of
the equation $(\ref{fangcheng})$ and the initial conditions
\begin{equation}
y(c)=h,\text{ }y^{[1]}(c)=k,\text{ }c\in \lbrack a,b].
\end{equation}%
Then, given $c_{j}\in \lbrack a,b],$ $h_{j},k_{j}\in
\mathbb{C}
,$ $1/p_{j},q_{j},r_{j},s_{j}\in L(J,%
\mathbb{R}
),$ $j=1,2,$ and given $\epsilon >0,$ there exists a number $\delta >0$ such
that if
\begin{eqnarray*}
\int_{a}^{b}(\left\vert 1/p_{1}-1/p_{2}\right\vert &+&\hspace{-3mm}%
\left\vert q_{1}-q_{2}\right\vert +\left\vert r_{1}-r_{2}\right\vert
+\left\vert s_{1}-s_{2}\right\vert ) \\
&+&\left\vert c_{1}-c_{2}\right\vert +\left\vert h_{1}-h_{2}\right\vert
+\left\vert k_{1}-k_{2}\right\vert<\delta ,
\end{eqnarray*}%
then%
\begin{equation*}
\left\vert
y(t,c_{2},h_{2},k_{2},1/p_{2},q_{2},r_{2},s_{2})-y(t,c_{1},h_{1},k_{1},1/p_{1},q_{1},r_{1},s_{1})\right\vert <\epsilon
\end{equation*}%
and%
\begin{equation*}
\left\vert
y^{[1]}(t,c_{2},h_{2},k_{2},1/p_{2},q_{2},r_{2},s_{2})-y^{[1]}(t,c_{1},h_{1},k_{1},1/p_{1},q_{1},r_{1},s_{1})\right\vert <\epsilon
\end{equation*}%
uniformly for all $t\in \lbrack a,b].$
\end{lemma}

\begin{proof}
This is a consequence of \cite[Theorem 1.6.2]{CC7}.
\end{proof}

Note that in this paper, we will denote the norm in $L(J,%
\mathbb{R}
)$ by $\left\Vert \cdot \right\Vert _{1}.$

\begin{lemma}
\label{norm}For $m\in
\mathbb{N}
,$ let $L_{m}$ denote the operators generated by the expression $(\ref{b})\ $%
and the self-adjoint boundary conditions $(\ref{25}),$ with the coefficients
$p,$ $q$, $s$ replaced by $p_{m},$ $q_{m}$, $s_{m},$ respectively$.$ The
coefficients $p_{m},$ $q_{m}$, $s_{m}$ are real-valued and%
\begin{equation*}
\text{ }1/p_{m},\text{ }q_{m},\text{ }s_{m}\in L(J,%
\mathbb{R}
),\text{ }p_{m}>0\text{ a.e. on }J.
\end{equation*}
For the operator $L$\ generated by the expression $(\ref{b})\ $and the
self-adjoint boundary conditions $(\ref{25}),$ if
\begin{equation}
\left\Vert 1/p_{m}-1/p\right\Vert _{1}\rightarrow 0,\left\Vert
q_{m}-q\right\Vert _{1}\rightarrow 0,\left\Vert s_{m}-s\right\Vert
_{1}\rightarrow 0,\text{ as }m\rightarrow \infty ,  \label{mbijin}
\end{equation}%
then $L_{m}\overset{R}{\Longrightarrow }L$.
\end{lemma}

\begin{proof}
Let $\lambda \in
\mathbb{C}
\backslash
\mathbb{R}
,$ then $\lambda \in $ $\rho (L)$ $\cap $ $\rho (L_{m}).$ Denote by $\phi
_{1,m}$ and $\phi _{2,m}$ the pair of solutions of the equation%
\begin{equation*}
-(p_{m}\left[ y^{\prime }+s_{m}y\right] )^{\prime }+s_{m}(p_{m}\left[
y^{\prime }+s_{m}y\right] )+q_{m}y=\lambda ry,\text{ }x\in (a,b),
\end{equation*}%
determined by the initial conditions%
\begin{equation}
\phi _{1,m}(a)=\phi _{2,m}^{[1]}(a)=1,\text{ }\phi _{2,m}(a)=\phi
_{1,m}^{[1]}(a)=0.
\end{equation}%
According to Theorem 1.6.1 in \cite{CC7}, (let $y^{[0]}=y$), we have for $%
j=0,1,$%
\begin{eqnarray}
&&\left\vert \phi _{1,m}^{[j]}(t)-\phi _{1}^{[j]}(t)\right\vert +\left\vert
\phi _{2,m}^{[j]}(t)-\phi _{2}^{[j]}(t)\right\vert  \label{1.11} \\
&\leq &C\left( \left\Vert 1/p_{m}-1/p\right\Vert _{1}+\left\Vert
q_{m}-q\right\Vert _{1}+\left\Vert s_{m}-s\right\Vert _{1}\right) ,  \notag
\end{eqnarray}%
where $\phi _{1}$ and $\phi _{2}$ are defined at the beginning of this
section and $C$ depends only on the chosen number $\lambda $ and the fixed
functions $1/p,\ q,$ $s,$ $r$.

Since the Wronskian of the pair $\phi _{1,m}$ and $\phi _{2,m}$ equals 1
identically, a straightforward calculation shows that the function%
\begin{equation*}
z_{m}(x)=\int_{a}^{x}\left( \phi _{1,m}(x)\phi _{2,m}(\xi )-\phi
_{2,m}(x)\phi _{1,m}(\xi )\right) r(\xi )f(\xi )d\xi
\end{equation*}%
satisfies the resolvent equation%
\begin{equation}
\frac{1}{r}(-(p_{m}\left[ y^{\prime }+s_{m}y\right] )^{\prime }+s_{m}(p_{m}%
\left[ y^{\prime }+s_{m}y\right] )+q_{m}y)-\lambda y=f\in L_{r}^{2}(J,%
\mathbb{R}
),x\in (a,b).  \label{14}
\end{equation}%
Also, the function
\begin{equation*}
z(x)=\int_{a}^{x}\left( \phi _{1}(x)\phi _{2}(\xi )-\phi _{2}(x)\phi
_{1}(\xi )\right) r(\xi )f(\xi )d\xi
\end{equation*}%
satisfies the resolvent equation%
\begin{equation}
\frac{1}{r}(-(p\left[ y^{\prime }+sy\right] )^{\prime }+s(p\left[ y^{\prime
}+sy\right] )+qy)-\lambda y=f\in L_{r}^{2}(J,%
\mathbb{R}
),\text{ }x\in (a,b).  \label{16}
\end{equation}%
As before, the solution is understood in the sense of Lemma $\ref{L2}$. By
the estimate $(\ref{1.11})$, we have that
\begin{eqnarray*}
&&\left\vert z_{m}(x)-z(x)\right\vert \\
&\leq &C\left( \left\Vert 1/p_{m}-1/p\right\Vert _{1}+\left\Vert
q_{m}-q\right\Vert _{1}+\left\Vert s_{m}-s\right\Vert _{1}\right)
\int_{a}^{b}\left\vert r(t)f(t)\right\vert dt \\
&\leq &C_{1}\left( \left\Vert 1/p_{m}-1/p\right\Vert _{1}+\left\Vert
q_{m}-q\right\Vert _{1}++\left\Vert s_{m}-s\right\Vert _{1}\right)
\int_{a}^{b}r(t)\left\vert f(t)\right\vert ^{2}dt,
\end{eqnarray*}%
where $C_{1}$ depends only on the chosen number $\lambda $ and the fixed
functions $1/p,\ q,$ $s,$ $r.$

A general solution of equation ($\ref{14}$) and ($\ref{16}$) has the
representation%
\begin{equation*}
y_{m}(x)=z_{m}(x)+a_{1}(m)\phi _{1,m}(x)+a_{2}(m)\phi _{2,m}(x)
\end{equation*}%
and%
\begin{equation*}
y(x)=z(x)+a_{1}\phi _{1}(x)+a_{2}\phi _{2}(x),
\end{equation*}%
respectively.

Substituting $y_{m}(x)$ and $y(x)$ into the boundary conditions $(\ref{25})$%
, denote
\begin{eqnarray*}
A &=&\left(
\begin{array}{cc}
a_{11} & a_{12} \\
a_{21} & a_{22}%
\end{array}%
\right) ,B=\left(
\begin{array}{cc}
b_{11} & b_{12} \\
b_{21} & b_{22}%
\end{array}%
\right) , \\
U_{j}(y) &=&a_{j1}y(a)+a_{j2}y^{[1]}(a)+b_{j1}y(b)+b_{j2}y^{[1]}(b),\text{ }%
j=1,2,
\end{eqnarray*}%
we get%
\begin{equation*}
a_{1}(m)=\triangle _{m}^{-1}\left\vert
\begin{array}{cc}
U_{1}(\phi _{2,m}) & U_{1}(z_{m}) \\
U_{2}(\phi _{2,m}) & U_{2}(z_{m})%
\end{array}%
\right\vert ,\text{ }a_{2}(m)=\triangle _{m}^{-1}\left\vert
\begin{array}{cc}
U_{1}(z_{m}) & U_{1}(\phi _{1,m}) \\
U_{2}(z_{m}) & U_{2}(\phi _{1,m})%
\end{array}%
\right\vert ,
\end{equation*}%
\begin{equation*}
a_{1}=\triangle ^{-1}\left\vert
\begin{array}{cc}
U_{1}(\phi _{2}) & U_{1}(z) \\
U_{2}(\phi _{2}) & U_{2}(z)%
\end{array}%
\right\vert ,\text{ }a_{2}=\triangle ^{-1}\left\vert
\begin{array}{cc}
U_{1}(z) & U_{1}(\phi _{1}) \\
U_{2}(z) & U_{2}(\phi _{1})%
\end{array}%
\right\vert ,
\end{equation*}%
where $\triangle _{m}$ and $\triangle $ are defined as follows:%
\begin{equation*}
\triangle _{m}=\left\vert
\begin{array}{cc}
U_{1}(\phi _{1,m}) & U_{1}(\phi _{2,m}) \\
U_{2}(\phi _{1,m}) & U_{2}(\phi _{2,m})%
\end{array}%
\right\vert ,\triangle =\left\vert
\begin{array}{cc}
U_{1}(\phi _{1}) & U_{1}(\phi _{2}) \\
U_{2}(\phi _{1}) & U_{2}(\phi _{2})%
\end{array}%
\right\vert .
\end{equation*}%
Note that $\triangle _{m}(\lambda )\neq 0$ and $\triangle \left( \lambda
\right) \neq 0$, otherwise, the chosen complex number $\lambda $ is an
eigenvalue of the operators $L_{m}\ $and $L.$

From the estimate $(\ref{1.11})$, we have that
\begin{eqnarray*}
&&\left\vert U_{j}(\phi _{1,m})-U_{j}(\phi _{1})\right\vert +\left\vert
U_{j}(\phi _{2,m})-U_{j}(\phi _{2})\right\vert \\
&\leq &C\left( \left\Vert 1/p_{m}-1/p\right\Vert _{1}+\left\Vert
q_{m}-q\right\Vert _{1}+\left\Vert s_{m}-s\right\Vert _{1}\right) ,\text{ }%
j=1,2,
\end{eqnarray*}%
and therefore
\begin{equation*}
\left\vert \triangle _{m}-\triangle \right\vert \leq C\left( \left\Vert
1/p_{m}-1/p\right\Vert _{1}+\left\Vert q_{m}-q\right\Vert _{1}+\left\Vert
s_{m}-s\right\Vert _{1}\right) ,
\end{equation*}%
where $C$ depends only on the chosen number $\lambda $ and the fixed
functions $1/p,\ q,$ $s,$ $r.$

Consequently,%
\begin{eqnarray*}
&&\left\vert a_{1}(m)-a_{1}\right\vert +\left\vert a_{2}(m)-a_{2}\right\vert
\\
&\leq &C\left( \left\Vert 1/p_{m}-1/p\right\Vert _{1}+\left\Vert
q_{m}-q\right\Vert _{1}+\left\Vert s_{m}-s\right\Vert _{1}\right)
\int_{a}^{b}r(t)\left\vert f(t)\right\vert ^{2}dt,
\end{eqnarray*}%
where $C$ depends only on the chosen number $\lambda $ and the fixed
functions $1/p,\ q,$ $s,$ $r.$ The estimates obtained show that the solutions%
\begin{equation*}
y_{m}=(L_{m}-\lambda )^{-1}f
\end{equation*}%
are subject to the inequality
\begin{eqnarray*}
&&\left\Vert (L_{m}-\lambda )^{-1}f-(L-\lambda )^{-1}f\right\Vert
_{L_{r}^{2}(J,%
\mathbb{R}
)} \\
&=&\left\Vert y_{m}-y\right\Vert _{L_{r}^{2}(J,%
\mathbb{R}
)}=\left( \int_{a}^{b}\left\vert r(t)\right\vert \left\vert
y_{m}-y\right\vert ^{2}dt\right) ^{\frac{1}{2}} \\
&\leq &C\max_{t\in \lbrack a,b]}\left\vert y_{m}(t)-y(t)\right\vert \\
&\leq &C\left( \left\Vert 1/p_{m}-1/p\right\Vert _{1}+\left\Vert
q_{m}-q\right\Vert _{1}+\left\Vert s_{m}-s\right\Vert _{1}\right)
\int_{a}^{b}r(t)\left\vert f(t)\right\vert ^{2}dt,
\end{eqnarray*}%
where $C$ depends only on the chosen number $\lambda $ and the fixed
functions $1/p,\ q,$ $s,$ $r.$

This estimate implies the norm resolvent convergence of the operators $L_{m}$%
.
\end{proof}

\begin{lemma}
\label{L1} A number $\lambda $ is an eigenvalue of Sturm-Liouville problem
consisting of $(\ref{fangcheng})$ and $(\ref{25})$ if and only if
\begin{equation*}
\Delta (\lambda )=:\det (A+B\Phi \left( b,\lambda \right) )=0.
\end{equation*}%
A number $\lambda $ is an eigenvalue of Sturm-Liouville problem consisting
of $(\ref{fangcheng})$ and $(\ref{ee})$ if and only if $D(\lambda )=2\cos
\gamma .\ $A number $\lambda $ is an eigenvalue of Sturm-Liouville problem
consisting of $(\ref{fangcheng})\ $and $(\ref{21})$ if and only if $%
D(\lambda )=2.\ $
\end{lemma}

\begin{proof}
The proof is similar to the proof of the classical Sturm-Liouville problem,
see Lemma 3.2.2, Lemma 3.2.6 in \cite{CC7}.
\end{proof}

\begin{lemma}
\label{L3}For a fixed $\lambda \in
\mathbb{R}
$, the function $y(x,\lambda )$ is a real-valued solution to the initial
problem consisting of equation $(\ref{fangcheng})$ with the initial value $%
y(x_{0},\lambda )=0,$ $y^{[1]}(x_{0},\lambda )=c,$ $c>0.\ $Then there exists
$\varepsilon >0$ such that $y(x,\lambda )<0$ for all $x\in
(x_{0}-\varepsilon ,x_{0})$ and $y(x,\lambda )>0$ for all $x\in
(x_{0},x_{0}+\varepsilon ).$
\end{lemma}

\begin{proof}
See \cite[Lemma 11.2]{xx14}. Since $y^{[1]}(x)=p(x)\left[ y^{\prime
}(x)+s(x)y(x)\right] ,$ it follows that $y^{\prime
}(x)=p(x)^{-1}y^{[1]}(x)-s(x)y(x),$ from the knowledge of the differential
equation, we have
\begin{equation*}
y(x)=e^{-S(x)}\int_{x_{0}}^{x}e^{S(t)}p(t)^{-1}y^{[1]}(t)dt,\text{ \ \ \ }%
S(x)=\int_{x_{0}}^{x}s(t)dt,
\end{equation*}%
thus the claim is obvious.
\end{proof}

\begin{lemma}
\label{L5} Let $y_{1}(x,\lambda _{1})$ and $y_{2}(x,\lambda _{2})$ be two
real-valued solutions of the equation $(\ref{fangcheng})$, and $\lambda
_{2}\geq \lambda _{1},\ $if $x_{1}$ and $x_{2}$ are two adjacent zeros of $%
y_{1}(x)$, then $y_{2}(x)$ has at least one zero on $[x_{1},x_{2}].\ $
\end{lemma}

\begin{proof}
The Lagrange's formula in \cite{xx14} gives
\begin{eqnarray}
(\lambda _{2}-\lambda _{1})\int_{x_{1}}^{x_{2}}y_{1}y_{2}rdt
&=&\int_{x_{1}}^{x_{2}}(y_{1}Ly_{2}-y_{2}Ly_{1})rdt  \notag \\
&=&y_{1}^{[1]}(x_{2})y_{2}(x_{2})-y_{1}^{[1]}(x_{1})y_{2}(x_{1}).
\label{126}
\end{eqnarray}%
Suppose $y_{2}(x)$ has no zeros on $[x_{1},x_{2}]$, then without loss of
generality, we assume that $y_{2}(x)>0$ on $[x_{1},x_{2}]$ and $y_{1}(x)>0$
on $(x_{1},x_{2}),$ thus Lemma \ref{L3} implies that $y_{1}^{[1]}(x_{2})<0$
and $y_{1}^{[1]}(x_{1})>0$. The last term of (\ref{126}) is negative, but
the first term of the equality is not negative, so the contradiction proves
the lemma. In fact, if $\lambda _{2}>\lambda _{1}$, through the similar
process, we can obtain $y_{2}(x)$ has at least one zero on $(x_{1},x_{2})$.
\end{proof}

\section{Preliminary lemmas for the main results}

In this section, we will give several lemmas that will be used in the proofs
of our main results.

To study the Sturm-Liouville problem consisting of $(\ref{fangcheng})\ $and $%
(\ref{25})$, we introduce a \textquotedblleft boundary value problem
space\textquotedblright\ with a metric. Let $\Omega =\{\omega
=(A,B,1/p,q,r,s);$ $(\ref{tx})$ and $(\ref{juzhen})$ hold$\}.$ For the
topology of $\Omega $ we use a metric $d$ defined as follows:

For $\omega =(A,B,1/p,q,r,s)\in \Omega ,$ $\omega _{0}=(A_{0},B_{0},\frac{1}{%
p_{0}},q_{0},r_{0},s_{0})\in \Omega ,$ define
\begin{equation*}
d(\omega ,\omega _{0})=\left\Vert A-A_{0}\right\Vert +\left\Vert
B-B_{0}\right\Vert +\int_{a}^{b}\left( \left\vert \frac{1}{p}-\frac{1}{p_{0}}%
\right\vert +\left\vert q-q_{0}\right\vert +\left\vert r-r_{0}\right\vert
+\left\vert s-s_{0}\right\vert \right) ,
\end{equation*}%
where $\left\Vert \cdot \right\Vert $ denotes any matrix norm.

Note that an element $\omega =(A,B,1/p,q,r,s)\in \Omega $ can be used to
represent a Sturm-Liouville problem consisting of $(\ref{fangcheng})\ $and $(%
\ref{25})$. By an eigenvalue of $\omega \in \Omega $ we mean an eigenvalue
of the Sturm-Liouville problem consisting of $(\ref{fangcheng})\ $and $(\ref%
{25})$.

\begin{lemma}
\bigskip \label{ZEREO}$($Continuity of the zeros of an analytic function$)$.
Let $A$ be an open set in the complex plane $%
\mathbb{C}
$, $F$ a metric space, $f$ a continuous complex valued function on $A\times
F $ such that for each $\alpha \in F$, the map $z\rightarrow f(z,\alpha )$
is an analytic function on $A$. Let $B$ be an open subset of $A$ whose
closure $B$ in $%
\mathbb{C}
$ is compact and contained in $A$, and let $\alpha _{0}$ $\in $ $F$ be such
that no zero of $f(z,\alpha _{0})$ is on the boundary of $B$. Then there
exists a neighborhood $W$ of $\alpha _{0}$ in $F$ such that :

$\left( a\right) $ For any $\alpha $ $\in $ $W$, $f(z,\alpha )$ has no zero
on the boundary of $B$.

$\left( b\right) $ For any $\alpha $ $\in $ $W$, the sum of the orders of
the zeros of $f(z,\alpha )$ contained in $B$ is independent of $\alpha $.
\end{lemma}

\begin{proof}
See \cite[9.17.4]{66}.
\end{proof}

\begin{lemma}
\bigskip \label{lianxuxing}Let $\omega _{0}=(A_{0},B_{0},\frac{1}{p_{0}}%
,q_{0},r_{0},s_{0})\in \Omega .$ Assume that $\lambda (\omega _{0})$ is an
eigenvalue of the problem $(\ref{fangcheng})$, $(\ref{25}).$ Then, given any
$\epsilon >0,$ there exists a $\delta >0$ such that if $\omega
=(A,B,1/p,q,r,s)\in \Omega $ satisfies
\begin{equation*}
d(\omega ,\omega _{0})<\delta ,
\end{equation*}%
then the Sturm-Liouville problem $\omega $ has an eigenvalue $\lambda
(\omega )$ satisfying
\begin{equation*}
\left\vert \lambda (\omega )-\lambda (\omega _{0})\right\vert <\epsilon .
\end{equation*}
\end{lemma}

\begin{proof}
On the basis of Lemma \ref{L2}, Lemma \ref{continuous} and Lemma \ref{L1},
the proof is similar to the proof of the classical Sturm-Liouville problem
with regular potentials, see \cite[Theorem 3.1]{kong}.
\end{proof}

Note that for the Sturm-Liouville problem with distributional potentials
consisting of $(\ref{fangcheng})$ and $(\ref{25})$, the \textbf{algebraic
multiplicity} of an eigenvalue is the order of the eigenvalue as a zero of
the characteristic function $\Delta (\lambda )$ discussed in Lemma \ref{L1}.

\begin{lemma}
Let $\omega _{0}=(A_{0},B_{0},\frac{1}{p_{0}},q_{0},r_{0},s_{0})\in \Omega .$
Assume that $r_{1}$ and $r_{2}$, $r_{1}$ $<$ $r_{2},$ are any two real
numbers such neither of them is an eigenvalue of $\omega _{0}$ and $n\geq 0$
is the number of eigenvalues of $\omega _{0}$ in the interval $%
(r_{1},r_{2}). $ Then there exists a neighborhood $O$ of $\omega _{0}$ in $%
\Omega $ such that any $\omega $ $\in O$ also has $n$ eigenvalues in the
interval $(r_{1},r_{2}).($Here the eigenvalues are counted with algebraic
multiplicity.$)$
\end{lemma}

\begin{proof}
On the basis of Lemma \ref{L2}, Lemma \ref{continuous} and Lemma \ref{L1},
this is a direct consequent of Lemma \ref{ZEREO}.
\end{proof}

\begin{remark}
\label{alge}Furthermore, we can obtain that each algebraically simple
eigenvalue is on a locally unique continuous eigenvalue branch, while each
algebraically double eigenvalue is on two locally unique continuous
eigenvalue branches, the number of the eigenvalue branches is counted with
algebraic multiplicity.
\end{remark}

\begin{lemma}
\bigskip \label{eigenfunction}$(i)$Assume the eigenvalue $\lambda (\omega
_{0})$ is geometrically simple for some $\omega _{0}\in \Omega $ and let $%
w=w(\cdot ,\omega _{0})$ denote a normalized eigenfunction of the eigenvalue
$\lambda (\omega _{0}).$ Then there exist normalized eigenfunctions $%
w=w(\cdot ,\omega )$ of $\lambda (\omega )$ such that
\begin{equation*}
w(\cdot ,\omega )\rightarrow w(\cdot ,\omega _{0}),\text{ }w^{[1]}(\cdot
,\omega )\rightarrow w^{[1]}(\cdot ,\omega _{0}),\text{ as }\omega
\rightarrow \omega _{0}\text{ in }\Omega ,
\end{equation*}%
both uniformly on the interval $[a,b].$

$(ii)$Assume that $\lambda (\omega )$ is a geometrically double eigenvalue
for all $\omega $ in some neighborhood $M$ of $\omega _{0}$ in $\Omega .$
Let $w=w(\cdot ,\omega _{0})$ be any normalized eigenfunction of the
eigenvalue $\lambda (\omega _{0}).$ Then there exist normalized
eigenfunctions $w=w(\cdot ,\omega )$ of $\lambda (\omega )$ such that
\begin{equation*}
w(\cdot ,\omega )\rightarrow w(\cdot ,\omega _{0}),\text{ }w^{[1]}(\cdot
,\omega )\rightarrow w^{[1]}(\cdot ,\omega _{0}),\text{ as }\omega
\rightarrow \omega _{0}\text{ in }\Omega ,
\end{equation*}%
both uniformly on the interval $[a,b].$
\end{lemma}

\begin{proof}
On the basis of Lemma \ref{lianxuxing} and Lemma \ref{continuous}, the proof is similar to \cite[Theorem 3.2]{kong}.
\end{proof}

\begin{lemma}
\label{L4} The eigenvalues $\left\{ \lambda _{n},n\in
\mathbb{N}
_{0}\right\} $ of the separated Sturm-Liouville problems consisting of $(\ref%
{fangcheng})\ $and $(\ref{11})$ are all geometrically simple, real, and form
a sequence accumulating to $+\infty :$%
\begin{equation*}
-\infty <\lambda _{0}<\lambda _{1}<\lambda _{2}<\lambda _{3}<\cdots ,
\end{equation*}%
the number of times an eigenvalue appears in the sequence is equal to its
geometric multiplicity$;$ the eigenfunction ${\Psi }_{n}(x)$ corresponding
to the eigenvalue $\lambda _{n}$ has exactly $n$ zeros on the interval $%
(a,b) $.
\end{lemma}

\begin{proof}
The assertion of the eigenvalues and the eigenfunctions can be proved by the
Pr\"{u}fer transformation introduced in \cite{CC4}. Denote $y_{1}=y,$ $%
y_{2}=y^{[1]},$ assume
\begin{equation*}
\left\{
\begin{array}{l}
y_{1}(t,\lambda )=\rho (t,\lambda )\sin \theta (t,\lambda ), \\
y_{2}(t,\lambda )=\rho (t,\lambda )\cos \theta (t,\lambda ),%
\end{array}%
\right.
\end{equation*}%
we can obtain
\begin{equation}
\theta ^{\prime }(t,\lambda )=\frac{1}{p(t)}\cos ^{2}\theta (t,\lambda
)-s(t)\sin 2\theta (t,\lambda )+\left( \lambda r(t)-q(t)\right) \sin
^{2}\theta (t,\lambda ).  \label{s1}
\end{equation}%
For the equation (\ref{s1}) with the initial condition
\begin{equation}
\theta (a,\lambda )=\alpha ,\text{ }0\leq \alpha <\pi ,\text{ }\lambda \in
\mathbb{R}
,  \label{s2}
\end{equation}%
using the method introduced in \cite[Theorem 4.5.3]{CC7}, we can obtain the
properties of the unique real valued solution $\theta (t,\lambda )$ defined
on $[a,b]:$

(1) For fixed $t\in (a,b],$ $\theta (t,\lambda )$ is continuous and strictly
increasing in $\lambda $.

(2) If $\theta (c,\lambda )=k\pi ,$ for some $c\in (a,b),$ $\lambda \in
\mathbb{R}
,$ and some $k\in
\mathbb{N}
,$ then $\theta (t,\lambda )>k\pi $ for $c<t\leq b.$

(3) $\theta (b,\lambda )\rightarrow \infty ,$ as $\lambda \rightarrow \infty
$.

(4) For all $t\in (a,b],$ $\theta (t,\lambda )\rightarrow 0,$ as $\lambda
\rightarrow -\infty $.

To prove part (1), assume $\lambda _{1}>\lambda _{2},$ $\theta _{1}=\theta
(t,\lambda _{1}),$ $\theta _{2}=\theta (t,\lambda _{2}),$ $V=\theta _{1}-$ $%
\theta _{2}$, then $V^{\prime }$$=fV+h,$ where
\begin{eqnarray*}
f &=&-2s(t)\frac{\sin 2\theta _{1}-\sin 2\theta _{2}}{2\theta _{1}-2\theta
_{2}}+\left[ \lambda _{2}r(t)-q(t)-\frac{1}{p(t)}\right] (\sin \theta
_{1}+\sin \theta _{2})\frac{\sin \theta _{1}-\sin \theta _{2}}{\theta
_{1}-\theta _{2}}, \\
h &=&(\lambda _{1}r(t)-\lambda _{2}r(t))\sin ^{2}\theta _{1}.
\end{eqnarray*}%
Then the claim follows from the proof that is similar to the proof in \cite[%
Theorem 4.5.1 and 4.5.2]{CC7}. The proof of (2) is similar to that of part
(1). To prove part (2), we consider the following equation, let $\theta
=\theta (t,\lambda ),$ $V(t)=\theta (t,\lambda )-k\pi ,$ then $V^{\prime }$$%
=fV+h,$ where%
\begin{eqnarray*}
f(t)\hspace{-3mm} &=&\hspace{-3mm}-2s(t)\frac{\sin 2\theta -\sin 2k\pi }{%
2\theta -2k\pi }+\left[ \lambda r(t)-q(t)-\frac{1}{p(t)}\right] (\sin \theta
+\sin k\pi )\frac{\sin \theta -\sin k\pi }{\theta -k\pi }, \\
h(t)\hspace{-3mm} &=&\hspace{-3mm}\frac{1}{p(t)}.
\end{eqnarray*}

To prove part (3) let $\tan \phi =\lambda ^{\frac{1}{2}}\tan \theta $ for $%
\lambda >0$ and determine $\phi $ uniquely by requiring $\left\vert \theta
-\phi \right\vert <\frac{\pi }{2}.$ Then
\begin{eqnarray*}
\phi ^{\prime } &=&\lambda ^{\frac{1}{2}}\frac{1}{p}\cos ^{2}\phi +\lambda ^{%
\frac{1}{2}}r\sin ^{2}\phi -\lambda ^{-\frac{1}{2}}q\sin ^{2}\phi -s\sin
2\phi , \\
\phi (b,\lambda )-\phi (a,\lambda ) &=&\lambda ^{\frac{1}{2}%
}\int_{a}^{b}\left( \frac{1}{p}\cos ^{2}\phi +r\sin ^{2}\phi \right)
-\lambda ^{-\frac{1}{2}}\int_{a}^{b}q\sin ^{2}\phi dt-\int_{a}^{b}s\sin
2\phi dt, \\
\phi (b,\lambda ) &\geq &\phi (a,\lambda )+\lambda ^{\frac{1}{2}%
}\int_{a}^{b}\min (\frac{1}{p},r)-\lambda ^{-\frac{1}{2}}\int_{a}^{b}\left%
\vert q\right\vert -\int_{a}^{b}\left\vert s\right\vert ,
\end{eqnarray*}%
hence $\phi (b,\lambda )\rightarrow \infty ,$ as $\lambda \rightarrow \infty
.$ Therefore $\theta (b,\lambda )\rightarrow \infty ,$ as $\lambda
\rightarrow \infty .$

Finally, for the proof of part (4), let $\theta _{-\infty
}(t)=\lim\limits_{\lambda \rightarrow -\infty }\theta (t,\lambda ),$ for $%
t\in (a,b].$ Since $\theta (t,\lambda )$ is strictly increasing in $\lambda
, $ and $0\leq \theta (t,\lambda )\leq \theta (t,0)\ $for $\lambda <0,$ so
the limit exists.%
\begin{eqnarray*}
\int_{a}^{b}\lambda r(t)\sin ^{2}\theta (t,\lambda )dt &=&\theta (b,\lambda
)-\alpha +\int_{a}^{b}s(t)\sin 2\theta (t,\lambda )dt \\
&&-\int_{a}^{b}\frac{1}{p(t)}\cos ^{2}\theta (t,\lambda
)dt+\int_{a}^{b}q\sin ^{2}\theta (t,\lambda )dt.
\end{eqnarray*}%
Hence
\begin{equation*}
\int_{a}^{b}r(t)\sin ^{2}\theta (t,\lambda )dt\rightarrow 0,\text{ as }%
\lambda \rightarrow -\infty .
\end{equation*}%
Let $\{\lambda _{n},n\in
\mathbb{N}
\}\rightarrow -\infty $ and define $f_{n}(s)=r(s)\sin ^{2}\theta (s,\lambda
_{n}),$ then $\left\vert f_{n}(s)\right\vert \leq r(s)$ and $%
f_{n}(s)\rightarrow r(s)\sin ^{2}\theta _{-\infty }(s)$ as $n\rightarrow
\infty .$ Hence $\int_{a}^{b}r(s)\sin ^{2}\theta _{-\infty }(s)ds=0,$ and
thus $r(s)\sin ^{2}\theta _{-\infty }(s)=0,$ a.e. and $\theta _{-\infty
}(s)=0($mod $\pi ).$ For $\lambda <0,$ $s<t,$ $s,t\in \lbrack a,b],$%
\begin{eqnarray*}
&&\theta (t,\lambda )-\theta (s,\lambda ) \\
&=&\int_{s}^{t}\frac{1}{p(x)}\cos ^{2}\theta (x,\lambda
)dx+\int_{s}^{t}\left( \lambda r(x)-q(x)\right) \sin ^{2}\theta (x,\lambda
)dx \\
&&-\int_{s}^{t}s(x)\sin 2\theta (x,\lambda )dx \\
&\leq &\int_{s}^{t}\frac{1}{p(x)}dx-\int_{s}^{t}q(x)\sin ^{2}\theta
(x,\lambda )dx-\int_{s}^{t}s(x)\sin 2\theta (x,\lambda )dx.
\end{eqnarray*}%
Let $\lambda \rightarrow -\infty ,$ then we have $\theta _{-\infty }(t)\leq
\theta _{-\infty }(s)+\int_{s}^{t}\frac{1}{p(x)}dx$. Then we can get $\theta
(t,\lambda )\rightarrow 0,$ as $\lambda \rightarrow -\infty $ by a similar
proof to \cite[Theorem 4.5.3]{CC7}.

As a similar proof to the classical Sturm-Liouville problems in \cite{xx10},
we can easily obtain that $\lambda $ is an eigenvalue of the problem$(\ref%
{fangcheng}),$ $(\ref{11})$ if and only if $\theta (b,\lambda )=\beta +n\pi
. $ Thus we can obtain our claims.
\end{proof}

The following result describes an asymptotic form of the fundamental
solutions of the equation $(\ref{fangcheng})$ for sufficiently negative $%
\lambda .$

\begin{lemma}
\label{solutio}There exists $\lambda _{0}$ $\in
\mathbb{R}
,$ $k>0$ and a continuous function%
\begin{equation*}
\alpha :[a,b]\times (-\infty ,\lambda _{0}]\rightarrow \lbrack 0,\infty )
\end{equation*}%
such that $\alpha (t,\lambda )$ is decreasing in $\lambda $ for each $t\in
(a,b],$ $\alpha _{t}(t,\lambda )$ exists a.e. on $[a,b]$ for $\lambda \in
(-\infty ,\lambda _{0}],$ $p(t)(\tanh (\alpha (t,\lambda ))\alpha
_{t}(t,\lambda )+s(t))$ is continuous on $[a,b]$ for $\lambda \in (-\infty
,\lambda _{0}],$ and%
\begin{equation*}
\lim\limits_{\lambda \rightarrow -\infty }\alpha (t,\lambda )=\infty ,\text{
}\lim\limits_{\lambda \rightarrow -\infty }p(t)(\tanh (\alpha (t,\lambda
))\alpha _{t}(t,\lambda )+s(t))=\infty
\end{equation*}%
for each $t\in (a,b].$ Moreover, for the fundamental solutions $\phi _{1}$
and $\phi _{2}$ of $(\ref{fangcheng})$ we have%
\begin{eqnarray}
\phi _{1}(t,\lambda ) &=&k\cosh (\alpha (t,\lambda )),  \label{tyjy1} \\
\phi _{1}^{[1]}(t,\lambda ) &=&k\cosh (\alpha (t,\lambda ))p(t)(\tanh
(\alpha (t,\lambda ))\alpha _{t}(t,\lambda )+s(t)),  \label{tyjy 2} \\
\phi _{2}(t,\lambda ) &=&\frac{1}{k^{2}}\phi _{1}(t,\lambda )\int_{a}^{t}%
\frac{\mathit{\mathrm{sech}}^{2}\mathit{(}\alpha (s,\lambda ))}{p(s)}ds,
\label{tyjy 3} \\
\phi _{2}^{[1]}(t,\lambda ) &=&\frac{1}{k^{2}}\phi _{1}^{[1]}(t,\lambda
)\int_{a}^{t}\frac{\mathit{\mathrm{sech}}^{2}\mathit{(}\alpha (s,\lambda ))}{%
p(s)}ds+\frac{1}{k}\mathit{\mathrm{sech}(}\alpha (t,\lambda ))
\label{tyjy 4}
\end{eqnarray}%
on $[a,b]\times (-\infty ,\lambda _{0}].$
\end{lemma}

\begin{proof}
Consider the Sturm-Liouville problem consisting of $(\ref{fangcheng})\ $and
\begin{equation*}
y^{[1]}(a)=y^{[1]}(b)=0.
\end{equation*}%
Let $\lambda _{0}$ be the smallest eigenvalue of this problem. Then $\phi
_{1}(\cdot ,\lambda _{0})$ is an eigenfunction for $\lambda _{0}.$ Hence $%
\phi _{1}(\cdot ,\lambda _{0})$ has no zero on $[a,b].$ So there exists $k>0$
such that $\phi _{1}(t,\lambda _{0})>k$ for each $t\in \lbrack a,b].$ Denote
$\theta (t,\lambda )$ the Pr\"{u}fer angle for $\phi _{1}(t,\lambda ),$
\begin{equation*}
\left\{
\begin{array}{l}
\phi _{1}(t,\lambda )=\rho (t,\lambda )\sin \theta (t,\lambda ) \\
\phi _{1}^{[1]}(t,\lambda )=\rho (t,\lambda )\cos \theta (t,\lambda ).%
\end{array}%
\right.
\end{equation*}%
For each $t\in (a,b]$ and $\lambda \in (-\infty ,\lambda _{0}]$, $\theta
(a,\lambda )=\theta (a,\lambda _{0})=\frac{\pi }{2},$ from Lemma $\ref{L4},$
we have $\theta (t,\lambda )$ is strictly increasing in $\lambda $ and $%
\theta (t,\lambda )\in (0,\pi ).$ Since $\cot \theta (t,\lambda )=\frac{\phi
_{1}^{[1]}(t,\lambda )}{\phi _{1}(t,\lambda )},$ hence we have%
\begin{equation}
\frac{\phi _{1}^{[1]}(t,\lambda )}{\phi _{1}(t,\lambda )}\geq \frac{\phi
_{1}^{[1]}(t,\lambda _{0})}{\phi _{1}(t,\lambda _{0})},\text{ \ for }t\in
(a,b]\text{ a.e., }\lambda \leq \lambda _{0},  \label{5.3}
\end{equation}%
\begin{equation*}
\frac{\phi _{1}^{\prime }(t,\lambda )}{\phi _{1}(t,\lambda )}\geq \frac{\phi
_{1}^{\prime }(t,\lambda _{0})}{\phi _{1}(t,\lambda _{0})},\text{ \ for }%
t\in (a,b]\text{ a.e., }\lambda \leq \lambda _{0},
\end{equation*}%
\begin{equation*}
\left( \ln \frac{\phi _{1}(t,\lambda )}{\phi _{1}(t,\lambda _{0})}\right)
^{\prime }\geq 0,\text{ i.e. }\ln \frac{\phi _{1}(t,\lambda )}{\phi
_{1}(t,\lambda _{0})}\geq \ln \frac{\phi _{1}(a,\lambda )}{\phi
_{1}(a,\lambda _{0})}=0,
\end{equation*}%
the above inequality implies that $\phi _{1}(t,\lambda )\geq \phi
_{1}(t,\lambda _{0})>k$ for each $t\in (a,b]$ and $\lambda \leq \lambda
_{0}. $ In the same way we see that $\phi _{1}(t,\lambda )$ is strictly
decreasing in $\lambda $ on $(-\infty ,\lambda _{0}]$ for each fixed $t\in
(a,b].$

There is a unique $\alpha :[a,b]\times (-\infty ,\lambda _{0}]\rightarrow
\lbrack 0,\infty )$ determined by $(\ref{tyjy1})$ which is continuous.
Moreover, $\alpha (t,\lambda )$ is decreasing in $\lambda $ on $(-\infty
,\lambda _{0}]$ for each $t\in (a,b],\alpha _{t}(t,\lambda )$ exists a.e. on
$[a,b]$. By the reduction of order formula we see that $\phi _{2}$ satisfies
$(\ref{tyjy 3})$ and $\phi _{1}^{[1]}$, $\phi _{2}^{[1]}$ satisfy $(\ref%
{tyjy 2}),$ $(\ref{tyjy 4}),$ respectively.

From Lemma $\ref{L4},$ we have for $t\in (a,b],$ $\theta (t,\lambda
)\rightarrow 0,$ as $\lambda \rightarrow -\infty $. So,%
\begin{equation}
\lim\limits_{\lambda \rightarrow -\infty }\frac{\phi _{1}^{[1]}(t,\lambda )}{%
\phi _{1}(t,\lambda )}=\infty ,\text{ \ \ for }t\in (a,b].  \label{5.2}
\end{equation}%
Now we show that%
\begin{equation*}
\lim\limits_{\lambda \rightarrow -\infty }\alpha (t,\lambda )=\infty ,\text{
\ \ for }t\in (a,b].
\end{equation*}%
Assume the contrary, without loss of generality, let $\lim\limits_{\lambda
\rightarrow -\infty }\alpha (b,\lambda )=r<\infty .$ Then $\alpha (b,\lambda
)\leq r$ on $(-\infty ,\lambda _{0}].$

From $(\ref{5.2}),$ there is $L\leq \lambda _{0}$ such that $\frac{\phi
_{1}^{[1]}(b,L)}{\phi _{1}(b,L)}>0,$ so $\phi _{1}^{[1]}(b,L)>0.$ By the
continuity of $\phi _{1}^{[1]}(\cdot ,L)$ we have that $\phi
_{1}^{[1]}(t,L)>0$ on $[c,b]$ for some $c\in (a,b).$ In view of $(\ref{5.3})$
with $\lambda _{0}$ replaced by $L,$ we see that for $\lambda \leq L,$ $%
\frac{\phi _{1}^{[1]}(t,\lambda )}{\phi _{1}(t,\lambda )}>0$ on $[c,b],$ so $%
\phi _{1}^{[1]}(t,\lambda )>0$ on $[c,b].$

From the knowledge of the differential equation,\ for a solution of
\begin{equation*}
y^{\prime }(x)=p(x)^{-1}y^{[1]}(x)-s(x)y(x),
\end{equation*}%
we have
\begin{equation*}
y(x)=e^{-S(x)}\left( \int_{x_{0}}^{x}e^{S(t)}\frac{1}{p(t)}%
y^{[1]}(t)dt+y(x_{0})\right) ,\text{ \ \ \ }S(x)=\int_{x_{0}}^{x}s(t)dt.
\end{equation*}%
Thus we have
\begin{equation*}
\phi _{1}(b,\lambda )\geq e^{-\int_{t}^{b}s(u)du}\phi _{1}(t,\lambda )\text{%
, \ for }t\in \lbrack c,b]\text{ and }\lambda \leq L.
\end{equation*}%
So
\begin{eqnarray*}
k &\leq &\phi _{1}(t,\lambda )\leq e^{\int_{t}^{b}s(u)du}\phi _{1}(b,\lambda
) \\
&\leq &e^{\left\vert \int_{t}^{b}s(u)du\right\vert }\phi _{1}(b,\lambda
)\leq e^{\int_{a}^{b}\left\vert s(u)\right\vert du}k\cosh (\alpha (b,\lambda
))\leq C,
\end{eqnarray*}%
where $C=$ $e^{\int_{a}^{b}\left\vert s(u)\right\vert du}k\cosh r$ is a
constant independent of $\lambda $ and $t$. However, for $\lambda \leq L,$
from the equation $(\ref{fangcheng}),$ we have%
\begin{equation*}
\phi _{1}^{[1]}(t,\lambda )=e^{\int_{c}^{t}s(u)du}\left(
\int_{c}^{t}e^{-\int_{c}^{u}s(v)dv}(q(u)-\lambda r(u))\phi _{1}(u,\lambda
)du+\phi _{1}^{[1]}(c,\lambda )\right) ,
\end{equation*}%
\begin{eqnarray*}
\phi _{1}(b,\lambda )
&=&e^{-\int_{c}^{b}s(t)dt}\int_{c}^{b}e^{\int_{c}^{t}s(v)dv}\frac{1}{p(t)}%
\phi _{1}^{[1]}(t,\lambda )dt+e^{-\int_{c}^{b}s(t)dt}\phi _{1}(c,\lambda ) \\
&=&e^{-\int_{c}^{b}s(t)dt}\int_{c}^{b}e^{2\int_{c}^{t}s(v)dv}\frac{1}{p(t)}%
\int_{c}^{t}e^{-\int_{c}^{u}s(v)dv}(q(u)-\lambda r(u))\phi _{1}(u,\lambda
)dudt \\
&&+e^{-\int_{c}^{b}s(t)dt}\int_{c}^{b}e^{2\int_{c}^{t}s(v)dv}\frac{1}{p(t)}%
\phi _{1}^{[1]}(c,\lambda )dt+e^{-\int_{c}^{b}s(t)dt}\phi _{1}(c,\lambda ) \\
&\geq &e^{-\int_{c}^{b}s(t)dt}\int_{c}^{b}e^{2\int_{c}^{t}s(v)dv}\frac{1}{%
p(t)}\int_{c}^{t}e^{-\int_{c}^{u}s(v)dv}q(u)\phi _{1}(u,\lambda )dudt \\
&&-\lambda e^{-\int_{c}^{b}s(t)dt}\int_{c}^{b}e^{2\int_{c}^{t}s(v)dv}\frac{1%
}{p(t)}\int_{c}^{t}e^{-\int_{c}^{u}s(v)dv}r(u)\phi _{1}(u,\lambda )dudt \\
&\rightarrow &\infty ,\text{ as }\lambda \rightarrow -\infty .
\end{eqnarray*}%
Now we reach a contradiction to obtain our claim.
\end{proof}

\begin{lemma}
\label{12 copy(1)}For $\frac{1}{p},s\in L(J,%
\mathbb{R}
),$ $p>0$ a.e. on $(a,b),$ then there exist $p_{m}\in C^{\infty }[a,b],$ $%
s_{m}\in C_{0}^{\infty }(a,b),$ $m\in
\mathbb{N}
,$ $p_{m}$ and $s_{m}$ are real-valued, $p_{m}>0$ on $[a,b],$ such that $%
\left\Vert \frac{1}{p_{m}}-\frac{1}{p}\right\Vert _{1}\rightarrow 0,$ $%
\left\Vert s_{m}-s\right\Vert _{1}\rightarrow 0,$ as $m\rightarrow \infty .$
\end{lemma}

\begin{proof}
For $\frac{1}{p},s\in L(J,%
\mathbb{R}
),$ define%
\begin{equation*}
\rho (x)=\left\{
\begin{array}{l}
Ce^{\frac{1}{\left\vert x\right\vert ^{2}-1}},\text{ }\left\vert
x\right\vert <1, \\
0,\hspace{13mm}\left\vert x\right\vert \geq 1,\ \text{\ }%
\end{array}%
\right.
\end{equation*}%
where $C=\left( \int_{-1}^{1}e^{\frac{1}{\left\vert x\right\vert ^{2}-1}%
}dx\right) ^{-1},$ and define $\rho _{\frac{1}{m}}(x)=m\rho (mx),$ $m\in
\mathbb{N}
,$ let $\widetilde{p}_{m}(x)=\int_{a}^{b}\rho _{\frac{1}{m}}(x-y)\frac{1}{%
p(y)}dy,$ it is clear that $\widetilde{p}_{m}(x)>0$ on $[a,b].$ From \cite[%
Theorem 2.29]{cc9}, we know that
\begin{equation*}
\widetilde{p}_{m}(x)\in C^{\infty }[a,b],\text{ }\left\Vert \text{ }%
\widetilde{p}_{m}-\frac{1}{p}\right\Vert _{1}\rightarrow 0,\text{ as }%
m\rightarrow \infty .
\end{equation*}%
So if we let $p_{m}=\frac{1}{\widetilde{p}_{m}},$ we can obtain
\begin{equation*}
p_{m}\in C^{\infty }[a,b],\text{ }\left\Vert \frac{1}{p_{m}}-\frac{1}{p}%
\right\Vert _{1}\rightarrow 0,\text{ as }m\rightarrow \infty .
\end{equation*}%
For $s\in L(J,%
\mathbb{R}
),$ since $C_{0}^{\infty }(a,b)$ is dense in $L(J,%
\mathbb{R}
),$ so we can find $s_{m}\in C_{0}^{\infty }(a,b),$ such that $\left\Vert
s_{m}-s\right\Vert _{1}\rightarrow 0,$ as $m\rightarrow \infty .$
\end{proof}

\begin{remark}
For $p,s$ which satisfy the condition $(\ref{tx})$, let $L=L(p,s)$ denote
the operator generated by $(\ref{b})$ with the boundary conditions $(\ref{25}%
).$ For the case of $p=p_{m},$ $s=s_{m},$ denote the operators $%
L_{m}=L(p_{m},s_{m}).$ For the operators $L_{m}$ which satisfy the
conditions $(\ref{25}),$ since $s_{m}\in C_{0}^{\infty }(a,b),$ we can
assume $s_{m}(a)=s_{m}(b)=0,$ thus
\begin{eqnarray*}
y^{[1]}(a) &=&\left[ p_{m}y^{\prime }+p_{m}s_{m}y\right] (a)=(p_{m}y^{\prime
})(a), \\
y^{[1]}(b) &=&\left[ p_{m}y^{\prime }+p_{m}s_{m}y\right] (b)=(p_{m}y^{\prime
})(b).
\end{eqnarray*}%
Hence the operators $L_{m}=L(p_{m},s_{m})$ are actually generated by the
expression
\begin{eqnarray*}
L_{m}y &=&\frac{1}{r}(-(p_{m}y^{\prime })^{\prime }-(p_{m}s_{m})^{\prime
}y+p_{m}s_{m}^{2}y+qy) \\
&=&\frac{1}{r}(-(p_{m}y^{\prime })^{\prime }+\left( -(p_{m}s_{m})^{\prime
}+p_{m}s_{m}^{2}+q\right) y)
\end{eqnarray*}%
with the boundary conditions%
\begin{equation*}
A\left(
\begin{array}{c}
y(a) \\
(p_{m}y^{\prime })(a)%
\end{array}%
\right) +B\left(
\begin{array}{c}
y(b) \\
(p_{m}y^{\prime })(b)%
\end{array}%
\right) =\left(
\begin{array}{c}
0 \\
0%
\end{array}%
\right) ,
\end{equation*}%
where the complex matrices $A$ and $B$ satisfy that the $2\times 4$ matrix $%
(A|B)$ has full rank, and
\begin{equation*}
AEA^{\ast }=BEB^{\ast },E=\left(
\begin{array}{cc}
0 & -1 \\
1 & 0%
\end{array}%
\right) .
\end{equation*}%
It is obvious that each of the operators $L_{m}$ is a classical self-adjoint
Sturm-Liouville operator with a regular potential and a positive leading
coefficient.
\end{remark}

\begin{remark}
Note that for the case of $p=p_{m},$ $s=s_{m},$ the replacement of the
boundary conditions $(\ref{21}),(\ref{ee}),(\ref{ff}),(\ref{gg})$ are%
\begin{eqnarray}
Y(b) &=&KY(a),  \label{z21} \\
Y(b) &=&e^{i\gamma }KY(a),  \label{zee} \\
y(a) &=&0,\text{ \ }k_{22}y(b)-k_{12}(p_{m}y^{\prime })(b)=0,  \label{zff} \\
(p_{m}y^{\prime })(a) &=&0,\text{ \ }k_{21}y(b)-k_{11}(p_{m}y^{\prime
})(b)=0,  \label{zgg}
\end{eqnarray}%
where $Y($\textperiodcentered $)=\left(
\begin{array}{c}
y(\text{\textperiodcentered }) \\
(p_{m}y^{\prime })(\text{\textperiodcentered })%
\end{array}%
\right) .\ $
\end{remark}

\begin{lemma}
\label{L13}The operators $L_{m}$ converge to the operator $L$ in the sense
of the norm resolvent convergence, i.e. $L_{m}\overset{R}{\Longrightarrow }L$
as $m\rightarrow \infty .$
\end{lemma}

\begin{proof}
This is a direct consequence of Lemma $\ref{norm}.$
\end{proof}

\begin{lemma}
\label{L0}The spectrum of the operators $L$ generated by the expression $(%
\ref{b})$ and the self-adjoint boundary conditions $(\ref{25})$ is discrete.
\end{lemma}

\begin{proof}
From the above lemma, we can find a family of classical Sturm-Liouville
operators $L_{m}$ with regular potentials such that $L_{m}\overset{R}{%
\Longrightarrow }L,$ for $\mu \in \rho (L)$ $\cap $ $\rho (L_{m}),$ the
operator $(L-\mu )^{-1}$ is compact since it is the norm limit of the
compact operators $(L_{m}-\mu )^{-1}$, from \cite{xx13}, the spectrum of the
operator $L$ is discrete.
\end{proof}

\begin{lemma}
For the Sturm-Liouville\ problem consisting of $(\ref{fangcheng})$ and $(\ref%
{25}),$ the geometric and algebraic multiplicity of each eigenvalue are
always equal.
\end{lemma}

\begin{proof}
From \cite{XIN1}, we know that the geometric and algebraic multiplicity of
each eigenvalue of self-adjoint Sturm-Liouville operators $L_{m}$ with
regular potentials are always equal. Thus this theorem is a direct
consequence of Remark $\ref{alge}$ and Lemma $\ref{geo}.$
\end{proof}

\begin{remark}
The eigenvalues for the self-adjoint Sturm-Liouville operators with
distributional potentials discussed in this paper can be ordered to form a
sequence so that the number of an eigenvalue appears in the sequence is
equal to its multiplicity without distinguishing the geometric and algebraic
multiplicity.
\end{remark}

For the proof of the following theorem, we introduce
\begin{equation*}
D_{m}(\lambda ):=k_{11}\phi _{2,m}^{[1]}(b,\lambda )-k_{21}\phi
_{2,m}(b,\lambda )+k_{22}\phi _{1,m}(b,\lambda )-k_{12}\phi
_{1,m}^{[1]}(b,\lambda ),
\end{equation*}%
where $\phi _{1,m}$ and $\phi _{2,m}$ are the pair of solutions of the
equation%
\begin{equation*}
-(p_{m}\left[ y^{\prime }+s_{m}y\right] )^{\prime }+s_{m}(p_{m}\left[
y^{\prime }+s_{m}y\right] )+q_{m}y=\lambda ry,\text{ }x\in (a,b),
\end{equation*}%
determined by the initial conditions%
\begin{equation*}
\phi _{1,m}(a)=\phi _{2,m}^{[1]}(a)=1,\text{ }\phi _{2,m}(a)=\phi
_{1,m}^{[1]}(a)=0.
\end{equation*}%
Together with the properties of the coefficients $p_{m}$ and $s_{m},$ we
have
\begin{equation*}
D_{m}(\lambda )=k_{11}(p_{m}\phi _{2,m}^{\prime })(b,\lambda )-k_{21}\phi
_{2,m}(b,\lambda )+k_{22}\phi _{1,m}(b,\lambda )-k_{12}(p_{m}\phi
_{1,m}^{\prime })(b,\lambda ),
\end{equation*}%
where $\phi _{1,m}$ and $\phi _{2,m}$ are the pair of solutions of the
equation%
\begin{equation*}
-(p_{m}y^{\prime })^{\prime }+\left( -(p_{m}s_{m})^{\prime
}+p_{m}s_{m}^{2}+q\right) y=\lambda ry,\text{ }x\in (a,b),
\end{equation*}%
determined by the initial conditions%
\begin{equation*}
\phi _{1,m}(a)=(p_{m}\phi _{2,m}^{\prime })(a)=1,\text{ }\phi
_{2,m}(a)=(p_{m}\phi _{1,m}^{\prime })(a)=0.
\end{equation*}%
The properties of $D_{m}(\lambda )$ introduced above have been investigated
in \cite{CC5}.

\begin{lemma}
\label{L0 copy(1)}The lowest eigenvalues of $L_{m}$ are uniformly
semi-bounded from below.
\end{lemma}

\begin{proof}
We divide our proof in four steps.

(A) First, we consider the case for the separated boundary condition. Denote
the first eigenvalue of $L=L(p,s)$ and $L_{m}=L(p_{m},s_{m})$ by $\lambda
_{0}(p,s)$ and $\lambda _{0}(p_{m},s_{m}),$ respectively. It suffices to
show that
\begin{equation*}
\lambda _{0}(p_{m},s_{m})\rightarrow \lambda _{0}(p,s),\text{ as }%
m\rightarrow \infty .
\end{equation*}
Since $L_{m}\overset{R}{\Longrightarrow }L,$ there exist eigenvalues $%
\Lambda (m)$ of $L_{m}$ such that $\Lambda (m)\rightarrow \lambda _{0}(p,s),$
as $m\rightarrow \infty ,$ and $\Lambda (m)$ is also simple when $m$ is
sufficiently large. Denote the normalized eigenfunction of $\Lambda (m)$ and
$\lambda _{0}(p,s)$ by $w_{m}=w(\cdot ,p_{m},s_{m})$ and $w=w(\cdot ,p,s),$
respectively.

According to Lemma \ref{eigenfunction}, for an arbitrary $\epsilon >0,$
there exists a number $M>0$ such that if $m>M,$%
\begin{equation}
\left\vert w(t,p_{m},s_{m})-w(t,p,s)\right\vert <\epsilon  \label{8.10}
\end{equation}%
and%
\begin{equation}
\left\vert w^{[1]}(t,p_{m},s_{m})\rightarrow w^{[1]}(t,p,s)\right\vert
<\epsilon  \label{8.11}
\end{equation}%
uniformly for all $t\in \lbrack a,b].$

Note that $w(t,p,s)$ does not have a zero in $(a,b)$ from Lemma \ref{L4}$.$
So we may assume that $w(t,p,s)>0$ on $(a,b).$

$($i$)$ If $w(a,p,s)=0$, from Lemma \ref{L3}, it is a fact that $%
w^{[1]}(a,p,s)>0.$ Hence there exists $\epsilon _{1}>0$ such that $%
w^{[1]}(t,p,s)>0$ on the interval $\left[ a,a+\epsilon _{1}\right] .$ From
the knowledge of $(\ref{8.11})$, when $m$ is sufficiently large, $%
w^{[1]}(t,p_{m},s_{m})>0$ on $\left[ a,a+\epsilon _{1}\right] .$ Since the
boundary condition is fixed, so $w(a,p_{m},s_{m})=0.$ Thus by Lemma \ref{L3}%
, $w(t,p_{m},s_{m})>0$ on the interval $\left( a,a+\epsilon _{1}\right) $
when $m$ is sufficiently large$.$

If $w(b,p,s)=0,$ through a similar process, we can also obtain that for
sufficiently large $m$, there must exist $\epsilon _{2}>0$ such that $%
w(t,p_{m},s_{m})>0$ on the interval $\left( b-\epsilon _{2},b\right) .$
Since $w(t,p,s)>0$ on $\left[ a+\epsilon _{1,}b-\epsilon _{2}\right] ,$
combining with $(\ref{8.10}),$ one deduces that $w(t,p_{m},s_{m})>0$ on $%
\left[ a+\epsilon _{1,}b-\epsilon _{2}\right] $ when $m$ is sufficiently
large$.$ So $w(t,p_{m},s_{m})>0$ on $(a,b)$ when $m$ is sufficiently large$.$

$($ii$)$ If $w(t,p,s)>0$ on $\left[ a,b\right] ,$ by $(\ref{8.10})$, $%
w(t,p_{m},s_{m})>0\ $on $\left[ a,b\right] $ when $m$ is sufficiently large$%
. $ Thus for sufficiently large $m,$ $\lambda _{0}(p_{m},s_{m})=\Lambda (m).$
Therefore, we complete the proof of this part..

In the following steps, we will denote the $n$-th eigenvalue for the
operators $L_{m}$ with one of the boundary conditions (\ref{z21}), (\ref{zee}%
), (\ref{zff}), (\ref{zgg}) by $\lambda _{n}(K,p_{m},s_{m})$, $\lambda
_{n}(\gamma ,K,p_{m},s_{m}),$ $\mu _{n}(K,p_{m},s_{m})$, $\nu
_{n}(K,p_{m},s_{m})$ respectively, $n\in
\mathbb{N}
_{0}$. And the similar denotations apply to the operator $L.$

(B) Next, assume that the self-adjoint boundary condition is the coupled one
(\ref{z21}) or (\ref{zee}), with\textit{\ }$k_{11}>0$\textit{, }$k_{12}\leq
0.$ The inequalities
\begin{eqnarray}
\nu _{0}(K,p_{m},s_{m}) &\leq &\lambda _{0}(K,p_{m},s_{m})<\lambda
_{0}(\gamma ,K,p_{m},s_{m})  \label{se in} \\
&<&\lambda _{0}(-K,p_{m},s_{m})\leq \mu _{0}(K,p_{m},s_{m})  \notag
\end{eqnarray}%
for Sturm-Liouville problems with regular potentials can be found in \cite%
{CC5}$.$ From the\ step (A), we know that $\{\nu
_{0}(K,p_{m},s_{m})\}_{m=1}^{\infty }$ is bounded from below, thus the
lowest eigenvalues of the operators $L_{m}$ with the boundary conditions
considered in this step are uniformly semi-bounded from below.

(C) Finally, we consider the case where the self-adjoint boundary condition
is the coupled one (\ref{z21}) or (\ref{zee}), with $k_{11}\leq 0$\textit{, }%
$k_{12}<0.$

(1) From the paper \cite{XIN1}, it is an obvious fact that
\begin{equation*}
D_{m}(\mu _{0}(K,p_{m},s_{m}))\leq -2.
\end{equation*}%
Then according to Lemma \ref{continuous} and the fact that $\mu
_{0}(K,p_{m},s_{m})\rightarrow \mu _{0}(K,p,s)$ as $m\rightarrow \infty ,$
it follows that
\begin{equation*}
D(\mu _{0}(K,p,s))\leq -2.
\end{equation*}

(2) From Lemma \ref{solutio}, we obtain that as $\lambda \rightarrow -\infty
,$ $\phi _{1}(b,\lambda )$ and $\phi _{1}^{[1]}(b,\lambda )$ approach
infinity. By the Bounded Convergence Theorem and the decreasing property of $%
\alpha $ in $\lambda ,$%
\begin{equation*}
\lim\limits_{\lambda \rightarrow -\infty }\int_{a}^{b}\frac{\mathit{\mathrm{%
sech}}^{2}\mathit{(}\alpha (t,\lambda )))}{p(s)}ds=0.
\end{equation*}%
Then it can be easily seen that among the functions $\phi
_{2}^{[1]}(b,\lambda ),\phi _{2}(b,\lambda ),\phi _{1}(b,\lambda )$ and $%
\phi _{1}^{[1]}(b,\lambda ),$ $\phi _{1}^{[1]}(b,\lambda )$ grows the
fastest and $\phi _{2}(b,\lambda )$ the slowest, as $\lambda \rightarrow
-\infty .$ Thus from the following equality
\begin{equation*}
D(\lambda )=k_{11}\phi _{2}^{[1]}(b,\lambda )-k_{21}\phi _{2}(b,\lambda
)+k_{22}\phi _{1}(b,\lambda )-k_{12}\phi _{1}^{[1]}(b,\lambda ),
\end{equation*}%
It is easy to see that if $k_{11}\leq 0$\textit{, }$k_{12}<0,$ $D(\lambda
)\rightarrow \infty $ as $\lambda \rightarrow -\infty .$

From the above results (1), (2) and Lemma \ref{L1}, there must exist an
eigenvalue $\lambda _{n_{0}}(K,p,s)$ of the operator $L=L(p,s)$ with the
boundary condition (\ref{21}) in which $k_{11}\leq 0$\textit{, }$k_{12}<0$
such that
\begin{equation*}
\lambda _{n_{0}}(K,p,s)<\mu _{0}(K,p,s).
\end{equation*}%
For Sturm-Liouville operators with regular potentials, as was proved in \cite%
{CC5}, the eigenvalue $\lambda _{0}(K,p_{m},s_{m})$ is the only eigenvalue
that satisfies the inequality
\begin{equation*}
\lambda _{0}(K,p_{m},s_{m})<\lambda _{0}(\gamma ,K,p_{m},s_{m})<\lambda
_{0}(-K,p_{m},s_{m})\leq \mu _{0}(K,p_{m},s_{m}),\text{ }m\in
\mathbb{N}
.
\end{equation*}%
Together with the fact $L_{m}\overset{R}{\Longrightarrow }L$ and$\ \mu
_{0}(K,p_{m},s_{m})\rightarrow \mu _{0}(K,p,s),$ as $m\rightarrow \infty ,$
we have
\begin{equation*}
\lambda _{0}(K,p_{m},s_{m})\rightarrow \lambda _{n_{0}}(K,p,s),\text{ as }%
m\rightarrow \infty .
\end{equation*}%
Thus the lowest eigenvalues of the operators $L_{m}$ with the boundary
condition considered in this step are uniformly semi-bounded from below.

(D) If neither Part (A) nor Part (B) applies to $K,$ then either Part (A) or
Part (B) applies to $-K.$
\end{proof}

\begin{lemma}
\label{L13 copy(1)}The eigenvalues of the self-adjoint differential
operators associated with the differential expressions $(\ref{a})$ can be
ordered to form a non-decreasing sequence%
\begin{equation*}
\lambda _{0},\lambda _{1},\lambda _{2},\lambda _{3},\ldots
\end{equation*}%
approaching $+\infty $. $($Note that the eigenvalues are ordered with
multiplicities without distinguishing the algebraic and geometric
multiplicities.$)$
\end{lemma}

\begin{proof}
For the separated boundary conditions, the claim has been proved in Lemma $%
\ref{L4}.$

From Lemma \ref{L-1} and Lemma \ref{L0 copy(1)}, it is a direct result that
the eigenvalues of the self-adjoint differential operator associated with
the differential expression (\ref{a}) are bounded from below. So it remains
to show that the sequence of the eigenvalues approaches $+\infty .$

As is well known in \cite{CC5}, for the Sturm-Liouville problems with
regular potentials, if\textit{\ }$k_{11}>0$ and\textit{\ }$k_{12}\leq 0,$ $($%
or $k_{11}\leq 0$, $k_{12}<0),$ for \textit{\ }$-\pi <\gamma <0$ or $%
0<\gamma <\pi ,$ we have
\begin{eqnarray*}
\mu _{n}(K,p_{m},s_{m}) &\leq &\lambda _{n+1}(-K,p_{m},s_{m})<\lambda
_{n+1}(\gamma ,K,p_{m},s_{m}) \\
&<&\lambda _{n+1}(K,p_{m},s_{m})\leq \mu _{n+1}(K,p_{m},s_{m}),\text{ \ }%
n\in
\mathbb{N}
_{0}.
\end{eqnarray*}%
For $n\in
\mathbb{N}
_{0},$ from Lemma \ref{L-1}, Lemma \ref{L4}, and Lemma \ref{L0 copy(1)}, one
deduces that
\begin{equation*}
\mu _{n}(K,p_{m},s_{m})\rightarrow \mu _{n}(K,p,s),\text{ as }m\rightarrow
\infty .
\end{equation*}
Hence there exists a sub-sequence $\left\{ \lambda
_{n+1}(K,p_{m_{j}},s_{m_{j}})\right\} _{j=1}^{\infty }$ such that as $%
j\rightarrow \infty ,$
\begin{equation*}
\lambda _{n+1}(K,p_{m_{j}},s_{m_{j}})\rightarrow c\text{ }
\end{equation*}%
and%
\begin{equation*}
\mu _{n}(K,p,s)\leq c\leq \mu _{n+1}(K,p,s),
\end{equation*}
where $c$ is a constant$.$ From the fact that $L_{m}\overset{R}{%
\Longrightarrow }L,$ $c$ must be an eigenvalue of the operator $L$ with the
boundary condition (\ref{21}).

So the eigenvalues of the operator $L$ with the boundary condition (\ref{21}%
) form a non-decreasing sequence approaching $+\infty $. The claim for the
operators $L$ with other boundary conditions follows from a similar proof.
\end{proof}

\begin{remark}
Note that from \cite{wede}, it is a fact that a self-adjoint operator is
bounded from below if and only if its spectrum is bounded from below. From
Lemma \ref{L13 copy(1)}, we obtain the self-adjoint operators with
distributional potentials discussed in this paper are bounded from below
which has been proved in \cite{xx14} by J. Eckhardt, F. Gesztesy, R. Nichols
and G. Teschl using a different approach.
\end{remark}

Based on the fact that the geometric and algebraic multiplicity of each
eigenvalue of the Sturm-Liouville problem consisting of $(\ref{fangcheng})$
and $(\ref{25})$ are equal, we will give a lemma on the continuity of the
eigenvalues which will be used to prove the continuous dependence of the $n$%
-th eigenvalue on the coefficients of the differential equation in Section 7.

\begin{lemma}
\label{princeple}Let $O$ be a subset of $\Omega .$ If $\lambda _{0}$ is
uniformly bounded from below on $O,$ $\omega _{0}\in \Omega ,$ then the
restrictions of the $n-$th eigenvalue to $O$ is continuous at $\omega _{0},$
i.e. $\lim\limits_{\Omega \ni \omega \rightarrow \omega _{0}}\lambda
_{n}(\omega )=\lambda _{n}(\omega _{0}),$ $n\in
\mathbb{N}
_{0}.$

\begin{proof}
The proof is similar to the classical Sturm-Liouville problems with regular
potentials, see \cite[Theorem 1.40]{CC5}.
\end{proof}
\end{lemma}

\begin{lemma}
\label{norm resovent}Denote the n-th eigenvalue of $L_{m}$ and $L$ by $%
\lambda _{n}(m)$ and $\lambda _{n}(0),$ respectively, $n\in
\mathbb{N}
_{0}$, then the sequence of the n-th eigenvalues $\lambda _{n}(m)$ of the
operators $L_{m}$ converges to the n-th eigenvalue $\lambda _{n}(0)$ of the
operator $L$, i.e. $\lambda _{n}(m)\rightarrow \lambda _{n}(0),$ as $%
m\rightarrow \infty .$ $($Note that the eigenvalues are ordered with
multiplicities without distinguishing the algebraic and geometric
multiplicities.$)$
\end{lemma}

\begin{proof}
This claim is a direct consequence of Lemma \ref{L-1} and Lemma \ref{L0
copy(1)}.
\end{proof}

\begin{remark}
Note that Lemma \ref{L-1} describes the continuity of the $n-$th eigenvalue
with respect to the coefficients $1/p,$ $q,$ $s$ in the Banach space $L(J,%
\mathbb{R}
)\ $when the eigenvalues are ordered with geometric multiplicity. And Lemma %
\ref{princeple} describes the continuity of the $n-$th eigenvalue with
respect to the coefficients in the equation $(\ref{fangcheng})$ when the
eigenvalues are ordered with algebraic multiplicity. However, based on Lemma %
\ref{L0 copy(1)} and the fact that the geometric and algebraic multiplicity
of each eigenvalue of the Sturm-Liouville problem consisting of $(\ref%
{fangcheng})$ and $(\ref{25})$ are always equal, either Lemma \ref{L-1} or
Lemma \ref{L0 copy(1)} can be used to obtain Lemma \ref{norm resovent}.
\end{remark}

\section{Continuity region of the $n$-th eigenvalue on separated boundary
conditions}

Since for the operators $L_{m}\ $and $L,$ the coefficients $q$ and $r$ in
the expression (\ref{a}) are fixed, for simplicity, we introduce a simpler
space. Let $\dot{\Omega}=\{\omega =(1/p,q,r,s);$ $(\ref{tx})$ holds
and $q,r$ fixed$\}.$ For the topology of $\dot{\Omega}$ we use a metric $d$
defined as follows:

For $\omega =(1/p,q,r,s)\in \dot{\Omega},$ $\omega _{0}=(\frac{1}{p_{0}}%
,q,r,s_{0})\in \dot{\Omega},$ define
\begin{equation*}
d(\omega ,\omega _{0})=\int_{a}^{b}\left( \left\vert \frac{1}{p}-\frac{1}{%
p_{0}}\right\vert +\left\vert s-s_{0}\right\vert \right) .
\end{equation*}

\begin{theorem}
\label{L6}Let $(\alpha ,\beta )\in \lbrack 0,\pi )\times $ $(0,\pi ],$ and $%
\lambda _{n}=\lambda _{n}(S_{\alpha ,\beta })$ be the n-th eigenvalue of the
Sturm-Liouville problems consisting of $(\ref{fangcheng})\ $and $(\ref{11}).$
Let $w_{n}($\textperiodcentered $,\alpha ,\beta )$ be the normalized real
valued eigenfunction associated with $\lambda _{n}(S_{\alpha ,\beta }).$
Then as a function of $(\alpha ,\beta ),$ $\lambda _{n}$ is continuous on $%
[0,\pi )\times $ $(0,\pi ].$ Moreover,

$(a)$ for a fixed $\beta \in (0,\pi ],$ $\lambda _{n}$ is continuously
differentiable in $\alpha $ on $[0,\pi ),$%
\begin{equation}
\lambda _{n}^{\prime }(\alpha )=-(w_{n}^{[1]}(a,\alpha
))^{2}-(w_{n}(a,\alpha ))^{2};  \label{w1}
\end{equation}

$(b)$ for a fixed $\alpha \in \lbrack 0,\pi ),$ $\lambda _{n}$ is
continuously differentiable in $\beta \ $on $(0,\pi ],\ $%
\begin{equation}
\lambda _{n}^{\prime }(\beta )=(w_{n}^{[1]}(b,\beta ))^{2}+(w_{n}(b,\beta
))^{2}.  \label{w2}
\end{equation}
\end{theorem}

\begin{corollary}
\label{L11}Let $(\alpha ,\beta )\in \lbrack 0,\pi )\times $ $(0,\pi ],$ and $%
\lambda _{n}=\lambda _{n}(S_{\alpha ,\beta })$ be the $n-$th eigenvalue of
the Sturm-Liouville problems consisting of $(\ref{fangcheng})\ $and $(\ref%
{11}).$

$(a)$ For a fixed $\beta \in (0,\pi ],$ $\lambda _{n}$ is strictly
decreasing in $\alpha $ on $[0,\pi ),$%
\begin{equation}
\lim_{\alpha \rightarrow \pi ^{-}}\lambda _{0}(S_{\alpha ,\beta })=-\infty
,\lim_{\alpha \rightarrow \pi ^{-}}\lambda _{n}(S_{\alpha ,\beta })=\lambda
_{n-1}(S_{0,\beta }),\text{ for }n\in
\mathbb{N}
.  \label{DD}
\end{equation}

$(b)$ For a fixed $\alpha \in \lbrack 0,\pi ),$ $\lambda _{n}$ is strictly
increasing in $\beta \ $on $(0,\pi ],$
\begin{equation}
\lim_{\beta \rightarrow 0^{+}}\lambda _{0}(S_{\alpha ,\beta })=-\infty
,\lim_{\beta \rightarrow 0^{+}}\lambda _{n}(S_{\alpha ,\beta })=\lambda
_{n-1}(S_{\alpha ,\pi }),\text{ for }n\in
\mathbb{N}
.  \label{CC}
\end{equation}
\end{corollary}

\begin{proof}[Proof of Theorem 3.1]
As is well known in \cite{CC5}, for the classical Sturm-Liouville problems
with regular potentials$,\ $the assertion of the theorem is true. Let $%
L=L(p,s)$ denote the operator generated by (\ref{b}) with the boundary
conditions (\ref{11})$.$ For the case of $p=p_{m},$ $s=s_{m},$ denote the
operators $L_{m}=L(p_{m},s_{m}).$ For the operators $L_{m}$, the
corresponding separated boundary condition (\ref{11}) becomes
\begin{equation}
S_{\alpha ,\beta }:\left\{
\begin{array}{c}
\cos \alpha y(a)-\sin \alpha (p_{m}y^{\prime })(a)=0,\text{ }\alpha \in
\lbrack 0,\pi ), \\
\cos \beta y(b)-\sin \beta (p_{m}y^{\prime })(b)=0,\text{ }\beta \in (0,\pi
].%
\end{array}%
\right.  \label{23}
\end{equation}%
From Lemma \ref{norm resovent}, the $n-$th eigenvalues of the operators $%
\{L_{m}\}_{m\in
\mathbb{N}
}$ converge to the $n-$th eigenvalue of the limit operator $L.$

(i) In this step, we denote $\omega _{m}=(\frac{1}{p_{m}},q,r,s_{m})\in \dot{%
\Omega},$ $\omega _{0}=(\frac{1}{p},q,r,s)\in \dot{\Omega},$ the $n-$th
eigenvalue of $L$ and $L_{m}$ by $\lambda _{n}(\omega _{0},(\alpha ,\beta ))$
and $\lambda _{n}(\omega _{m},(\alpha ,\beta )),$ respectively, and also let
$S=[0,\pi )\times $ $(0,\pi ].$

Suppose $\Lambda $ is a simple continuous eigenvalue branch through the
eigenvalue $\lambda _{n}(\omega _{0},(\alpha _{0},\beta _{0}))$ defined on a
connected neighborhood $O$ of $(\omega _{0},(\alpha _{0},\beta _{0}))$ in $%
\dot{\Omega}\times S.$ Since
\begin{equation*}
\lambda _{n}(\omega _{m},(\alpha _{0},\beta _{0}))\rightarrow \lambda
_{n}(\omega _{0},(\alpha _{0},\beta _{0}))\text{ as }m\rightarrow \infty ,
\end{equation*}
we obtain
\begin{equation*}
\Lambda (\omega _{m},(\alpha _{0},\beta _{0}))=\lambda _{n}(\omega
_{m},(\alpha _{0},\beta _{0}))
\end{equation*}%
when $m$ is sufficiently large.

It is clear that for sufficiently large $m,$ $\Lambda (\omega _{m},(\alpha
,\beta ))$ is a continuous eigenvalue branch through $\lambda _{n}(\omega
_{m},(\alpha _{0},\beta _{0}))$ defined on a neighborhood of $(\alpha
_{0},\beta _{0})$ in $[0,\pi )\times $ $(0,\pi ].$ However, as was proved in
\cite{CC5}, $\lambda _{n}(\omega _{m},(\alpha ,\beta ))$ is the unique
continuous eigenvalue branch through $\lambda _{n}(\omega _{m},(\alpha
_{0},\beta _{0}))$ defined on $[0,\pi )\times $ $(0,\pi ]$, thus
\begin{equation*}
\Lambda (\omega _{m},(\alpha ,\beta ))=\lambda _{n}(\omega _{m},(\alpha
,\beta ))
\end{equation*}%
for sufficiently large $m$ and $\left( \omega _{m},(\alpha ,\beta )\right)
\in O.$

By Lemma \ref{lianxuxing}, for an arbitrary $\epsilon >0,$ there exists a $%
\delta _{1}>0\ $such that if
\begin{equation*}
\left\vert (\alpha ,\beta )-(\alpha _{0},\beta _{0})\right\vert <\delta
_{1}/2,\int_{a}^{b}\left( \left\vert \frac{1}{p_{m}}-\frac{1}{p}\right\vert
+\left\vert s_{m}-s\right\vert \right) <\delta _{1}/2,
\end{equation*}%
then%
\begin{equation*}
\left\vert \lambda _{n}(\omega _{m},(\alpha ,\beta ))-\lambda _{n}(\omega
_{0},(\alpha _{0},\beta _{0}))\right\vert =\left\vert \Lambda (\omega
_{m},(\alpha ,\beta ))-\lambda _{n}(\omega _{0},(\alpha _{0},\beta
_{0}))\right\vert <\epsilon /2.
\end{equation*}%
Moreover$,$ for such a $\epsilon >0,$ and a fixed point $(\alpha ,\beta )\in
\lbrack 0,\pi )\times $ $(0,\pi ],$\ there exists a $\delta _{2}>0\ $such
that if
\begin{equation*}
\int_{a}^{b}\left( \left\vert \frac{1}{p_{m}}-\frac{1}{p}\right\vert
+\left\vert s_{m}-s\right\vert \right) <\delta _{2}/2,
\end{equation*}%
then%
\begin{equation*}
\left\vert \lambda _{n}(\omega _{m},(\alpha ,\beta ))-\lambda _{n}(\omega
_{0},(\alpha ,\beta ))\right\vert <\epsilon /2.
\end{equation*}%
Thus for an arbitrary $\epsilon >0,$ there exists a $\delta =\delta _{1}/2\ $%
such that if
\begin{equation*}
\left\vert (\alpha ,\beta )-(\alpha _{0},\beta _{0})\right\vert <\delta ,
\end{equation*}%
then%
\begin{eqnarray*}
\left\vert \lambda _{n}(\omega _{0},(\alpha ,\beta ))-\lambda _{n}(\omega
_{0},(\alpha _{0},\beta _{0}))\right\vert &\leq &\left\vert \lambda
_{n}(\omega _{0},(\alpha ,\beta ))-\lambda _{n}(\omega _{m},(\alpha ,\beta
))\right\vert \\
&&+\left\vert \lambda _{n}(\omega _{m},(\alpha ,\beta ))-\lambda _{n}(\omega
_{0},(\alpha _{0},\beta _{0}))\right\vert \\
&<&\epsilon .
\end{eqnarray*}%
So it is a direct result that the eigenvalue $\lambda _{n}(S_{\alpha ,\beta
})$ of the problem (\ref{b}), (\ref{11}) is continuous on $[0,\pi )\times $ $%
(0,\pi ]$.

(ii) In this step we will show that for a fixed $\beta ,\ \lambda
_{n}=\lambda _{n}(\alpha )$ is continuously differentiable in $\alpha $ on $%
[0,\pi )$. Also we assume $\alpha \neq \frac{\pi }{2},\ $the proof for the
case $\alpha =\frac{\pi }{2}$ can be completed similarly. For sufficiently
small $h\in
\mathbb{R}
$, denote the normalized real valued eigenfunctions of $\lambda _{n}(\alpha
) $ and $\lambda _{n}(\alpha +h)$ by $w_{n}=w_{n}($\textperiodcentered $%
,\alpha )$ and $v_{n}=w_{n}($\textperiodcentered $,\alpha +h)$,
respectively. According to the Lagrange's formula in \cite{xx14}, it follows
that
\begin{equation*}
(\lambda _{n}(\alpha +h)-\lambda _{n}(\alpha
))\int_{a}^{b}v_{n}w_{n}rdt=-\left. [w_{n},v_{n}]\right\vert _{a}^{b},
\end{equation*}%
where $[w_{n},v_{n}]:=w_{n}v_{n}^{[1]}-v_{n}w_{n}^{[1]}$. From the boundary
condition (\ref{11}), we obtain
\begin{eqnarray*}
(\lambda _{n}(\alpha +h)-\lambda _{n}(\alpha ))\int_{a}^{b}v_{n}w_{n}rdt
&=&w_{n}(a)v_{n}^{[1]}(a)-v_{n}(a)w_{n}^{[1]}(a) \\
&=&(\tan \alpha -\tan (\alpha +h))w_{n}^{[1]}(a)v_{n}^{[1]}(a) \\
&=&-(\tan (\alpha +h)-\tan \alpha )w_{n}^{[1]}(a,\alpha
)w_{n}^{[1]}(a,\alpha +h).
\end{eqnarray*}%
Since $\lambda _{n}(S_{\alpha ,\beta })$ is continuous on $[0,\pi )\times $ $%
(0,\pi ]$ which has been proved in step $($i$),$ then by Lemma \ref%
{eigenfunction}, we obtain
\begin{equation*}
\left\vert w_{n}(x,\alpha )-w_{n}(x,\alpha +h)\right\vert \rightarrow 0,%
\text{ }\left\vert w_{n}^{[1]}(x,\alpha )-w_{n}^{[1]}(x,\alpha
+h)\right\vert \rightarrow 0,\text{ as }h\rightarrow 0,
\end{equation*}%
both uniformly for $x\in \lbrack a,b].$

Thus we get
\begin{equation*}
\lambda _{n}^{\prime }(\alpha )=-\sec ^{2}\alpha (w_{n}^{[1]}(a,\alpha
))^{2}=-(w_{n}^{[1]}(a,\alpha ))^{2}-\tan ^{2}\alpha (w_{n}^{[1]}(a,\alpha
))^{2}<0.
\end{equation*}%
Hence for a fixed $\beta $, $\lambda _{n}(\alpha )$ is differentiable in $%
\alpha $ on $[0,\pi ).\ $

The statement on the differentiability of $\lambda _{n}(\beta )$ can be
proved by the method analogous to that used above.
\end{proof}

\begin{proof}[Proof of Corollary 3.2]
The strict monotonicity of $\lambda _{n}$ is a direct consequence of (\ref%
{w1}) and (\ref{w2}).

As in the proof of Theorem \ref{L6}, we denote the $n-$th eigenvalue of $L$
and $L_{m}$ with the boundary conditions (\ref{11}) and (\ref{23}) by $%
\lambda _{n}(\omega _{0},(\alpha ,\beta ))$ and $\lambda _{n}(\omega
_{m},(\alpha ,\beta )),$ respectively. For a fixed $\beta \in (0,\pi ],$ as
is well known from \cite{CC5},
\begin{equation*}
\inf\limits_{\alpha \in \lbrack 0,\pi )}\lambda _{n}(\omega _{m},(\alpha
,\beta ))=\lambda _{n-1}(\omega _{m},(0,\beta )),\text{ }n\in
\mathbb{N}
.
\end{equation*}%
Since%
\begin{equation*}
\lim\limits_{m\rightarrow \infty }\lambda _{n}(\omega _{m},(\alpha ,\beta
))=\lambda _{n}(\omega _{0},(\alpha ,\beta )),\text{ }n\in
\mathbb{N}
_{0},
\end{equation*}%
so we can easily obtain for $\alpha \in \lbrack 0,\pi ),$
\begin{equation*}
\lambda _{n}(\omega _{0},(\alpha ,\beta ))\geq \lambda _{n-1}(\omega
_{0},(0,\beta )),\text{ }n\in
\mathbb{N}
.
\end{equation*}%
As we have known that for a fixed $\beta \in (0,\pi ],$ $\lambda _{n}(\omega
_{0},(\alpha ,\beta ))$ is strictly decreasing in $\alpha $ on $[0,\pi ),$ so%
$\lim\limits_{\alpha \rightarrow \pi ^{-}}\lambda _{n}(\omega _{0},(\alpha
,\beta ))$ exists, and is equal to an eigenvalue $\lambda (\omega
_{0},(0,\beta ))$, thus
\begin{equation*}
\lambda _{n-1}(\omega _{0},(0,\beta ))\leq \lambda (\omega _{0},(0,\beta
))<\lambda _{n}(\omega _{0},(0,\beta )),
\end{equation*}%
so
\begin{equation*}
\lim_{\alpha \rightarrow \pi ^{-}}\lambda _{n}(\omega _{0},(\alpha ,\beta
))=\lambda _{n-1}(\omega _{0},(0,\beta )).
\end{equation*}%
It is obvious that%
\begin{equation*}
\lim_{m\rightarrow \infty }\lambda _{0}(\omega _{m},(\alpha ,\beta
))=\lambda _{0}(\omega _{0},(\alpha ,\beta )),
\end{equation*}%
since $\lambda _{0}(\omega _{0},(\alpha ,\beta ))$ is strictly decreasing in
$\alpha $ for a fixed $\beta \in (0,\pi ],$ suppose that
\begin{equation*}
\lim\limits_{\alpha \rightarrow \pi ^{-}}\lambda _{0}(\omega _{0},(\alpha
,\beta ))=\inf\limits_{\alpha \in \lbrack 0,\pi )}\lambda _{0}(\omega
_{0},(\alpha ,\beta ))=c>-\infty ,
\end{equation*}%
thus $c$ must be an eigenvalue $\lambda (\omega _{0},(0,\beta ))$ and $%
\lambda (\omega _{0},(0,\beta ))<\lambda _{0}(\omega _{0},(0,\beta )).$ This
contradiction leads to the conclusion (\ref{DD}). The proof for (\ref{CC})
is similar.$\ $
\end{proof}

\section{Inequalities among eigenvalues}

\begin{theorem}
\label{L7} Let $K\in \mathit{\mathrm{SL}}(2,%
\mathbb{R}
).\ \mathit{(a)}$ \textit{If }$k_{11}>0$\textit{\ and }$k_{12}\leq 0,$ $-\pi
<\gamma <0$\textit{\ or }$0<\gamma <\pi ,$ \textit{then }$\lambda _{0}(K)$%
\textit{\ is simple, and}%
\begin{eqnarray}
\nu _{0} &\leq &\lambda _{0}(K)<\lambda _{0}(\gamma ,K)<\lambda _{0}(-K)\leq
\{\mu _{0},\nu _{1}\}  \label{tt} \\
&\leq &\lambda _{1}(-K)<\lambda _{1}(\gamma ,K)<\lambda _{1}(K)\leq \{\mu
_{1},\nu _{2}\}  \notag \\
&\leq &\lambda _{2}(K)<\lambda _{2}(\gamma ,K)<\lambda _{2}(-K)\leq \{\mu
_{2},\nu _{3}\}  \notag \\
&\leq &\lambda _{3}(-K)<\lambda _{3}(\gamma ,K)<\lambda _{3}(K)\leq \cdots .
\notag
\end{eqnarray}%
$(b)$ If $k_{11}\leq 0$ and $k_{12}<0,$ for $-\pi <\gamma <0$ or $0<\gamma
<\pi ,$ we have%
\begin{eqnarray}
\lambda _{0}(K) &<&\lambda _{0}(\gamma ,K)<\lambda _{0}(-K)\leq \{\mu
_{0},\nu _{0}\}  \label{ww} \\
&\leq &\lambda _{1}(-K)<\lambda _{1}(\gamma ,K)<\lambda _{1}(K)\leq \{\mu
_{1},\nu _{1}\}  \notag \\
&\leq &\lambda _{2}(K)<\lambda _{2}(\gamma ,K)<\lambda _{2}(-K)\leq \{\mu
_{2},\nu _{2}\}  \notag \\
&\leq &\lambda _{3}(-K)<\lambda _{3}(\gamma ,K)<\lambda _{3}(K)\leq \cdots .
\notag
\end{eqnarray}%
$(c)$ For $0<\gamma _{1}<\gamma _{2}<\pi $, we have
\begin{equation*}
\lambda _{0}(\gamma _{1},K)<\lambda _{0}(\gamma _{2},K)<\lambda _{1}(\gamma
_{2},K)<\lambda _{1}(\gamma _{1},K)<\lambda _{2}(\gamma _{1},K)<\lambda
_{2}(\gamma _{2},K)<\cdots .
\end{equation*}%
$(d)$ If $K$ is not included in the case $(a)$ and $(b)$, then $-K$ is
included in either case $(a)$ or case $(b)$.
\end{theorem}

\begin{theorem}
\label{L8}Recall that $\lambda _{n}^{D}$ is the n-th Dirichlet eigenvalue.
For the n-th eigenvalue $\lambda _{n}(A,B)$ of the Sturm-Liouville problems
consisting of $(\ref{fangcheng})$ and $(\ref{25})$, we obtain the following
conclusions$:$ $(a)$ $\lambda _{0}(A,B)$ $\leq \lambda _{0}^{D}$, $\lambda
_{1}(A,B)\leq \lambda _{1}^{D}$, $(b)$ $\lambda _{n-2}^{D}<\lambda _{n}(A,B)$
$\leq \lambda _{n}^{D},$ for $n\geq 2.\ $
\end{theorem}

\begin{proof}[Proof of Theorem 4.1]
(i) Let $L=L(p,s)$ denote the operator generated by (\ref{b}) with one of
the boundary conditions (\ref{21}), (\ref{ee}), (\ref{ff}), (\ref{gg})$.$
For the case of $p=p_{m},$ $s=s_{m},$ denote the operators $%
L_{m}=L(p_{m},s_{m}).$ For the case of $L_{m},$ since $s_{m}(a)=s_{m}(b)=0,$
so$\ $the replacements of the boundary conditions (\ref{21}), (\ref{ee}), (%
\ref{ff}), (\ref{gg}) are (\ref{z21}), (\ref{zee}), (\ref{zff}), (\ref{zgg}).

For the operators $L_{m},\ $the inequalities among the eigenvalues $\lambda
_{n}(K),$ $\lambda _{n}(\gamma ,K),\ \mu _{n},$ $\nu _{n}$ can be found in
\cite{XIN1}$.$ From Lemma \ref{norm resovent}, the $n-$th eigenvalues of the
operators $\{L_{m}\}_{m\in
\mathbb{N}
}$ converge to the $n-$th eigenvalue of the limit operator $L.$ Lemma \ref%
{L1} implies that $D(\lambda _{n}(\pm K))=\pm 2$, $\left\vert D(\lambda
_{n}(\gamma ,K))\right\vert <2,\ $thus the inequalities in $(a)$\ and $(b)$
can be obtained.

(ii) Now we will show the monotonicity of $\lambda _{n}(\gamma ,K),$ for $%
0<\gamma _{1}<\gamma _{2}<\pi .$ It suffices to show that $D^{\prime
}(\lambda )$ is not zero at values of $\lambda $ such that $\left\vert
D(\lambda )\right\vert <2.\ $

Let $\Phi _{\lambda }\left( x,\lambda \right) =(d/d\lambda )\Phi \left(
x,\lambda \right) ,\ $it follows from (\ref{3})$,$ (\ref{4}) that%
\begin{equation*}
\Phi _{\lambda }^{\prime }=(P-\lambda W)\Phi _{\lambda }-W\Phi ,\Phi
_{\lambda }\left( a,\lambda \right) =0.
\end{equation*}%
By the variation of parameters formula, we have%
\begin{equation*}
\Phi _{\lambda }\left( x,\lambda \right) =-\int\nolimits_{a}^{x}\Phi \left(
x,\lambda \right) \Phi ^{-1}\left( s,\lambda \right) W(s)\Phi \left(
s,\lambda \right) ds.
\end{equation*}%
So by (\ref{GP}), we can obtain that
\begin{eqnarray}
D^{\prime }(\lambda ) &=&\mathrm{trace}K^{-1}\Phi _{\lambda }\left(
b,\lambda \right)  \notag \\
&=&-\mathrm{trace}\int\nolimits_{a}^{b}K^{-1}\Phi \left( b,\lambda \right)
\Phi ^{-1}\left( s,\lambda \right) W(s)\Phi \left( s,\lambda \right) ds
\notag \\
&=&-\mathrm{trace}\int\nolimits_{a}^{b}\left(
\begin{array}{cc}
-D_{2}(\lambda )\phi _{1}\phi _{2}+C(\lambda )\phi _{1}^{\text{ }2} & \ast
\\
\ast & -A(\lambda )\phi _{2}^{\text{ }2}+D_{1}(\lambda )\phi _{1}\phi _{2}%
\end{array}%
\right) ds  \notag \\
&=&\int\nolimits_{a}^{b}[A(\lambda )\phi _{2}^{\text{ }2}(s,\lambda
)-B(\lambda )\phi _{1}(s,\lambda )\phi _{2}(s,\lambda )-C(\lambda )\phi
_{1}^{\text{ }2}(s,\lambda )]ds.  \label{15}
\end{eqnarray}%
Since
\begin{equation*}
D_{1}(\lambda )D_{2}(\lambda )-A(\lambda )C(\lambda )=\det K^{-1}\Phi \left(
b,\lambda \right) =1,
\end{equation*}%
we have
\begin{eqnarray*}
4-D^{2}(\lambda ) &=&4-(D_{1}(\lambda )+D_{2}(\lambda ))^{2} \\
&=&4(1-D_{1}(\lambda )D_{2}(\lambda ))-B^{2}(\lambda )=-(4A(\lambda
)C(\lambda )+B^{2}(\lambda )).
\end{eqnarray*}%
Hence from (\ref{15}) it follows that
\begin{eqnarray}
4C(\lambda )D^{\prime }(\lambda ) &=&\int\nolimits_{a}^{b}[4A(\lambda
)C(\lambda )\phi _{2}^{\text{ }2}-4B(\lambda )C(\lambda )\phi _{1}\phi
_{2}-4C^{2}(\lambda )\phi _{1}^{\text{ }2}]ds  \notag \\
&=&\int\nolimits_{a}^{b}[-(2C(\lambda )\phi _{1}+B(\lambda )\phi
_{2})^{2}+(4A(\lambda )C(\lambda )+B^{2}(\lambda ))\phi _{2}^{\text{ }2}]ds
\notag \\
&=&-\int\nolimits_{a}^{b}(2C(\lambda )\phi _{1}+B(\lambda )\phi
_{2})^{2}ds-(4-D^{2}(\lambda ))\int\nolimits_{a}^{b}\phi _{2}^{\text{ }2}ds.
\label{rr}
\end{eqnarray}%
Thus if $\left\vert D(\lambda )\right\vert <2$, $D^{\prime }(\lambda )\neq
0.\ $
\end{proof}

\begin{proof}[Proof of Theorem 4.2]
First we prove the inequality for separated boundary conditions, and denote
the $n-$eigenvalue for some $S_{\alpha ,\beta }$ by $\lambda _{n},$ suppose $%
\lambda _{n}^{D}<\lambda _{n}$ for some $n\in
\mathbb{N}
_{0}$. Denote the eigenfunctions of $\lambda _{n}^{D}$ and $\lambda _{n}$ by
$\Psi ($\textperiodcentered ,$\lambda _{n}^{D})$ and $\Psi ($%
\textperiodcentered ,$\lambda _{n}),\ $respectively. From Lemma \ref{L4} and
Lemma \ref{L5}, we obtain that $\Psi ($\textperiodcentered ,$\lambda _{n})$
has at least $n+1$ zeros on the interval $(a,b)$, this contradicts the
conclusion of Lemma \ref{L4}. Then the general inequality
\begin{equation}
\lambda _{n}(A,B)\leq \lambda _{n}^{D}  \label{oo}
\end{equation}%
follows from (\ref{tt}), (\ref{ww}). Of course, the general inequality $%
\lambda _{n}(A,B)$ $\leq \lambda _{n}^{D}$ can also be obtained by using the
fact that the $n-$th eigenvalues of the operators $\{L_{m}\}_{m\in
\mathbb{N}
}$ converge to the $n-$th eigenvalue of the limit operator $L.$ For the
operators $L_{m},\ $this inequality has been proved in \cite{CC5}.

Now we denote the set of all the separated boundary conditions $S_{\alpha
,\beta }$ by $\Gamma .$

By Corollary \ref{L11}, for an arbitrary $\beta \in (0,\pi ],$ we obtain that%
$\ $%
\begin{equation}
\lambda _{n}(S_{0,\beta })>\lim\limits_{\gamma \rightarrow 0^{+}}\lambda
_{n}(S_{0,\gamma })=\lambda _{n-1}(S_{0,\pi })=\lambda _{n-1}^{D},n\in
\mathbb{N}
,  \label{qq}
\end{equation}%
\begin{equation}
\inf\limits_{C\in \Gamma }\lambda _{n}(C)=\inf\limits_{0\leq \alpha <\pi
}(\inf\limits_{0<\beta \leq \pi }\lambda _{n}(S_{\alpha ,\beta
}))=\inf\limits_{0\leq \alpha <\pi }\lambda _{n-1}(S_{\alpha ,\pi })=\lambda
_{n-2}^{D},n\geq 2.  \label{mm}
\end{equation}%
It can be obtained that the infimum in (\ref{mm}) can not be achieved by
using Corollary \ref{L11}.

By Theorem \ref{L7}, for any $K\in \mathit{\mathrm{SL}}(2,%
\mathbb{R}
)$, there exists $0<\beta \leq \pi $, such that
\begin{equation}
\lambda _{n}(K)\leq \lambda _{n}(S_{0,\beta })\leq \lambda _{n+1}(K),n\in
\mathbb{N}
_{0},  \label{ss}
\end{equation}%
and for $0<\gamma <\pi $ or $-\pi <\gamma <0$,
\begin{equation}
\lambda _{n}(K)<\lambda _{n}(\gamma ,K)<\lambda _{n}(-K)\text{ or }\lambda
_{n}(-K)<\lambda _{n}(\gamma ,K)<\lambda _{n}(K),n\in
\mathbb{N}
_{0},  \label{aw}
\end{equation}%
thus by (\ref{oo})$-$(\ref{aw}), the proof is completed.
\end{proof}

\section{Discontinuity of $\protect\lambda _{n}$ on the space of
self-adjoint boundary conditions}

In this section, we describe the continuity region of the $n$-th eigenvalue
as a function on the space of self-adjoint boundary conditions. As a similar
work on the classical Sturm-Liouville problems with regular potentials,
following \cite{8}, we give some notations and results on the space of
self-adjoint boundary conditions.

$\mathrm{M}_{2\times 4}^{\ast }\left(
\mathbb{C}
\right) $ stands for the set of $2$ by $4$ matrices over $%
\mathbb{C}
$ with rank $2$ and let $\mathit{\mathrm{GL}}(2,%
\mathbb{C}
)$ be the Lie group of invertible complex matrices in dimension $2.$ As
mentioned in the introduction, a complex boundary condition (not necessarily
self-adjoint) is just a system of two linearly independent homogeneous
equations on $y(a)$, $y^{[1]}(a)$, $y(b)$ and $y^{[1]}(b)$ with complex
coefficients, i.e.%
\begin{equation}
A\left(
\begin{array}{c}
y(a) \\
y^{[1]}(a)%
\end{array}%
\right) +B\left(
\begin{array}{c}
y(b) \\
y^{[1]}(b)%
\end{array}%
\right) =\left(
\begin{array}{c}
0 \\
0%
\end{array}%
\right) ,  \label{BC}
\end{equation}%
with the $2\times 4$ matrix $(A|B)\in $ $\mathrm{M}_{2\times 4}^{\ast
}\left(
\mathbb{C}
\right) .$

Following \cite{8}, we will take the quotient space

\begin{equation*}
\mathit{\mathrm{GL}}(2,%
\mathbb{C}
)\backslash \mathrm{M}_{2\times 4}^{\ast }\left(
\mathbb{C}
\right)
\end{equation*}%
as the space $\mathscr{B}^{%
\mathbb{C}
}$ of complex boundary conditions, i.e., each boundary condition is an
equivalence class of coefficient matrices $($with the elements of $\mathit{%
\mathrm{GL}}(2,%
\mathbb{C}
)$ multiplying from the left$)$ of linear systems $(\ref{BC})$, and the
boundary condition represented by the linear system $(\ref{BC})$ will be
denoted by $\left[ A\left\vert B\right. \right] .$ Note here, that square
brackets, not parentheses, are used. Similarly, the space $\mathscr{B}^{%
\mathbb{R}
}$ of real boundary conditions is just $\mathit{\mathrm{GL}}(2,%
\mathbb{R}
)\backslash \mathrm{M}_{2\times 4}^{\ast }\left(
\mathbb{R}
\right) $. Note that the space $\mathscr{B}_{S}^{%
\mathbb{R}
}$ of self-adjoint real boundary conditions consists of the separated real
boundary conditions and the coupled real boundary conditions of the form $%
\left[ K\left\vert -I\right. \right] $ where $K\in \mathit{\mathrm{SL}}(2,%
\mathbb{R}
).$ Denote all the self-adjoint complex boundary conditions by $\mathscr{B}%
_{S}^{%
\mathbb{C}
}.$ In this section, we characterize the discontinuity set of $\lambda _{n}$
as a function on $\mathscr{B}_{S}^{%
\mathbb{R}
}$ or $\mathscr{B}_{S}^{%
\mathbb{C}
}$ and determine the behavior of $\lambda _{n}$ near each discontinuity
point.

As a similar work on the classical Sturm-Liouville problems with regular
potentials, the space $\mathscr{B}_{S}^{%
\mathbb{R}
}$ can be obtained by \textquotedblleft gluing\textquotedblright\ the open
sets
\begin{eqnarray*}
\mathscr{O}_{1,S}^{%
\mathbb{R}
} &=&\mathscr{O}_{6,s}^{%
\mathbb{R}
}=\left\{ \left[ K\left\vert -I\right. \right] ;K\in \mathit{\mathrm{SL}}(2,%
\mathbb{R}
)\right\} , \\
\mathscr{O}_{2,S}^{%
\mathbb{R}
} &=&\left\{ \left[
\begin{array}{cccc}
1 & a_{12} & 0 & a_{22} \\
0 & a_{22} & -1 & b_{22}%
\end{array}%
\right] ;a_{12},a_{22},b_{22}\in
\mathbb{R}
\right\} , \\
\mathscr{O}_{3,S}^{%
\mathbb{R}
} &=&\left\{ \left[
\begin{array}{cccc}
1 & a_{12} & -a_{22} & 0 \\
0 & a_{22} & b_{21} & -1%
\end{array}%
\right] ;a_{12},a_{22},b_{21}\in
\mathbb{R}
\right\} , \\
\mathscr{O}_{4,S}^{%
\mathbb{R}
} &=&\left\{ \left[
\begin{array}{cccc}
a_{11} & 1 & 0 & -a_{21} \\
a_{21} & 0 & -1 & b_{22}%
\end{array}%
\right] ;a_{11},a_{21},b_{22}\in
\mathbb{R}
\right\} , \\
\mathscr{O}_{5,S}^{%
\mathbb{R}
} &=&\left\{ \left[
\begin{array}{cccc}
a_{11} & 1 & a_{21} & 0 \\
a_{21} & 0 & b_{21} & -1%
\end{array}%
\right] ;a_{11},a_{21},b_{21}\in
\mathbb{R}
\right\}
\end{eqnarray*}%
via the coordinate transformations among these open sets. Also, the space $%
\mathscr{B}_{S}^{%
\mathbb{C}
}$ can be obtained by \textquotedblleft gluing\textquotedblright\ the open
sets%
\begin{eqnarray*}
\mathscr{O}_{1,S}^{%
\mathbb{C}
} &=&\mathscr{O}_{6,S}^{%
\mathbb{C}
}=\left\{ \left[ e^{i\theta }K\left\vert -I\right. \right] ;\theta \in
\lbrack 0,\pi ),K\in \mathit{\mathrm{SL}}(2,%
\mathbb{R}
)\right\} , \\
\mathscr{O}_{2,S}^{%
\mathbb{C}
} &=&\left\{ \left[
\begin{array}{cccc}
1 & a_{12} & 0 & \overline{\bar{z}} \\
0 & z & -1 & b_{22}%
\end{array}%
\right] ;a_{12},b_{22}\in
\mathbb{R}
,z\in
\mathbb{C}
\right\} , \\
\mathscr{O}_{3,S}^{%
\mathbb{C}
} &=&\left\{ \left[
\begin{array}{cccc}
1 & a_{12} & -\overline{\bar{z}} & 0 \\
0 & z & b_{21} & -1%
\end{array}%
\right] ;a_{12},b_{21}\in
\mathbb{R}
,z\in
\mathbb{C}
\right\} , \\
\mathscr{O}_{4,S}^{%
\mathbb{C}
} &=&\left\{ \left[
\begin{array}{cccc}
a_{11} & 1 & 0 & -\overline{\bar{z}} \\
z & 0 & -1 & b_{22}%
\end{array}%
\right] ;a_{11},b_{22}\in
\mathbb{R}
,z\in
\mathbb{C}
\right\} , \\
\mathscr{O}_{5,S}^{%
\mathbb{C}
} &=&\left\{ \left[
\begin{array}{cccc}
a_{11} & 1 & \overline{\bar{z}} & 0 \\
z & 0 & b_{21} & -1%
\end{array}%
\right] ;a_{11},b_{21}\in
\mathbb{R}
,z\in
\mathbb{C}
\right\}
\end{eqnarray*}%
via the coordinate transformations among these open sets. Details of the
above results have been described in \cite{8}.

The following are some continuity results about $\lambda _{n}$ on $%
\mathscr{B}_{S}^{%
\mathbb{R}
}.$ In this context, we will use the notation

\begin{eqnarray*}
\mathscr{F}_{-}^{%
\mathbb{R}
} &=&\left\{ \left[ K\left\vert -I\right. \right] ;K\in \mathit{\mathrm{SL}}%
(2,%
\mathbb{R}
),k_{11}k_{12}\leq 0\right\} , \\
\mathscr{G}_{-}^{%
\mathbb{R}
} &=&\left\{ \left[
\begin{array}{cccc}
a_{1} & 1 & 0 & -r \\
r & 0 & -1 & b_{2}%
\end{array}%
\right] ;b_{2}\leq 0,\text{ }a_{1},r\in
\mathbb{R}
\right\} , \\
\mathscr{H}_{-}^{%
\mathbb{R}
} &=&\left\{ \left[
\begin{array}{cccc}
1 & a_{2} & -r & 0 \\
0 & r & b_{1} & -1%
\end{array}%
\right] ;a_{2}\leq 0,\text{ }b_{1},r\in
\mathbb{R}
\right\} , \\
\mathscr{I}_{-}^{%
\mathbb{R}
} &=&\left\{ \left[
\begin{array}{cccc}
1 & a_{2} & 0 & r \\
0 & r & -1 & b_{2}%
\end{array}%
\right] ;a_{2},b_{2}\leq 0,\text{ }r\in
\mathbb{R}
,a_{2}b_{2}\geq r^{2}\right\} , \\
\mathscr{K}^{%
\mathbb{R}
} &=&\left\{ \left[ K\left\vert -I\right. \right] ;K\in \mathit{\mathrm{SL}}%
(2,%
\mathbb{R}
),k_{12}=0\right\} \\
&&\cup \left\{ \left[
\begin{array}{cccc}
a_{1} & a_{2} & 0 & 0 \\
0 & 0 & b_{1} & b_{2}%
\end{array}%
\right] \in \mathscr{B}_{s}^{%
\mathbb{R}
};a_{2}b_{2}=0\right\} .
\end{eqnarray*}%
Note that in this section we will still use the space $\dot{\Omega}$
introduced in Section 3.

\begin{proposition}
\label{3.10}Let $n\in
\mathbb{N}
_{0}.$ Then as a function on the space $\mathscr{B}_{s}^{%
\mathbb{R}
},$ $\lambda _{n}$ is continuous at each point not in $\mathscr{K}^{%
\mathbb{R}
}.$
\end{proposition}

\begin{proof}
Denote $\omega _{m}=(\frac{1}{p_{m}},q,r,s_{m})\in \dot{\Omega},$ $\omega
_{0}=(\frac{1}{p},q,r,s)\in \dot{\Omega}.$ For every $\mathbf{A\in }%
\mathscr{B}_{s}^{%
\mathbb{R}
},$ the eigenvalues $\lambda _{n}(\omega _{0},\mathbf{A})$ and $\lambda
_{n}(\omega _{m},\mathbf{A})$ of the operators $L$ and $L_{m}$ are well
defined, respectively.

Let us consider $\mathbf{A}_{0}\mathbf{\in }\mathscr{B}_{s}^{%
\mathbb{R}
}\backslash \mathscr{K}^{%
\mathbb{R}
},$ if the eigenvalue $\lambda _{n}(\omega _{0},\mathbf{A}_{0})$ is double,
assume that $\lambda _{n}(\omega _{0},\mathbf{A}_{0})=\lambda _{n+1}(\omega
_{0},\mathbf{A}_{0}).$ Suppose $\Lambda _{n}$ and $\Lambda _{n+1}$ are the
continuous eigenvalue branches through the eigenvalue $\lambda _{n}(\omega
_{0},\mathbf{A}_{0})$ defined on a sufficiently small connected neighborhood
$O$ of $(\omega _{0},\mathbf{A}_{0})$ in $\dot{\Omega}\times \mathscr{B}%
_{s}^{%
\mathbb{R}
}$ and $\Lambda _{n}(\omega ,\mathbf{A})\leq \Lambda _{n+1}(\omega ,\mathbf{A%
})$ for any element $(\omega ,\mathbf{A})\in O.$ Note that the number of the
eigenvalue branches is counted with its multiplicities. It can be seen from
\cite[Proposition 3.10]{CC5} that when the neighborhood $O$ is sufficiently
small, $\mathbf{A}$ is in $\mathscr{B}_{s}^{%
\mathbb{R}
}\backslash \mathscr{K}^{%
\mathbb{R}
}$ for every $(\omega ,\mathbf{A})\in O.$

By Lemma \ref{norm resovent},
\begin{equation*}
\lambda _{n}(\omega _{m},\mathbf{A}_{0})\rightarrow \lambda _{n}(\omega _{0},%
\mathbf{A}_{0}),\lambda _{n+1}(\omega _{m},\mathbf{A}_{0})\rightarrow
\lambda _{n+1}(\omega _{0},\mathbf{A}_{0}),\text{ as }m\rightarrow \infty .
\end{equation*}%
Hence the multiplicity of $\lambda _{n}(\omega _{0},\mathbf{A}_{0})$ implies
that
\begin{equation*}
\Lambda _{n}(\omega _{m},\mathbf{A}_{0})=\lambda _{n}(\omega _{m},\mathbf{A}%
_{0}),\Lambda _{n+1}(\omega _{m},\mathbf{A}_{0})=\lambda _{n+1}(\omega _{m},%
\mathbf{A}_{0})
\end{equation*}%
when $m$ is sufficiently large. It is clear that\ for sufficiently large $m,$
$\Lambda _{n}(\omega _{m},\mathbf{A})$ and $\Lambda _{n+1}(\omega _{m},%
\mathbf{A})$ are continuous eigenvalue branches through $\lambda _{n}(\omega
_{m},\mathbf{A}_{0})$ and $\lambda _{n+1}(\omega _{m},\mathbf{A}_{0})$
defined on $O$ respectively$.$ Thus, from the multiplicity of $\Lambda _{n}$
and $\Lambda _{n+1},$ the continuity of $\lambda _{n}(\omega _{m},\mathbf{A}%
) $ and $\lambda _{n+1}(\omega _{m},\mathbf{A})$ on $\mathscr{B}_{s}^{%
\mathbb{R}
}\backslash \mathscr{K}^{%
\mathbb{R}
}$ as was proved in \cite{CC5}, one deduces that%
\begin{equation*}
\Lambda _{n}(\omega _{m},\mathbf{A})=\lambda _{n}(\omega _{m},\mathbf{A}%
),\Lambda _{n+1}(\omega _{m},\mathbf{A})=\lambda _{n+1}(\omega _{m},\mathbf{A%
})
\end{equation*}%
for sufficiently large $m$ and $\left( \omega _{m},\mathbf{A}\right) \in O.$

For an arbitrary $\epsilon >0,$ there exists a $\delta _{1}>0\ $such that if
\begin{equation*}
\left\Vert \mathbf{A-A}_{0}\right\Vert <\delta _{1}/2,\int_{a}^{b}\left(
\left\vert \frac{1}{p_{m}}-\frac{1}{p}\right\vert +\left\vert
s_{m}-s\right\vert \right) <\delta _{1}/2,
\end{equation*}%
then%
\begin{equation*}
\left\vert \lambda _{n}(\omega _{m},\mathbf{A})-\lambda _{n}(\omega _{0},%
\mathbf{A}_{0})\right\vert =\left\vert \Lambda _{n}(\omega _{m},\mathbf{A}%
)-\lambda _{n}(\omega _{0},\mathbf{A}_{0})\right\vert <\epsilon /2.
\end{equation*}%
Moreover, for such a $\epsilon >0\ $and a fixed point $\mathbf{A\in }%
\mathscr{B}_{s}^{%
\mathbb{R}
}$, there exists a $\delta _{2}>0\ $such that if
\begin{equation*}
\int_{a}^{b}\left( \left\vert \frac{1}{p_{m}}-\frac{1}{p}\right\vert
+\left\vert s_{m}-s\right\vert \right) <\delta _{2}/2,
\end{equation*}%
then%
\begin{equation*}
\left\vert \lambda _{n}(\omega _{m},\mathbf{A})-\lambda _{n}(\omega _{0},%
\mathbf{A})\right\vert <\epsilon /2.
\end{equation*}%
Thus for an arbitrary $\epsilon >0,$ there exists a $\delta =\delta _{1}/2$
such that if
\begin{equation*}
\left\Vert \mathbf{A-A}_{0}\right\Vert <\delta ,
\end{equation*}%
then%
\begin{eqnarray}
\left\vert \lambda _{n}(\omega _{0},\mathbf{A})-\lambda _{n}(\omega _{0},%
\mathbf{A}_{0})\right\vert &\leq &\left\vert \lambda _{n}(\omega _{0},%
\mathbf{A})-\lambda _{n}(\omega _{m},\mathbf{A})\right\vert \\
&&+\left\vert \lambda _{n}(\omega _{m},\mathbf{A})-\lambda _{n}(\omega _{0},%
\mathbf{A}_{0})\right\vert  \notag \\
&<&\epsilon .
\end{eqnarray}%
So it is a direct result that the eigenvalue $\lambda _{n}(\omega _{0},%
\mathbf{A})$ is continuous at each point not in $\mathscr{K}^{%
\mathbb{R}
}$.

If the eigenvalue $\lambda _{n}(\omega _{0},\mathbf{A}_{0})$ is simple, the
proof is similar and simpler.
\end{proof}

\begin{proposition}
\label{3.18}For every $n\in
\mathbb{N}
_{0},$ the restriction of $\lambda _{n}$ to each of $\mathscr{F}_{-}^{%
\mathbb{R}
},\mathscr{G}_{-}^{%
\mathbb{R}
},\mathscr{H}_{-}^{%
\mathbb{R}
}$ and $\mathscr{I}_{-}^{%
\mathbb{R}
}$ is continuous.

\begin{proof}
The proof is similar to that of Proposition \ref{3.10}.
\end{proof}
\end{proposition}

Note that the continuity claim in Theorem \ref{L6} for the eigenvalues of
the separated boundary conditions is a consequence of Propositions \ref{3.10}
and \ref{3.18} .

In order to describe the discontinuity of $\lambda _{n}$ on $\mathscr{B}%
_{s}^{%
\mathbb{R}
},$ we let
\begin{eqnarray*}
\mathscr{F}_{+}^{%
\mathbb{R}
} &=&\mathscr{O}_{6,S}^{%
\mathbb{R}
}\backslash \mathscr{F}_{-}^{%
\mathbb{R}
},\text{ \ }\mathscr{G}_{+}^{%
\mathbb{R}
}=\mathscr{O}_{4,S}^{%
\mathbb{R}
}\backslash \mathscr{G}_{-}^{%
\mathbb{R}
},\text{ }\mathscr{H}_{+}^{%
\mathbb{R}
}=\mathscr{O}_{3,S}^{%
\mathbb{R}
}\backslash \mathscr{H}_{-}^{%
\mathbb{R}
}, \\
\mathscr{I}_{+}^{%
\mathbb{R}
} &=&\left\{ \left[
\begin{array}{cccc}
1 & a_{2} & 0 & r \\
0 & r & -1 & b_{2}%
\end{array}%
\right] ;a_{2},b_{2}>0,\text{ }r\in
\mathbb{R}
,a_{2}b_{2}>r^{2}\right\} , \\
\mathscr{I}_{0}^{%
\mathbb{R}
} &=&\mathscr{O}_{2,S}^{%
\mathbb{R}
}\backslash (\mathscr{I}_{-}^{%
\mathbb{R}
}\cup \mathscr{I}_{+}^{%
\mathbb{R}
}).
\end{eqnarray*}%
Note that the coupled boundary conditions in $\mathscr{K}^{%
\mathbb{R}
}$ are all in $\mathscr{F}_{-}^{%
\mathbb{R}
},$ and
\begin{equation*}
\mathscr{K}^{%
\mathbb{R}
}\cap \Gamma =(\mathscr{K}^{%
\mathbb{R}
}\cap \mathscr{G}_{-}^{%
\mathbb{R}
}\cap \Gamma )\cup (\mathscr{K}^{%
\mathbb{R}
}\cap \mathscr{H}_{-}^{%
\mathbb{R}
}\cap \Gamma )\cup (\mathbf{D})
\end{equation*}%
where $\mathbf{D}$ is the Dirichlet boundary condition and $\Gamma $ is the
set of all the separated boundary conditions $S_{\alpha ,\beta }.$

\begin{theorem}
\label{3.18 copy(1)}The function $\lambda _{0}$ on $\mathscr{B}_{S}^{%
\mathbb{R}
}$ is continuous on $\mathscr{B}_{S}^{%
\mathbb{R}
}\backslash \mathscr{K}^{%
\mathbb{R}
}$ and discontinuous at each point of $\mathscr{K}^{%
\mathbb{R}
}$. For $n\in
\mathbb{N}
$, the function $\lambda _{n}$ is continuous on $\mathscr{B}_{S}^{%
\mathbb{R}
}\backslash \mathscr{K}^{%
\mathbb{R}
}$ and at each coupled boundary condition in $\mathscr{K}^{%
\mathbb{R}
}$ where $\lambda _{n}=\lambda _{n-1}$ \ and discontinuous at any other
point of $\mathscr{K}^{%
\mathbb{R}
}$. More precisely, for each coupled \ boundary condition $\mathbf{A}\in %
\mathscr{K}^{%
\mathbb{R}
}$, the restriction of $\lambda _{n}$ to $\mathscr{F}_{-}^{%
\mathbb{R}
}$ is continuous at $\mathbf{A}$ \ for $n\in
\mathbb{N}
_{0}$ and%
\begin{equation}
\lim\limits_{\mathscr{F}_{+}^{%
\mathbb{R}
}\ni \mathbf{B}\rightarrow \mathbf{A}}\lambda _{0}(\mathbf{B})=-\infty ,%
\text{ \ \ \ \ \ \ \ }\lim\limits_{\mathscr{F}_{+}^{%
\mathbb{R}
}\ni \mathbf{B}\rightarrow \mathbf{A}}\lambda _{n}(\mathbf{B})=\lambda
_{n-1}(\mathbf{A})\text{ for }n\in
\mathbb{N}
;\text{ }  \label{3.40}
\end{equation}%
for each $\mathbf{A}\in \mathscr{K}^{%
\mathbb{R}
}\cap \mathscr{G}_{-}^{%
\mathbb{R}
}\cap \Gamma $, the restriction of $\lambda _{n}$ to $\mathscr{G}_{-}^{%
\mathbb{R}
}$ is continuous at $\mathbf{A}$ for $n\in
\mathbb{N}
_{0}$ and%
\begin{equation}
\lim\limits_{\mathscr{G}_{+}^{%
\mathbb{R}
}\ni \mathbf{B}\rightarrow \mathbf{A}}\lambda _{0}(\mathbf{B})=-\infty ,%
\text{ \ \ \ \ \ \ \ }\lim\limits_{\mathscr{G}_{+}^{%
\mathbb{R}
}\ni \mathbf{B}\rightarrow \mathbf{A}}\lambda _{n}(\mathbf{B})=\lambda
_{n-1}(\mathbf{A})\text{ for }n\in
\mathbb{N}
;\text{ }  \label{3.41}
\end{equation}%
for $\mathbf{A}\in \mathscr{K}^{%
\mathbb{R}
}\cap \mathscr{H}_{-}^{%
\mathbb{R}
}\cap \Gamma $, the restriction of $\lambda _{n}$ to $\mathscr{H}_{-}^{%
\mathbb{R}
}$ is continuous at $\mathbf{A}$ for $n\in
\mathbb{N}
_{0}$ and%
\begin{equation}
\lim\limits_{\mathscr{H}_{+}^{%
\mathbb{R}
}\ni \mathbf{B}\rightarrow \mathbf{A}}\lambda _{0}(\mathbf{B})=-\infty ,%
\text{ \ \ \ \ \ \ \ }\lim\limits_{\mathscr{H}_{+}^{%
\mathbb{R}
}\ni \mathbf{B}\rightarrow \mathbf{A}}\lambda _{n}(\mathbf{B})=\lambda
_{n-1}(\mathbf{A})\text{ for }n\in
\mathbb{N}
;\text{ }  \label{3.42}
\end{equation}%
while the restriction $\lambda _{n}$ to $\mathscr{I}_{-}^{%
\mathbb{R}
}$ is continuous at the Dirichlet boundary condition $\mathbf{D}$ for $n\in
\mathbb{N}
_{0}$ and%
\begin{eqnarray}
\lim\limits_{\mathscr{I}_{0}^{%
\mathbb{R}
}\cup \mathscr{I}_{+}^{%
\mathbb{R}
}\ni \mathbf{B}\rightarrow \mathbf{D}}\lambda _{0}(\mathbf{B})
&=&\lim\limits_{\mathscr{I}_{+}^{%
\mathbb{R}
}\ni \mathbf{B}\rightarrow \mathbf{D}}\lambda _{1}(\mathbf{B})=-\infty ,%
\text{ \ \ \ \ \ }  \label{3.43} \\
\text{\ \ }\lim\limits_{\mathscr{I}_{0}^{%
\mathbb{R}
}\ni \mathbf{B}\rightarrow \mathbf{D}}\lambda _{n}(\mathbf{B}) &=&\lambda
_{n-1}(\mathbf{D})\text{ for }n\in
\mathbb{N}
,  \label{3.44} \\
\text{ }\lim\limits_{\mathscr{I}_{+}^{%
\mathbb{R}
}\ni \mathbf{B}\rightarrow \mathbf{D}}\lambda _{n}(\mathbf{B}) &=&\lambda
_{n-2}(\mathbf{D})\text{ for }n\geq 2.  \label{3.45}
\end{eqnarray}
\end{theorem}

\begin{proof}
By Proposition \ref{3.10} and \ref{3.18}, we only need to prove $(\ref{3.40}%
)-(\ref{3.45})$.

Fix a $K\in \mathit{\mathrm{SL}}(2,%
\mathbb{R}
)$ with $k_{11}>0$\textit{\ }and\textit{\ }$k_{12}=0.$ When $\mathbf{L=[}%
L\left\vert -I\right. \mathbf{]\in }\mathscr{F}_{+}^{%
\mathbb{R}
}$ is sufficiently close to $\mathbf{K=[}K\left\vert -I\right. \mathbf{]\in }%
\mathscr{K}^{%
\mathbb{R}
}$, we have $l_{11}>0$ and\textit{\ }$l_{12}>0.$ Part $(a)$ and $(b)$ of
Theorem \ref{L7} implies
\begin{eqnarray}
\lambda _{0}(L) &\leq &\{\mu _{0}(L),\nu _{0}(L)\},  \label{555} \\
\left\{ \mu _{2n}(L),\nu _{2n}(L)\right\} &\leq &\lambda _{2n+1}(L)<\left\{
\mu _{2n+1}(L),\nu _{2n+1}(L)\right\} ,  \notag \\
\left\{ \mu _{2n+1}(L),\nu _{2n+1}(L)\right\} &<&\lambda _{2n+2}(L)\leq
\{\mu _{2n+2}(L),\nu _{2n+2}(L)\},  \notag \\
\nu _{0}(K) &\leq &\lambda _{0}(K)<\{\mu _{0}(K),\nu _{1}(K)\},  \notag \\
\left\{ \mu _{2n}(K),\nu _{2n+1}(K)\right\} &<&\lambda _{2n+1}(K)\leq
\left\{ \mu _{2n+1}(K),\nu _{2n+2}(K)\right\} ,  \notag \\
\left\{ \mu _{2n+1}(K),\nu _{2n+2}(K)\right\} &\leq &\lambda
_{2n+2}(K)<\{\mu _{2n+2}(K),\nu _{2n+3}(K)\},  \notag
\end{eqnarray}%
where $\mu _{0}(L)$ and $\nu _{n}(L)$ are the eigenvalues for the separated
boundary conditions%
\begin{equation*}
\left[
\begin{array}{cccc}
1 & 0 & 0 & 0 \\
0 & 0 & l_{22} & -l_{12}%
\end{array}%
\right] \text{ and }\left[
\begin{array}{cccc}
1 & 0 & 0 & 0 \\
0 & 0 & -l_{21} & l_{11}%
\end{array}%
\right] ,
\end{equation*}%
respectively. By Corollary \ref{L11}, $\mu _{0}(L)\rightarrow -\infty $ and $%
\nu _{n}(L)\rightarrow \nu _{n}(K)$ as $\mathbf{L}$ in $\mathscr{F}_{+}^{%
\mathbb{R}
}$ approaches $\mathbf{K,}$ since then $l_{12}\rightarrow 0^{+},$ $%
l_{22}\rightarrow k_{22}>0$ and $l_{11}\rightarrow k_{11}>0.$ Thus from $(%
\ref{555}),$ it follows that%
\begin{equation*}
\lim\limits_{\mathscr{F}_{+}^{%
\mathbb{R}
}\ni \mathbf{L}\rightarrow \mathbf{K}}\lambda _{0}(\mathbf{L})=-\infty .
\end{equation*}

Assume $\lim\limits_{\mathscr{F}_{+}^{%
\mathbb{R}
}\ni \mathbf{L}\rightarrow \mathbf{K}}\lambda _{n}(\mathbf{L})=\lambda
_{n-1}(\mathbf{K})$ is false for $n\in
\mathbb{N}
$, \ from $(\ref{555}),$ we can find a sequence $\left\{ \mathbf{L}%
_{k}\right\} \subset \mathscr{F}_{+}^{%
\mathbb{R}
}$ such that
\begin{equation*}
\lambda _{n}(\mathbf{L}_{k})\rightarrow c\neq \lambda _{n-1}(\mathbf{K}),%
\text{ as }\mathbf{L}_{k}\rightarrow \mathbf{K,}
\end{equation*}
where $c$ is a constant$.$ From Lemma \ref{L1}, Lemma \ref{continuous}, and $%
(\ref{555})$, one deduces that $c$ is an eigenvalue for $\mathbf{K}$\textbf{%
\ }and $\nu _{n-1}(K)\leq c\leq \nu _{n}(K).$ This contrary implies%
\begin{equation*}
\lim\limits_{\mathscr{F}_{+}^{%
\mathbb{R}
}\ni \mathbf{L}\rightarrow \mathbf{K}}\lambda _{n}(\mathbf{L})=\lambda
_{n-1}(\mathbf{K})\text{ for }n\in
\mathbb{N}
.
\end{equation*}%
Similarly, we prove $(\ref{3.40})$ for $K\in \mathit{\mathrm{SL}}(2,%
\mathbb{R}
)$ with $k_{11}<0$\textit{\ }and\textit{\ }$k_{12}=0.$

Next, we will prove $(\ref{3.41}),$ denote $\omega _{m}=(\frac{1}{p_{m}}%
,q,r,s_{m})\in \dot{\Omega},$ $\omega _{0}=(\frac{1}{p},q,r,s)\in \dot{\Omega%
}.$ For every $\mathbf{A\in }\mathscr{B}_{s}^{%
\mathbb{R}
},$ the eigenvalues $\lambda _{n}(\omega _{0},\mathbf{A})$ and $\lambda
_{n}(\omega _{m},\mathbf{A})$ of $L$ and $L_{m}$ are well defined,
respectively.

Let us consider $\mathbf{A}\in \mathscr{K}^{%
\mathbb{R}
}\cap \mathscr{G}_{-}^{%
\mathbb{R}
}\cap \Gamma ,$ and an arbitrary point $\mathbf{B}\in \mathscr{G}_{+}^{%
\mathbb{R}
},$
\begin{eqnarray}
\left\vert \lambda _{n}(\omega _{0},\mathbf{B})-\lambda _{n-1}(\omega _{0},%
\mathbf{A})\right\vert &\leq &\left\vert \lambda _{n}(\omega _{0},\mathbf{B}%
)-\lambda _{n}(\omega _{m},\mathbf{B})\right\vert  \label{6.2} \\
&&+\left\vert \lambda _{n}(\omega _{m},\mathbf{B})-\lambda _{n-1}(\omega
_{0},\mathbf{A})\right\vert .  \notag
\end{eqnarray}%
By Lemma \ref{L4}, the eigenvalues for $\mathbf{A}$ are all simple$.$
Suppose $\Lambda _{n}$ is the continuous simple eigenvalue branch through
the eigenvalue $\lambda _{n-1}(\omega _{0},\mathbf{A})$ defined on a
connected neighborhood $O$ of $(\omega _{0},\mathbf{A})$ in $\dot{\Omega}%
\times \left\{ \mathscr{G}_{+}^{%
\mathbb{R}
}\cup \left\{ \mathbf{A}\right\} \right\} $.

By Lemma \ref{norm resovent},
\begin{equation}
\lambda _{n-1}(\omega _{m},\mathbf{A})\rightarrow \lambda _{n-1}(\omega _{0},%
\mathbf{A}),\text{ as }m\rightarrow \infty .  \label{1.110}
\end{equation}%
Hence the simplicity of $\Lambda _{n}$ implies that $\Lambda _{n}(\omega
_{m},\mathbf{A})=\lambda _{n-1}(\omega _{m},\mathbf{A})$ when $m$ is
sufficiently large. It is clear that for sufficiently large $m,$ $\Lambda
_{n}(\omega _{m},\mathbf{B})$ is a continuous eigenvalue branch through $%
\lambda _{n-1}(\omega _{m},\mathbf{A})$ defined on $O.$ Thus, from the
simplicity of $\Lambda _{n},$ the continuity of $\lambda _{n}(\omega _{m},%
\mathbf{B})$ on $\mathscr{G}_{+}^{%
\mathbb{R}
}$ and the fact $\lim\limits_{\mathscr{G}_{+}^{%
\mathbb{R}
}\ni \mathbf{B}\rightarrow \mathbf{A}}\lambda _{n}(\omega _{m},\mathbf{B}%
)=\lambda _{n-1}(\omega _{m},\mathbf{A})$ for $n\in
\mathbb{N}
$ as was proved in \cite{CC5}, one deduces that
\begin{equation*}
\Lambda _{n}(\omega _{m},\mathbf{B})=\lambda _{n}(\omega _{m},\mathbf{B})
\end{equation*}
for sufficiently large $m$ and $(\omega _{m},\mathbf{B})\in O.$

For an arbitrary $\epsilon >0,$ there exists a $\delta _{1}>0\ $such that if
\begin{equation*}
\left\Vert \mathbf{B-A}\right\Vert <\delta _{1}/2,\int_{a}^{b}\left(
\left\vert \frac{1}{p_{m}}-\frac{1}{p}\right\vert +\left\vert
s_{m}-s\right\vert \right) <\delta _{1}/2,
\end{equation*}%
then%
\begin{equation*}
\left\vert \lambda _{n}(\omega _{m},\mathbf{B})-\lambda _{n-1}(\omega _{0},%
\mathbf{A})\right\vert =\left\vert \Lambda _{n}(\omega _{m},\mathbf{B}%
)-\lambda _{n-1}(\omega _{0},\mathbf{A})\right\vert <\epsilon /2.
\end{equation*}%
Moreover, for such a $\epsilon >0$ and a fixed point $\mathbf{B}\in %
\mathscr{G}_{+}^{%
\mathbb{R}
},$ there exists a $\delta _{2}>0$ such that if
\begin{equation*}
\int_{a}^{b}\left( \left\vert \frac{1}{p_{m}}-\frac{1}{p}\right\vert
+\left\vert s_{m}-s\right\vert \right) <\delta _{2}/2,
\end{equation*}%
then%
\begin{equation*}
\left\vert \lambda _{n}(\omega _{m},\mathbf{B})-\lambda _{n}(\omega _{0},%
\mathbf{B})\right\vert <\epsilon /2.
\end{equation*}%
Thus for an arbitrary $\epsilon >0,$ there exists a $\delta =\delta _{1}/2\ $%
such that if
\begin{equation*}
\left\Vert \mathbf{B-A}\right\Vert <\delta ,
\end{equation*}%
then%
\begin{eqnarray*}
\left\vert \lambda _{n}(\omega _{0},\mathbf{B})-\lambda _{n-1}(\omega _{0},%
\mathbf{A})\right\vert &\leq &\left\vert \lambda _{n}(\omega _{0},\mathbf{B}%
)-\lambda _{n}(\omega _{m},\mathbf{B})\right\vert \\
&&+\left\vert \lambda _{n}(\omega _{m},\mathbf{B})-\lambda _{n-1}(\omega
_{0},\mathbf{A})\right\vert \\
&<&\epsilon .
\end{eqnarray*}%
So it is a direct result that%
\begin{equation*}
\ \lim\limits_{\mathscr{G}_{+}^{%
\mathbb{R}
}\ni \mathbf{B}\rightarrow \mathbf{A}}\lambda _{n}(\mathbf{B})=\lambda
_{n-1}(\mathbf{A})\text{ for }n\in
\mathbb{N}
.
\end{equation*}%
Assume that $\lim\limits_{\mathscr{G}_{+}^{%
\mathbb{R}
}\ni \mathbf{B}\rightarrow \mathbf{A}}\lambda _{0}(\mathbf{B})=-\infty $ is
false, then there exists a sequence $\{\mathbf{B}_{k}\}\subset \mathscr{G}%
_{+}^{%
\mathbb{R}
}$ such that $\{\lambda _{0}(\mathbf{B}_{k})\}$ is bounded. Without loss of
generality, assume $\lambda _{0}(\mathbf{B}_{k})\rightarrow c$ ($c$ is a
constant) as $\mathbf{B}_{k}\rightarrow \mathbf{A.}$ From Lemma \ref{L1},
Lemma \ref{continuous}, and the fact that $\lambda _{0}(\mathbf{A})$ is a
simple eigenvalue, one deduces that $c$ is an eigenvalue for $\mathbf{A}$%
\textbf{\ }and $c<\lambda _{0}(\mathbf{A}).$ This contrary implies
\begin{equation*}
\lim\limits_{\mathscr{G}_{+}^{%
\mathbb{R}
}\ni \mathbf{B}\rightarrow \mathbf{A}}\lambda _{0}(\mathbf{B})=-\infty .
\end{equation*}%
Similarly, one proves $(\ref{3.42}),(\ref{3.43}),(\ref{3.44}),(\ref{3.45}).$
\end{proof}

In order to describe the discontinuity set of $\lambda _{n}$ as a function
on $\mathscr{B}_{S}^{%
\mathbb{C}
}$, we set%
\begin{eqnarray*}
\mathscr{F}_{-}^{%
\mathbb{C}
} &=&\left\{ \left[ e^{i\theta }K\left\vert -I\right. \right] ;K\in \mathit{%
\mathrm{SL}}(2,%
\mathbb{R}
),k_{11}k_{12}\leq 0\right\} , \\
\mathscr{G}_{-}^{%
\mathbb{C}
} &=&\left\{ \left[
\begin{array}{cccc}
a_{1} & 1 & 0 & -\bar{z} \\
z & 0 & -1 & b_{2}%
\end{array}%
\right] ;b_{2}\leq 0,\text{ }a_{1}\in
\mathbb{R}
,z\in
\mathbb{C}
\right\} , \\
\mathscr{H}_{-}^{%
\mathbb{C}
} &=&\left\{ \left[
\begin{array}{cccc}
1 & a_{2} & -\bar{z} & 0 \\
0 & z & b_{1} & -1%
\end{array}%
\right] ;a_{2}\leq 0,\text{ }b_{1}\in
\mathbb{R}
,z\in
\mathbb{C}
\right\} , \\
\mathscr{I}_{-}^{%
\mathbb{C}
} &=&\left\{ \left[
\begin{array}{cccc}
1 & a_{2} & 0 & \bar{z} \\
0 & z & -1 & b_{2}%
\end{array}%
\right] ;a_{2},b_{2}\leq 0,z\in
\mathbb{C}
,a_{2}b_{2}\geq z\overline{z}\right\} , \\
\mathscr{K}^{%
\mathbb{C}
} &=&\left\{ \left[ e^{i\theta }K\left\vert -I\right. \right] ;K\in \mathit{%
\mathrm{SL}}(2,%
\mathbb{R}
),k_{12}=0,\theta \in \lbrack 0,\pi )\right\} \\
&&\text{\ }\cup \left\{ \left[
\begin{array}{cccc}
a_{1} & a_{2} & 0 & 0 \\
0 & 0 & b_{1} & b_{2}%
\end{array}%
\right] \in \mathscr{B}_{s}^{%
\mathbb{R}
};a_{2}b_{2}=0\right\} . \\
\mathscr{F}_{+}^{%
\mathbb{C}
} &=&\mathscr{O}_{6,S}^{%
\mathbb{C}
}\backslash \mathscr{F}_{-}^{%
\mathbb{C}
},\text{ \ }\mathscr{G}_{+}^{%
\mathbb{C}
}=\mathscr{O}_{4,S}^{%
\mathbb{C}
}\backslash \mathscr{G}_{-}^{%
\mathbb{C}
},\text{ }\mathscr{H}_{+}^{%
\mathbb{C}
}=\mathscr{O}_{3,S}^{%
\mathbb{C}
}\backslash \mathscr{H}_{-}^{%
\mathbb{C}
}, \\
\mathscr{I}_{+}^{%
\mathbb{C}
} &=&\left\{ \left[
\begin{array}{cccc}
1 & a_{2} & 0 & \bar{z} \\
0 & z & -1 & b_{2}%
\end{array}%
\right] ;a_{2},b_{2}>0,\text{ }z\in
\mathbb{C}
,a_{2}b_{2}>z\overline{z}\right\} , \\
\mathscr{I}_{0}^{%
\mathbb{C}
} &=&\mathscr{O}_{2,S}^{%
\mathbb{C}
}\backslash (\mathscr{I}_{-}^{%
\mathbb{C}
}\cup \mathscr{I}_{+}^{%
\mathbb{C}
}).
\end{eqnarray*}

The proofs of the following results are similar to those of Proposition \ref%
{3.10}. \ref{3.18} and Theorem \ref{3.18 copy(1)}, so we omit them.

\begin{proposition}
Let $n\in
\mathbb{N}
_{0}.$Then as a function on the space $\mathscr{B}_{s}^{%
\mathbb{C}
},$ $\lambda _{n}$ is continuous at each point not in $\mathscr{K}^{%
\mathbb{C}
}.$
\end{proposition}

\begin{proposition}
For every $n\in
\mathbb{N}
_{0},$ the restriction of $\lambda _{n}$ to each of $\mathscr{F}_{-}^{%
\mathbb{C}
},\mathscr{G}_{-}^{%
\mathbb{C}
},\mathscr{H}_{-}^{%
\mathbb{C}
}$ and $\mathscr{I}_{-}^{%
\mathbb{C}
}$ is continuous.
\end{proposition}

\begin{theorem}
The conclusions of Theorem \ref{3.18 copy(1)} still hold when the super
indices $%
\mathbb{R}
$ in them are replaced by $%
\mathbb{C}
.$
\end{theorem}

\begin{remark}
On the basis of the lemmas and theorems we give in the previous sections of
this paper, we can also use the similar methods in \cite{CC5} for the
Sturm-Liouville problems with regular potentials to prove the propositions
and theorems on the discontinuity set of $\lambda _{n}$ as a function on $%
\mathscr{B}_{S}^{%
\mathbb{R}
}$ or $\mathscr{B}_{S}^{%
\mathbb{C}
}$ in this section. However, in this paper, we supply a simpler method
which relies heavily on Lemma \ref{norm resovent}.
\end{remark}

\section{\protect\bigskip Oscillation theorems}

\begin{theorem}
\label{L9} $(a)$For any $K\in \mathit{\mathrm{SL}}(2,%
\mathbb{R}
)$, let $\psi _{n}(x)$ be a real eigenfunction of $\lambda _{n}(K),\ $then
the number of zeros of $\psi _{n}(x)$ on the interval $[a,b)$ is $0$ or $1$
if $n=0$, and if $n\geq 1,\ $the number is $n-1$, or $n$ or $n+1.\ $

$(b)$For any $K\in \mathit{\mathrm{SL}}(2,%
\mathbb{R}
),\ 0<\gamma <\pi $ or $-\pi <\gamma <0,\ $let $\psi _{n}(x)$ be an
eigenfunction of $\lambda _{n}(\gamma ,K),\ $then the number of zeros of $%
\mathit{\mathrm{Re}}\psi _{n}(x)$ on the interval $[a,b)$ is $0$ or $1$ if $%
n=0$, and if $n\geq 1,\ $the number is $n-1$, or $n$ or $n+1.\ $This claim
also holds for $\mathit{\mathrm{Im}}\psi _{n}(x).\ $Moreover, $\psi _{n}(x)$
has no zeros on $[a,b].\ $
\end{theorem}

\begin{theorem}
\label{L10} For any $K\in \mathit{\mathrm{SL}}(2,%
\mathbb{R}
),$ if $k_{11}>0\ $and $k_{12}=0,$ let $\psi _{n}(x)$ and $\xi _{n}(x)$ be
real eigenfunctions of $\lambda _{n}(K)$ and $\lambda _{n}(-K),$
respectively. Then we obtain the following conclusions$:$ $\mathbf{(}a%
\mathbf{)}$\textbf{\ }$\psi _{0}(x)$\textit{\ has no zeros in }$[a,b].$ $(b)%
\mathit{\ }\psi _{2m+1}(x)$\textit{\ and\ }$\psi _{2m+2}(x)$\textit{\ have\
exactly\ }$2m+2$\textit{\ zeros\ in\ }$[a,b).$ $(c)\mathit{\ }\xi _{2m}(x)%
\mathit{\ }$\textit{and}$\mathit{\ }\xi _{2m+1}(x)\mathit{\ }$\textit{have\
exactly}$\mathit{\ }2m+1\mathit{\ }$\textit{zeros\ in}$\mathit{\ }[a,b).$
\end{theorem}

\begin{proof}[Proof of Theorem 6.1]
$(a)$ Denote the eigenfunction of $\lambda _{n}^{D}$ by ${\Psi }_{n}(x)$,
according to Lemma \ref{L4}, it follows that ${\Psi }_{n}(x)$ has $n+2$
zeros on the interval $[a,b].\ $It suffices to prove the theorem in the
following two cases.

Case 1: $k_{12}\neq 0,\ $by Theorem \ref{L7}, there exists $0<\beta <\pi ,\ $%
such that
\begin{equation*}
\lambda _{n}(K)\leq \lambda _{n}(S_{0,\beta })<\lambda _{n}(S_{0,\pi
})=\lambda _{n}^{D},n\in
\mathbb{N}
_{0},
\end{equation*}%
thus from Theorem \ref{L8}$,$ $\lambda _{n-2}^{D}<\lambda _{n}(K)$ $<\lambda
_{n}^{D}$, $n\geq 2,\ $so this theorem is now a direct consequence of Lemma %
\ref{L5}.

Case 2: $k_{12}=0,$\ now the condition (\ref{21}) is%
\begin{equation}
\left\{
\begin{array}{l}
y(b)=k_{11}y(a), \\
y^{[1]}(b)=k_{21}y(a)+k_{22}y^{[1]}(a).%
\end{array}%
\right.  \label{2a}
\end{equation}%
If $\lambda _{n}(K)$ $<\lambda _{n}^{D},\ $our claims can be obtained from
Lemma \ref{L5} and Theorem \ref{L8}. If $\lambda _{n}(K)=\lambda _{n}^{D},\ $%
assume $\psi _{n}(x)$ has $n+2$ zeros on the interval $[a,b),\ $now we can
easily reach a contradiction by using Lemma \ref{L5} no matter $\psi
_{n}(a)=0$ or not. Thus the conclusions can be obtained according to Lemma %
\ref{L5} and Theorem \ref{L8},.

$(b)$ It can be easily seen that $\mathit{\mathrm{Re}}\psi _{n}(x)$ and $%
\mathit{\mathrm{Im}}\psi _{n}(x)$ are both nontrivial solutions of equation (%
\ref{b}) with $\lambda =\lambda _{n}(\gamma ,K),\ $thus by Lemma \ref{L5}
and Theorem \ref{L8}, the conclusions on them can be easily obtained. Since
the eigenfunction $\psi _{n}(x)$ of $\lambda _{n}(\gamma ,K)$ can not be
real, it follows that $\mathit{\mathrm{Re}}\psi _{n}(x)$ and $\mathit{%
\mathrm{Im}}\psi _{n}(x)$ are linearly independent solutions of the
equation, thus they do not have the same zeros on $[a,b],\ $so we obtain
that $\psi _{n}(x)$ has no zero on $[a,b].$
\end{proof}

\begin{proof}[Proof of Theorem 6.2]
$(a)$ According to Theorem \ref{L7}, for such a $K\in \mathit{\mathrm{SL}}(2,%
\mathbb{R}
),$
\begin{eqnarray*}
\lambda _{0}(K) &<&\lambda _{0}(-K)\leq \lambda _{0}^{D}\leq \lambda
_{1}(-K)<\lambda _{1}(K)\leq \lambda _{1}^{D} \\
&\leq &\lambda _{2}(K)<\lambda _{2}(-K)\leq \lambda _{2}^{D}\leq \lambda
_{3}(-K)<\lambda _{3}(K)\leq \cdots .
\end{eqnarray*}%
Denote the eigenfunction of $\lambda _{n}^{D}$ by ${\Psi }_{n}(x)$. Since ${%
\Psi }_{0}(x)$ has no zeros on $(a,b)$, and $\lambda _{0}(K)<\lambda
_{0}^{D},$ it follows that $\psi _{0}(x)$ has at most one zero on $[a,b]$.
Without loss of generality, we assume that $\psi _{0}(x_{0})=0,$ $\psi
_{0}^{[1]}(x_{0})>0\ $for some $x_{0}\in (a,b)$. Hence\ by Lemma \ref{L3}
and the continuity of $\psi _{0}(x),$ one deduces that $\psi _{0}(a)<0,\psi
_{0}(b)>0$. This leads to a contradiction\ since $\psi _{0}(x)$ satisfies
the boundary condition (\ref{2a}) with $k_{11}>0$. This completes the proof
of part $\mathbf{(}a\mathbf{)}$\textbf{. }It remains to prove part $\mathbf{(%
}b\mathbf{)}$ of the theorem.

$(b)$ Recall that\ $\lambda _{2m}^{D}<\lambda _{2m+1}(K)\leq \lambda
_{2m+1}^{D}$\thinspace , ${\Psi }_{2m}(x)$ and ${\Psi }_{2m+1}(x)$ has $2m$
and $2m+1$ zeros\ respectively on $(a,b)$. Thus it follows from Lemma \ref%
{L5} that $\psi _{2m+1}(x)$ has at least\ $2m+1$ and at most $2m+2$ zeros on
$(a,b)$.

It suffices to show that $\psi _{2m+1}(x)$ have an\ even number of zeros on $%
[a,b).$ In fact, if $k_{11}\psi _{2m+1}(a)=\psi _{2m+1}(b)\neq 0,$ combining
with the continuity of $\psi _{2m+1}(x),$ one deduces that $\psi _{2m+1}(x)$
must have an\ even number of zeros on $[a,b)$. On the other hand, if $\psi
_{2m+1}(a)=\psi _{2m+1}(b)=0,$ assume $\psi _{2m+1}^{[1]}(a)>0,$ thus we
have $\psi _{2m+1}^{[1]}(b)>0$ since $k_{22}=$ $1/k_{11}>0$. From Lemma \ref%
{L3} and the continuity of $\psi _{2m+1}(x),$ we obtain that $\psi
_{2m+1}(x) $ must have an\ even number of zeros on $[a,b).$

With the inequality
\begin{equation*}
\lambda _{2m+1}^{D}\leq \lambda _{2m+2}(K)<\lambda _{2m+2}^{D},
\end{equation*}%
the result for\ $\psi _{2m+2}(x)$\ can be proved by the same method. This
proves part $\mathbf{(}b\mathbf{)}$\textbf{.}

Part $(c)$\ can be proved in the same way as in the proof of part $\mathbf{(}%
b\mathbf{)}$.

$(c)$ In order to prove this part$\mathbf{,}$ we need the following fact:%
\begin{equation*}
\lambda _{0}(-I)\leq \lambda _{0}^{D},
\end{equation*}%
\begin{equation*}
\lambda _{2m-1}^{D}<\lambda _{2m}(-K)\leq \lambda _{2m}^{D}\text{ and }%
\lambda _{2m}^{D}\leq \lambda _{2m+1}(-K)<\lambda _{2m+1}^{D},\text{ }m\in
\mathbb{N}
.
\end{equation*}%
The only\ difference is the fact that $\xi _{2m}(x)$ and $\xi _{2m+1}(x)$
must have an\ odd number of zeros on $[a,b)$\ which is implied by Lemma \ref%
{L3} and the boundary\ condition
\begin{equation*}
\left\{
\begin{array}{l}
y(b)=-k_{11}y(a), \\
y^{[1]}(b)=-k_{21}y(a)-k_{22}y^{[1]}(a).%
\end{array}%
\right.
\end{equation*}
\end{proof}

\section{Differentiability properties of eigenvalues}

As a similar space we have introduced in Section 2, in this section we
introduce a \textquotedblleft coefficient space\textquotedblright\ with a
metric. Let $\tilde{\Omega}=\{\omega =(1/p,q,r,s);$ $(\ref{tx})$ holds$\}.$
For the topology of $\tilde{\Omega}$ we use a metric $d$ defined as follows:

For $\omega =(1/p,q,r,s)\in \tilde{\Omega},$ $\omega _{0}=(\frac{1}{p_{0}}%
,q_{0},r_{0},s_{0})\in \tilde{\Omega},$ define
\begin{equation*}
d(\omega ,\omega _{0})=\int_{a}^{b}\left( \left\vert \frac{1}{p}-\frac{1}{%
p_{0}}\right\vert +\left\vert q-q_{0}\right\vert +\left\vert
r-r_{0}\right\vert +\left\vert s-s_{0}\right\vert \right) .
\end{equation*}

\begin{theorem}
\label{L9 copy(1)}For any $n\in
\mathbb{N}
_{0},$ the $n$-th eigenvalue of the problem with a fixed boundary condition
depends continuously on the coefficients of the differential equation.
\end{theorem}

\begin{proof}
First, we consider the case where the self-adjoint boundary condition is a
separated one. Let $\omega _{0}=\left( \frac{1}{p_{0}},q_{0},r_{0},s_{0}%
\right) \in \tilde{\Omega}$. Then $\lambda _{0}(\omega _{0})$ is simple.
Consider the continuous eigenvalue branch $\Lambda (\omega )$ through $%
\lambda _{0}(\omega _{0})$ defined on a neighborhood of $\omega _{0}$ in $%
\tilde{\Omega}.$ Let $w=w(\cdot ,\omega _{0})$ denote a normalized
eigenfunction of the eigenvalue $\lambda (\omega _{0}).$ From Lemma \ref%
{eigenfunction}, there exist normalized eigenfunctions $w=w(\cdot ,\omega )$
of $\Lambda (\omega )$ such that
\begin{equation}
w(\cdot ,\omega )\rightarrow w(\cdot ,\omega _{0}),\text{ }w^{[1]}(\cdot
,\omega )\rightarrow w^{[1]}(\cdot ,\omega _{0}),\text{ as }\omega
\rightarrow \omega _{0}\text{ in }\tilde{\Omega},  \label{11eigenfun}
\end{equation}%
both uniformly on the interval $[a,b].$

Note that $w(t,\omega _{0})$ does not have a zero in $(a,b)$ from Lemma \ref%
{L4}$.$ So, we may assume that $w(t,\omega _{0})>0$ on $(a,b).$

$($i$)$ If $w(a,\omega _{0})=0$, it follows from Lemma \ref{L3} that $%
w^{[1]}(a,\omega _{0})>0.$ Hence there exists $\epsilon _{1}>0$ such that $%
w^{[1]}(t,\omega _{0})>0$ on the interval $\left[ a,a+\epsilon _{1}\right] .$
By $(\ref{11eigenfun})$, when $\omega $ is sufficiently close to $\omega
_{0},$ $w^{[1]}(t,\omega )>0$ on $\left[ a,a+\epsilon _{1}\right] .$ It is a
fact that $w(a,\omega )=0$ since the boundary condition is fixed$.$Thus $%
w(t,\omega )>0$ on the interval $\left( a,a+\epsilon _{1}\right) $ when $%
\omega $ is sufficiently close to $\omega _{0}.$

If $w(b,\omega _{0})=0,$ through a similar process, there exists $\epsilon
_{2}>0$ such that $w(t,\omega )>0$ on the interval $\left( b-\epsilon
_{2},b\right) $ when $\omega $ is sufficiently close to $\omega _{0}.$ Since
$w(t,\omega _{0})>0$ on $\left[ a+\epsilon _{1,}b-\epsilon _{2}\right] ,$ it
follows from $(\ref{11eigenfun})\ $that $w(t,\omega )>0$ on $\left[
a+\epsilon _{1,}b-\epsilon _{2}\right] $ when $\omega $ is sufficiently
close to $\omega _{0}.$ Hence $w(t,\omega )>0$ on $(a,b)$ when $\omega $ is
sufficiently close to $\omega _{0}.$

$($ii$)$ If $w(t,\omega _{0})>0$ on $\left[ a,b\right] ,$ by $(\ref%
{11eigenfun})$, $w(t,\omega )>0\ $on $\left[ a,b\right] $ when $\omega $ is
sufficiently close to $\omega _{0}.$

Thus by Lemma \ref{L4}, when $\omega $ is sufficiently close to $\omega
_{0}, $ $\Lambda (\omega )=\lambda _{0}(\omega ).$ According to Lemma \ref%
{princeple}, it follows that $\lambda _{1}(\omega ),$ $\lambda _{2}(\omega )$%
, $\lambda _{2}(\omega ),\cdots $ are continuous at $\omega _{0}.$

Next, assume that the self-adjoint boundary condition is the coupled one (%
\ref{21}) or (\ref{ee}), with\textit{\ }$k_{11}>0$\textit{, }$k_{12}\leq 0.$
Then $\nu _{0}(\omega ,K)$ is continuous at $\omega _{0}$ by the proven
case. On the other hand, by part (a) of Theorem \ref{L7},
\begin{equation*}
\nu _{0}(\omega ,K)\leq \lambda _{0}(\omega ,K)<\lambda _{0}(\omega ,\gamma
,K)<\lambda _{0}(\omega ,-K).
\end{equation*}%
Thus $\lambda _{0}(\omega ,K),$ $\lambda _{0}(\omega ,\gamma ,K),$ $\lambda
_{0}(\omega ,-K)$ are uniformly bounded from below in a small neighborhood
of $\omega _{0}.$ Therefore, Lemma \ref{princeple} implies that for each $%
n\in
\mathbb{N}
_{0,\text{ }}$ $\lambda _{n}(\omega ,K),$ $\lambda _{n}(\omega ,\gamma ,K),$
$\lambda _{n}(\omega ,-K)$ are continuous at $\omega _{0}.$

Finally, we consider the case where the self-adjoint boundary condition is
the coupled one (\ref{21}) or (\ref{ee}), with $k_{11}\leq 0$\textit{, }$%
k_{12}<0.$ Fix an $\omega _{0}\in \tilde{\Omega}$ and consider the
continuous eigenvalue branch $\Lambda $ through $\lambda _{0}(\omega ,K)\ $%
defined on a connected neighborhood $O$ of $\omega _{0}.$ By part (b) of
Theorem \ref{L7}, $\Lambda (\omega _{0})=\lambda _{0}(\omega _{0},K)$ $<$ $%
\nu _{0}(\omega _{0},K)$ and $\Lambda (\omega )\neq \nu _{0}(\omega ,K)$ for
any $\omega \in O.$ Hence, we have $\Lambda (\omega )<$ $\nu _{0}(\omega ,K)$
for any $\omega \in O,$ since both $\Lambda $ and$\ \nu _{0}$ are continuous
functions on $O.$ Therefore, $\Lambda (\omega )=\lambda _{0}(\omega ,K)$ for
any $\omega \in O$ still by part (b) of Theorem \ref{L7}, i.e., $\lambda
_{0}(\omega ,K)$ is continuous at $\omega _{0}.$ On the other hand, by part
(b) of Theorem \ref{L7},
\begin{equation*}
\lambda _{0}(\omega ,K)<\lambda _{0}(\omega ,\gamma ,K)<\lambda _{0}(\omega
,-K)\leq \nu _{0}(\omega ,K).
\end{equation*}%
Thus $\lambda _{0}(\omega ,K),$ $\lambda _{0}(\omega ,\gamma ,K),$ $\lambda
_{0}(\omega ,-K)$ are uniformly bounded from below in a small neighborhood
of $\omega _{0}.$ Therefore, by Lemma \ref{princeple}, for each $n\in
\mathbb{N}
_{0},$ $\lambda _{n}(\omega ,K),$ $\lambda _{n}(\omega ,\gamma ,K),$ $%
\lambda _{n}(\omega ,-K)$ are continuous at $\omega _{0}.$

Note that if neither of the above cases applies to $K,$ then either of the
above cases applies to $-K.$
\end{proof}

In the following we show that the eigenvalues are differentiable functions
of the coefficients $1/p,q,r,s$ in the equation. Recall the definition of
the Frechet derivative:

\begin{definition}
Let $X$ and $Y$ be Banach spaces, with norms $\left\Vert \cdot \right\Vert
_{X}$ and $\left\Vert \cdot \right\Vert _{Y}$ respectively. Let $U\subset X$
be an open set, and let $A:U$ $\rightarrow Y$ be a map. We say that $A$ is
Frechet differentiable at a point $x_{0}\in X$ if there exists a bounded
linear operator $B:X$ $\rightarrow Y$ such that for $h\in X,$%
\begin{equation*}
\left\Vert A\left( x_{0}+h\right) -A\left( x_{0}\right) -B(h)\right\Vert
_{Y}=o(\left\Vert h\right\Vert _{X})\text{ as }h\rightarrow 0,
\end{equation*}%
and denote the bounded linear operator $B\ $by $A^{\prime }\left(
x_{0}\right) .$

\begin{remark}
For investigating the differentiability of the eigenvalue $\lambda _{n}$ as
a function of the coefficients $1/p$ and $r,$ we recall the definition of
the Frechet derivative on the positive cone%
\begin{equation*}
V=\left\{ f\in L(J,%
\mathbb{R}
)\left\vert f\geq 0\text{ a.e. on }J\right. \right\}
\end{equation*}%
of the Banach space $L(J,%
\mathbb{R}
).$ Considering a $($nonlinear$)$ functional $\lambda _{n}$ from $V$ to $%
\mathbb{R}
,$ we say that $\lambda _{n}$ is Frechet differentiable at a point $r$ in $V$
if there exists a bounded linear functional $f:V$ $\rightarrow
\mathbb{R}
$ such that for $h\in V,$%
\begin{equation*}
\left\vert \lambda _{n}\left( r+h\right) -\lambda _{n}\left( r\right)
-f(h)\right\vert =o(\left\Vert h\right\Vert _{L(J,%
\mathbb{R}
)})\text{ as }h\rightarrow 0,
\end{equation*}%
and denote the bounded linear functional $f\ $by $\lambda _{n}^{\prime
}\left( r\right) .$
\end{remark}
\end{definition}

\begin{theorem}
Let $\omega =(A,B,1/p,q,r,s)\in \Omega .$ Fix $A,B.$ Assume that $\lambda
_{n}$ is a simple eigenvalue of $\omega $ for some $n\in
\mathbb{N}
_{0}$ and $w_{n}$ is a normalized eigenfunction of $\lambda _{n},$ then
there is a simple closed curve $\Gamma $ in $%
\mathbb{C}
$ with $\lambda _{n}\left( \omega \right) $ in its interior and a
neighborhood $O$ of $\omega $ in $\Omega $ such that for any $\rho $ in $O,$
the Sturm-Liouville problem $\rho $ has exactly one eigenvalue in the
interior of $\Gamma $ and this eigenvalue is simple.

$(1)$ Fix $q,r,s$ and consider $\lambda _{n}$ as a function of $1/p,$ $p>0\ $%
a.e.$\ $on $J.$ Then $\lambda _{n}$ is Frechet differentiable at $1/p$ in $V$
and its Frechet derivative is the bounded linear transformation given by%
\begin{equation}
\lambda _{n}^{\prime }(1/p)h=-\int\nolimits_{a}^{b}\left\vert
w_{n}^{[1]}(\cdot ,1/p)\right\vert ^{2}h,\text{ }h\in L(J,%
\mathbb{R}
);  \label{i/p}
\end{equation}%
$(2)$Fix $1/p,q,r$ and consider $\lambda _{n}$ as a function of $s.$ Then $%
\lambda _{n}$ is Frechet differentiable at $s$ in $L(J,%
\mathbb{R}
)$ and its Frechet derivative is the bounded linear transformation given by%
\begin{equation}
\lambda _{n}^{\prime }(s)h=2\int\nolimits_{a}^{b}\mathit{\mathrm{Re}}(w_{n}%
\bar{w}_{n}^{[1]})h,\text{ }h\in L(J,%
\mathbb{R}
);  \label{sss}
\end{equation}%
$(3)$Fix $1/p,s,r$ and consider $\lambda _{n}$ as a function of $q.$ Then $%
\lambda _{n}$ is Frechet differentiable at $q$ in $L(J,%
\mathbb{R}
)$ and its Frechet derivative is the bounded linear transformation given by%
\begin{equation}
\lambda _{n}^{\prime }(q)h=\int\nolimits_{a}^{b}\left\vert w_{n}\right\vert
^{2}h,\text{ }h\in L(J,%
\mathbb{R}
);  \label{qqq}
\end{equation}%
$(4)$Fix $1/p,s,q$ and consider $\lambda _{n}$ as a function of $r,$ $r>0\ $%
a.e.$\ $on $J.$ Then $\lambda _{n}$ is Frechet differentiable at $r$ in $V$
and its Frechet derivative is the bounded linear transformation given by%
\begin{equation}
\lambda _{n}^{\prime }(r)h=-\lambda _{n}(r)\int\nolimits_{a}^{b}\left\vert
w_{n}\right\vert ^{2}h,\text{ }h\in L(J,%
\mathbb{R}
).  \label{rrr}
\end{equation}
\end{theorem}

\begin{proof}
The conclusion that the Sturm-Liouville problem $\rho $ has exactly one
eigenvalue in the interior of $\Gamma $ and this eigenvalue is simple is an
obvious result$.$

In the following, we only prove $(\ref{i/p})$ and $(\ref{sss})$. The conclusion $(\ref{qqq})$ and $(\ref{rrr})$ can be proved similarly.

$(1)$Denote $w_{n}=w_{n}(\cdot ,1/p),v_{n}=w_{n}(\cdot ,1/p_{h})\ $where $%
1/p_{h}=1/p+h,$ $h\in V.$ Note that $1/p$ $\in L(J,%
\mathbb{R}
)$ implies that $1/p_{h}\in L(J,%
\mathbb{R}
)$ and $p-p_{h}=pp_{h}h.$ Using $(\ref{fangcheng})$ and integration by parts
we obtain%
\begin{eqnarray*}
&&(\lambda _{n}(1/p_{h})-\lambda _{n}(1/p))\int\nolimits_{a}^{b}w_{n}\bar{v}%
_{n}r \\
&=&\lambda _{n}(1/p_{h})\int\nolimits_{a}^{b}w_{n}\bar{v}_{n}r-\lambda
_{n}(1/p)\int\nolimits_{a}^{b}w_{n}\bar{v}_{n}r \\
&=&\int\nolimits_{a}^{b}w_{n}(-(\bar{v}_{n}^{[1]})^{\prime }+s\bar{v}%
_{n}^{[1]}+q\bar{v}_{n})-\int\nolimits_{a}^{b}\bar{v}_{n}(-(w_{n}^{[1]})^{%
\prime }+sw_{n}^{[1]}+qw_{n}) \\
&=&\left. \left[ -w_{n}\bar{v}_{n}^{[1]}+\bar{v}_{n}w_{n}^{[1]}\right]
\right\vert _{a}^{b}+\int\nolimits_{a}^{b}\left( w_{n}^{\prime }\bar{v}%
_{n}^{[1]}+sw_{n}\bar{v}_{n}^{[1]}-\bar{v}_{n}^{\prime }w_{n}^{[1]}-s\bar{v}%
_{n}w_{n}^{[1]}\right) \\
&=&\left. \left[ -w_{n}\bar{v}_{n}^{[1]}+\bar{v}_{n}w_{n}^{[1]}\right]
\right\vert _{a}^{b}+\int\nolimits_{a}^{b}\left( \frac{1}{p}w_{n}^{[1]}\bar{v%
}_{n}^{[1]}-\frac{1}{p_{h}}w_{n}^{[1]}\bar{v}_{n}^{[1]}\right) \\
&=&\left. \left[ -w_{n}\bar{v}_{n}^{[1]}+\bar{v}_{n}w_{n}^{[1]}\right]
\right\vert _{a}^{b}-\int\nolimits_{a}^{b}w_{n}^{[1]}\bar{v}_{n}^{[1]}h.
\end{eqnarray*}%
For all boundary conditions we have that
\begin{equation*}
\left. \left[ -w_{n}\bar{v}_{n}^{[1]}+\bar{v}_{n}w_{n}^{[1]}\right]
\right\vert _{a}^{b}=0.
\end{equation*}%
Noting that $1/p_{h}\rightarrow 1/p$ as $h\rightarrow 0\ $in $L(J,%
\mathbb{R}
)$ and using Lemma $\ref{eigenfunction}$ we have%
\begin{equation*}
(\lambda _{n}(1/p+h)-\lambda
_{n}(1/p))(1+o(1))=-\int\nolimits_{a}^{b}\left\vert w_{n}^{[1]}\right\vert
^{2}h+o(h),
\end{equation*}%
and consequently,%
\begin{eqnarray*}
\lambda _{n}(1/p+h)-\lambda _{n}(1/p) &=&\left(
-\int\nolimits_{a}^{b}\left\vert w_{n}^{[1]}\right\vert ^{2}h+o(h)\right)
(1+o(1))^{-1} \\
&=&-\int\nolimits_{a}^{b}\left\vert w_{n}^{[1]}\right\vert ^{2}h+o(h),
\end{eqnarray*}%
as $h\rightarrow 0\ $in $L(J,%
\mathbb{R}
).$ This completes the proof of $(\ref{i/p}).$

$(2)$Denote $w_{n}=w_{n}(\cdot ,s),v_{n}=w_{n}(\cdot ,s_{h})\ $where $%
s_{h}=s+h,$ $h\in L(J,%
\mathbb{R}
).$ Note that $s$ $\in L(J,%
\mathbb{R}
)$ implies that $s_{h}\in L(J,%
\mathbb{R}
).$ Using $(\ref{fangcheng})$ and integration by parts we obtain%
\begin{eqnarray*}
&&(\lambda _{n}(s_{h})-\lambda _{n}(s))\int\nolimits_{a}^{b}w_{n}\bar{v}_{n}r
\\
&=&\lambda _{n}(s_{h})\int\nolimits_{a}^{b}w_{n}\bar{v}_{n}r-\lambda
_{n}(s)\int\nolimits_{a}^{b}w_{n}\bar{v}_{n}r \\
&=&\int\nolimits_{a}^{b}w_{n}(-(\bar{v}_{n}^{[1]})^{\prime }+s_{h}\bar{v}%
_{n}^{[1]}+q\bar{v}_{n})-\int\nolimits_{a}^{b}\bar{v}_{n}(-(w_{n}^{[1]})^{%
\prime }+sw_{n}^{[1]}+qw_{n}) \\
&=&\left. \left[ -w_{n}\bar{v}_{n}^{[1]}+\bar{v}_{n}w_{n}^{[1]}\right]
\right\vert _{a}^{b}+\int\nolimits_{a}^{b}\left( w_{n}^{\prime }\bar{v}%
_{n}^{[1]}+s_{h}w_{n}\bar{v}_{n}^{[1]}-\bar{v}_{n}^{\prime }w_{n}^{[1]}-s%
\bar{v}_{n}w_{n}^{[1]}\right) \\
&=&\left. \left[ -w_{n}\bar{v}_{n}^{[1]}+\bar{v}_{n}w_{n}^{[1]}\right]
\right\vert _{a}^{b}+\int\nolimits_{a}^{b}\left( hw_{n}\bar{v}_{n}^{[1]}+h%
\bar{v}_{n}w_{n}^{[1]}\right) .
\end{eqnarray*}%
For all boundary conditions we have that
\begin{equation*}
\left. \left[ -w_{n}\bar{v}_{n}^{[1]}+\bar{v}_{n}w_{n}^{[1]}\right]
\right\vert _{a}^{b}=0.
\end{equation*}%
Noting that $s_{h}\rightarrow s$ as $h\rightarrow 0\ $in $L(J,%
\mathbb{R}
)$ and using Lemma $\ref{eigenfunction}$ we have%
\begin{equation*}
(\lambda _{n}(s+h)-\lambda _{n}(s))(1+o(1))=\int\nolimits_{a}^{b}\left( w_{n}%
\bar{w}_{n}^{[1]}+\bar{w}_{n}w_{n}^{[1]}\right) h+o(h),
\end{equation*}%
and consequently,%
\begin{eqnarray*}
\lambda _{n}(s+h)-\lambda _{n}(s) &=&\left( 2\int\nolimits_{a}^{b}\mathit{%
\mathrm{Re}}(w_{n}\bar{w}_{n}^{[1]})h+o(h)\right) (1+o(1))^{-1} \\
&=&2\int\nolimits_{a}^{b}\mathit{\mathrm{Re}}(w_{n}\bar{w}_{n}^{[1]})h+o(h),
\end{eqnarray*}%
as $h\rightarrow 0\ $in $L(J,%
\mathbb{R}
).$ This completes the proof of $(\ref{sss}).$
\end{proof}
\section{Application to a class of transmission problems}

Consider the Sturm-Liouville operator
\begin{equation}
Ly(x):=\frac{1}{w(x)}(-(p(x)y^{\prime}(x)) ^{\prime}+q(x)y(x)),\ x\in
J=(a,c)\cup(c,b),\label{2y}%
\end{equation}
with the transmission conditions
\begin{equation}
\left\{
\begin{array}{l}
y(c+)=y(c-), \\
(py^{\prime }) (c+)-(py^{\prime }) (c-)=\alpha y(c),%
\end{array}%
\right. \label{2s}
\end{equation}
where $1/p,\ q,\ w\in L(J,\mathbb{R}),$ $p>0$, $w>0$ a.e. on $J,$ and $\alpha\in\mathbb{R}$ is a constant.

According to \cite{fi14}, we consider the self-adjoint boundary conditions as
follows:
\begin{equation}
A\left(
\begin{array}{c}
y(a) \\
(py^{\prime })(a)%
\end{array}%
\right) +B\left(
\begin{array}{c}
y(b) \\
(py^{\prime })(b)%
\end{array}%
\right) =\left(
\begin{array}{c}
0 \\
0%
\end{array}%
\right) ,  \label{2r}
\end{equation}%
where $A$, $B$ also satisfy the condition (\ref{juzhen}). As is well known,
the boundary conditions (\ref{2r}) can be divided into three
classes of boundary conditions as follows:

1.Seperated self-adjoint boundary conditions:%
\begin{equation}
S_{\alpha ,\beta }:\left\{
\begin{array}{c}
\cos \alpha y(a)-\sin \alpha (py^{\prime })(a)=0,\text{ }\alpha \in \lbrack
0,\pi ), \\
\cos \beta y(b)-\sin \beta (py^{\prime })(b)=0,\text{ }\beta \in (0,\pi ].%
\end{array}%
\right.  \label{2e}
\end{equation}

2.Real coupled self-adjoint boundary conditions:%
\begin{equation}
Y(b)=KY(a),\text{ }K\in \mathrm{SL}_{2}(%
\mathbb{R}
).  \label{2d}
\end{equation}

3.Complex coupled self-adjoint boundary conditions:%
\begin{equation}
Y(b)=e^{i\gamma }KY(a),\text{ }-\pi <\gamma <0\text{ or }0<\gamma <\pi ,%
\text{ }K\in \mathrm{SL}_{2}(%
\mathbb{R}
),  \label{2t}
\end{equation}%
where $Y(\cdot )=\left(
\begin{array}{c}
y(\cdot ) \\
py^{\prime }(\cdot )%
\end{array}%
\right) .$

\begin{theorem}
All the conclusions we have obtained in Sections $3-6$ can be applied to the
self-adjoint transmission problems $(\ref{2y})-(\ref{2r})$. \end{theorem}
 \begin{proof}
Denote $C=\frac{\left(\int\nolimits_{a}^{b}q(s)ds+\alpha\right)  }%
{\int\nolimits_{a}^{b}w(s)ds}$, $\tilde{q}=q-Cw$. Let $\tilde{u}(x)=\int\nolimits_{a}^{x}\tilde{q}(t)dt,$
then $\tilde{u}(a)=0$ and
\[
\tilde{u}(b)=\int\nolimits_{a}^{b}\tilde{q}(t)dt=\int\nolimits_{a}%
^{b}q(t)dt-C\int\nolimits_{a}^{b}w(t)dt=-\alpha.
\]
Define the function
\[
u_{0}(x)=\left\{
\begin{array}{cc}
\alpha, & x\in [c,b], \\
0, & x\in [a,c),%
\end{array}%
\right.
\]
and let $\bar{u}(x)=u_{0}(x)+\tilde{u}(x),$ then define the following operators on $(a,b)$,
\begin{eqnarray*}
  &&\bar{L}y=\frac{1}{w}(-(py^{\prime}+\bar{u}y)^{\prime}+\frac{\bar{u}}%
{p}(py^{\prime}+\bar{u}y)-\frac{\bar{u}^{2}}{p}y), \\
  && D(\bar{L}) =\left\{ y\in AC([a,b])\left\vert
\begin{array}
[c]{l}%
y^{[1]}=py^{\prime}+\bar{u}y\in
AC([a,b]),\\
y\ \text{satisfies}\ (\ref{25})
,\ \bar{L}y\in L_{w}^{2}(J)
\end{array}
\right.  \right\}.
\end{eqnarray*}
It is obvious that $\bar{L}$ is the Sturm-Liouville operator we have mainly considered in this paper, thus the eigenvalues and eigenfunctions of it satisfy the conclusion in Sections 3--6.

Moreover, since $\bar{u}(a)=\bar{u}(b)=0,$ then $\bar{L}$ can be written as
\begin{eqnarray*}
&&\tilde{L}y=\frac{1}{w}(-(py^{\prime})^{\prime}+\tilde{q}y),\ x\in(a,c)\cup(c,b),\\
&&D(\tilde{L})=\left\{ y\in AC([a,b])\left\vert
\begin{array}
[c]{l}
y\ \text{satisfies}\ (\ref{2r}),\\ py^{\prime }\in
AC([a,c]), py^{\prime}\in AC([c,b]), \\
(py^{\prime }) (c_{+}) -(py^{\prime }) (c_{-}) =y(c),\ \tilde{L}y\in
L_{w}^{2}(J)
\end{array}%
\right. \right\}.
\end{eqnarray*}
In conclusion, from the relation of the Sturm-Liouville problem $(\ref{2y})-(\ref{2r})$ and the operator $\tilde{L}$, the proof is completed.
\end{proof}

\end{document}